\documentclass[12pt]{ut-thesis}

\usepackage{amsmath,amsfonts,amsthm,amssymb, xypic}
\usepackage{graphicx, stmaryrd, mathrsfs}
\usepackage{paralist}

\usepackage{amsmath,amsfonts,amsthm,amssymb, stmaryrd, mathrsfs,enumerate}
\usepackage{color}

\usepackage{listings}
\lstdefinelanguage{Sage}[]{Python}
{morekeywords={False,sage,True},sensitive=true}
\definecolor{dblackcolor}{rgb}{0.0,0.0,0.0}
\definecolor{dbluecolor}{rgb}{0.01,0.02,0.7}
\definecolor{dgreencolor}{rgb}{0.2,0.4,0.0}
\definecolor{dgraycolor}{rgb}{0.30,0.3,0.30}

\usepackage{tikz}
\usetikzlibrary{arrows, positioning, shapes, decorations.pathreplacing, calc}

\usepackage{hyperref}

\allowdisplaybreaks

\degree{Doctor of Philosophy}
\department{Mathematics}
\gradyear{2014}
\author{Jonathan Michael Fisher}
\title{The Topology and Geometry of Hyperk\"ahler Quotients}

\theoremstyle{plain}

\newtheorem{theorem}{Theorem}[section]
\newtheorem{conjecture}[theorem]{Conjecture}
\newtheorem{lemma}[theorem]{Lemma}
\newtheorem{proposition}[theorem]{Proposition}
\newtheorem{corollary}[theorem]{Corollary}

\theoremstyle{remark}
\newtheorem{example}[theorem]{Example}
\newtheorem{remark}[theorem]{Remark}

\theoremstyle{definition}
\newtheorem{definition}[theorem]{Definition}
\newtheorem{construction}[theorem]{Construction}

\newcommand{\double}[1]{\overline{\overline{#1}}}

\newcommand{\tx}{{\tilde{x}}}
\newcommand{\ty}{{\tilde{y}}}

\newcommand{\BH}{\mathbf{H}}
\newcommand{\BC}{\mathbf{C}}
\newcommand{\BP}{\mathbf{P}}
\newcommand{\BQ}{\mathbf{Q}}
\newcommand{\BR}{\mathbf{R}}

\newcommand{\BZ}{\mathbf{Z}}

\newcommand{\CA}{\mathcal{A}}
\newcommand{\CB}{\mathcal{B}}

\newcommand{\CE}{\mathcal{E}}

\newcommand{\CI}{\mathcal{I}}
\newcommand{\CJ}{\mathcal{J}}

\newcommand{\CN}{\mathcal{N}}
\newcommand{\CO}{\mathcal{O}}
\newcommand{\CP}{\mathcal{P}}
\newcommand{\CQ}{\mathcal{Q}}

\newcommand{\CS}{\mathcal{S}}

\newcommand{\CV}{\mathcal{V}}
\newcommand{\CW}{\mathcal{W}}

\newcommand{\CY}{\mathcal{Y}}

\newcommand{\FM}{\mathfrak{M}}

\newcommand{\FX}{\mathfrak{X}}

\newcommand{\Crit}{ \mathrm{Crit} }

\newcommand{\ftd}{\mathfrak{t}^\ast}

\newcommand{\Span}{\mathrm{span}}

\newcommand{\loj}{{\L}ojasiewicz}
\newcommand{\gli}{global {\L}ojasiewicz inequality}

\newcommand{\norm}[1]{\left| #1 \right|}
\newcommand{\normsq}[1]{\left| #1 \right|^2}

\newcommand{\suchthat}{ \ | \ }

\newcommand{\Ad}{ \mathrm{Ad} }

\newcommand{\Sym}{\mathrm{Sym}}

\newcommand{\ann}{\mathrm{ann}}

\newcommand{\khk}{\kappa_{\mathrm{HK}}}

\newcommand{\Hom}{\mathrm{Hom}}

\newcommand{\End}{\mathrm{End}}
\newcommand{\Rep}{\mathrm{Rep}}

\newcommand{\gl}{\mathfrak{gl}}

\newcommand{\FF}{\mathfrak{F}}

\newcommand{\fsp}{\mathfrak{sp}}
\newcommand{\fu}{\mathfrak{u}}
\newcommand{\fsu}{\mathfrak{su}}

\newcommand{\fk}{\mathfrak{k}}

\newcommand{\fb}{\mathfrak{b}}
\newcommand{\fg}{\mathfrak{g}}
\newcommand{\fgc}{\mathfrak{g}_\BC}
\newcommand{\ft}{\mathfrak{t}}

\newcommand{\fkd}{\mathfrak{k}^\ast}

\newcommand{\fgd}{\mathfrak{g}^\ast}

\newcommand{\grad}{\nabla}

\newcommand{\Tr}{\mathrm{Tr}}
\newcommand{\rk}{\mathrm{rk}}

\newcommand{\vev}[1]{\langle#1\rangle}

\newcommand{\red}{/\!/}
\newcommand{\reda}[1]{ \underset{#1}{/\!/} }
\newcommand{\rred}{/\!/\!/}
\newcommand{\rreda}[1]{ \underset{#1}{/\!/\!/} }

\newcommand{\diag}{\mathrm{diag}}

\newcommand{\Diff}{\mathrm{Diff}}

\newcommand{\Gr}{\mathrm{Gr}}

\newcommand{\Stab}{\mathrm{Stab}}
\newcommand{\stab}{\mathrm{stab}}

\newcommand{\im}{\mathrm{im}}

\newcommand{\FV}{\mathfrak{V}}

\newcommand{\bv}{\mathbf{v}}
\newcommand{\bw}{\mathbf{w}}
\newcommand{\bd}{\mathbf{d}}

\newcommand{\omegar}{\omega_\BR}
\newcommand{\omegac}{\omega_\BC}
\newcommand{\mur}{\mu_r}
\newcommand{\muc}{\mu_c}
\newcommand{\muhk}{\mu_\mathrm{HK}}

\newcommand{\into}{\hookrightarrow}
\newcommand{\onto}{\twoheadrightarrow}
\newcommand{\iso}{\cong}

\newcommand{\ev}[1]{\left\langle{#1}\right\rangle}
\newcommand{\res}[2]{\mathrm{res}\left(\frac{#1}{#2}\right)}

\newcommand{\Res}{\mathrm{Res}}

\newcommand{\Jac}{\mathrm{Jac}}

\newcommand{\npvz}{\norm{\normsq{\phi(v)}-\normsq{\phi_c}}}

\newcommand{\doublearrow}[2] {
  \begin{scope}
    \coordinate (avg) at ( $0.5*(#1)+0.5*(#2)$ );
    \coordinate (c1) at ( $(avg)!0.15cm!90:(#2)$ );
    \coordinate (c2) at ( $(avg)!0.15cm!270:(#2)$ );
    \draw[->,>=latex',draw] (#1) .. controls (c1) ..  (#2);
    \draw[->,>=latex',draw,dashed] (#2) .. controls (c2) .. (#1);
  \end{scope}
}

\newcommand{\soliddoublearrow}[2] {
  \begin{scope}
    \coordinate (avg) at ( $0.5*(#1)+0.5*(#2)$ );
    \coordinate (c1) at ( $(avg)!0.15cm!90:(#2)$ );
    \coordinate (c2) at ( $(avg)!0.15cm!270:(#2)$ );
    \draw[->,>=latex',draw] (#1) .. controls (c1) .. (#2);
    \draw[->,>=latex',draw] (#2) .. controls (c2) .. (#1);
  \end{scope}
}

\newcommand{\dotnoderight}[3] {
    \node[circle, draw, fill=black, inner sep=0pt, minimum width=2pt] (#1) at (#2) {};
    \node [right of=#1] (label#1) {#3};
}

\newcommand{\dotnodebelow}[3] {
    \node[circle, draw, fill=black, inner sep=0pt, minimum width=2pt] (#1) at (#2) {};
    \node [below of=#1] (label#1) {#3};
}

\begin{document}

\begin{preliminary}

\maketitle



\begin{abstract}
In this thesis we study the topology and geometry of hyperk\"ahler quotients,
as well as some related non-compact K\"ahler quotients, from the point of view of
Hamiltonian group actions. The main technical tool we employ is Morse theory with moment maps. 
We prove a {\L}ojasiewicz inequality 
which permits the use of Morse theory in the non-compact setting.
We use this to deduce Kirwan surjectivity for an interesting class of non-compact
quotients, and obtain a new proof of hyperk\"ahler Kirwan surjectivity for hypertoric varieties.
We then turn our attention to quiver varieties, obtaining an explicit
inductive procedure to compute the Betti numbers of the fixed-point sets of the natural $S^1$-action on these varieties.
To study the kernel of the Kirwan map, we adapt the Jeffrey-Kirwan residue
formula to our setting. The residue formula may be used to compute intersection pairings on
certain compact subvarieties, and in good cases these provide a complete description of
the kernel of the hyperk\"ahler Kirwan map. We illustrate this technique with several
detailed examples.
Finally, we investigate the Poisson geometry of a certain family of Nakajima varieties.
We construct an explicit Lagrangian fibration on these varieties by embedding them into
Hitchin systems. This construction provides an interesting class of toy models of Hitchin systems for which
the hyperk\"ahler metric may be computed explicitly.
\end{abstract}





\begin{acknowledgements}
  I am indebted to my supervisor Lisa Jeffrey for suggesting these problems to me, 
and for her support and encouragement. Without her vast knowledge and expertise, much
of this thesis would not have been possible. Special thanks go to my undergraduate supervisor 
Li-Hong Xu, who first exposed me to academic research and encouraged me to pursue graduate studies.

I have greatly benefited from the activities of the department and the neighboring Fields Institute. 
I would like to thank Marco Gualtieri, Joel Kamnitzer, Yael Karshon, Boris Khesin, and Eckhard Meinrenken 
for offering stimulating courses and seminars.

During my time in Toronto, I have had many enlightening discussions with
Alan Lai, Alejandro Cabrera, Peter Crooks, Bruce Fontaine, Kevin Luk, 
James Mracek, Brent Pym, Dan Rowe, and Jordan Watts. I would like to thank
Steven Rayan in particular for his help during the preparation of Chapters \ref{ch-quiver}
and \ref{ch-integrable}. Additional thanks go to Peter Crooks for his careful reading
of Chapters \ref{ch-residue} and \ref{ch-integrable}.

Finally, I would like to thank my partner Charis for her love and support.
\end{acknowledgements}


\tableofcontents




\end{preliminary}

\chapter{Introduction} \label{ch-intro}

\section{Hyperk\"ahler Quotients}
A hyperk\"ahler manifold is a Riemannian manifold $(M, g)$ together with a triple $(I, J, K)$
of parallel skew endomorphisms of $TM$ which satisfy the quaternion relations
\begin{equation}
  I^2 = J^2 = K^2 = IJK = -1.
\end{equation}
From these we obtain a triple $(\omega_I, \omega_J, \omega_K)$ of K\"ahler forms, with
$\omega_I(u,v) := g(Iu,v)$, etc. If a Lie group $G$ acts on $M$ preserving this structure,
we call it Hamiltonian if it is Hamiltonian with respect to each of these symplectic
forms, i.e. if there is a triple of moment maps $\muhk = (\mu_I, \mu_J, \mu_K): M \to \fg^\ast \otimes \BR^3$.
One may then define the hyperk\"ahler quotient $M \rreda{\alpha} G$ to be
\begin{equation}
  M \rreda{\alpha} G := \muhk^{-1}(\alpha) / G.
\end{equation}
It is straightforward to check that the complex 2-form $\omegac := \omega_J + \sqrt{-1} \omega_K$, 
is a holomorphic symplectic form, i.e. a closed non-degenerate (2,0)-form. Hence,
a hyperk\"ahler quotient may alternatively be thought of as a \emph{holomorphic symplectic quotient}.
This leads to the natural question: which techniques from Hamiltonian group actions
generalize to the setting of \emph{holomorphic} Hamiltonian group actions?
There are several fundamental obstacles which prevent an easy answer to this question:
\begin{inparaenum}[(i)]
  \item there is no hyperk\"ahler Darboux theorem, since the Riemann tensor is a local invariant;
  \item most interesting examples are non-compact; and
  \item the complex moment maps that we consider are never proper.
\end{inparaenum}

In this thesis, we focus on the special case $M = T^\ast \BC^n$ with its Euclidean metric.
This may be identified with the flat quaternionic vector space $\BH^n$ via the complex structures
\begin{equation}
  I(x,y) = (ix, iy), \ J(x,y) = (-\bar{y}, \bar{x}), \ K(x,y) = (-i\bar{y}, i\bar{x}).
\end{equation}
For any $G \subseteq Sp(n, \BH)$ the natural linear $G$-action on $M$ is Hamiltonian,
and we may take the hyperk\"ahler quotient $\FM_\alpha := M \rreda{\alpha} G$.
Although $T^\ast \BC^n$ is not interesting topologically, the hyperk\"ahler varieties
$\FM_\alpha$ constructed in this way turn out to be very interesting, both topologically
and geometrically. This class of hyperk\"ahler quotients includes hypertoric varieties and
Nakajima quiver varieties. In all the examples that we will consider, we will in fact
take $G \subset U(n) \subset Sp(n, \BH)$. In this case, there is a natural inclusion
from the symplectic quotient $\FX_\alpha := \BC^n \reda{\alpha} G$ into the 
hyperk\"ahler quotient $\FM_{(\alpha,0)}$. Following Proudfoot \cite{ProudfootThesis}, we call $\FM_{(\alpha,0)}$
the \emph{hyperk\"ahler analogue} of $\FX_\alpha$.

There is a canonical map $H^\ast(BG) \to H^\ast(\FM_\alpha)$, called the \emph{hyperk\"ahler Kirwan map}.
There is an analogous map for symplectic quotients which is well-known to be surjective in
many cases. Three fundamental problems which serve as motivation for the work in this
thesis are the following:
\begin{inparaenum}[(i)]
  \item to determine whether the hyperk\"ahler Kirwan map is surjective;
  \item to determine an effective method to compute the image of the hyperk\"ahler Kirwan map; and
  \item to determine an effective method to compute the Betti numbers of $\FM_\alpha$.
\end{inparaenum}
We will address all three problems.
Unfortunately, we do not obtain a complete solution to the first problem. However,
solutions to the second and third problems provide an indirect way to prove or
disprove Kirwan surjectivity in particular cases. The techniques we develop
in this thesis provide a toolkit for producing many new examples satisfying
hyperk\"ahler Kirwan surjectivity. In addition, several techniques we develop may
be applied to a larger class of non-compact quotients, not only to those arising
from the hyperk\"ahler quotient construction.

\section{Motivation from Quantum Field Theory} 
Many of the constructions and examples considered in this thesis were motivated 
either directly from or by analogy with quantum field theory.
The hyperk\"ahler quotient construction itself
has its origins in quantum field theory, having been first discovered in the context of
gauged supersymmetric sigma model \cite{HKLR}.

In Chapter \ref{ch-morse}, we study Morse theory with moment maps.
Morse theory with $|\mu|^2$ was developed by Kirwan \cite{Kirwan} by direct
analogy with Morse theory with the Yang-Mills functional in two dimensions \cite{YangMillsRiemannSurface}.
In two dimensions, one may also study the Yang-Mills-Higgs functional \cite{WilkinHiggs},
which is analogous to $|\muhk|^2$ on a hyperk\"ahler manifold. In fact, it was
in this context that we first encountered {\L}ojasiewicz estimates, which provided
the key idea of the main theorem of Chapter \ref{ch-morse}.

In four dimensions, one may also consider Yang-Mills instantons, which are
solutions to the Euclidean equations of motion with finite action. These give important
non-perturbative corrections to the quantum path integral and are the subject of intense
study in theoretical physics. The self-duality equations defining instantons take the
form of a hyperk\"ahler moment map, and the moduli space of instantons is in fact
an \emph{infinite-dimensional} hyperk\"ahler quotient.
In \cite{ADHM}, the moduli space of instantons on $\BR^4$
was famously constructed as a finite dimensional hyperk\"ahler quotient of Euclidean space.
This construction was later generalized to any ALE space \cite{KronheimerNakajima}, which
are themselves hyperk\"ahler quotients \cite{KronheimerALE}. In \cite{DouglasMoore},
a physical interpretation of these constructions was given in terms of gauge theories
arising from certain $D$-brane configurations in string theory. 
The ALE and instanton moduli spaces above are all examples of Nakajima quiver varieties \cite{Nakajima94},
and for this reason we focus on moduli spaces of quiver representations in Chapter \ref{ch-quiver}.

In Chapter \ref{ch-residue}, we revisit the residue formula in the context of hyperk\"ahler
quotients. The residue formula \cite{JeffreyKirwan95} was itself developed by analogy with Witten's
quantum field theory arguments \cite{WittenNL}, and to a certain extent made many of 
Witten's calculations rigorous. In Witten's recent work \cite{GukovWitten, WittenAC, WittenPI},
he takes the fruitful viewpoint that symplectic manifolds should be thought of as middle-dimensional
integration cycles in an ambient almost hyperk\"ahler manifold. This served as motivation
to re-examine the residue formula in the context of hyperk\"ahler quotients.

In Chapter \ref{ch-integrable}, we construct algebraic integrable systems from quiver
representations. We prove integrability by embedding these spaces into certain
Hitchin systems \cite{Hitchin87, HitchinStableBundles}. Hitchin systems are themselves
moduli spaces of solutions to the Hitchin equations, which are a dimensional reduction
of the self-duality equations. Using the quantum field theoretic techniques of
BPS state counting and wall-crossing, Gaiotto, Moore, and Neitzke have given a conjectural
construction of the hyperk\"ahler metric on Hitchin systems \cite{GMN}.
Although we will not have much to say about this, the primary motivation of Chapter \ref{ch-integrable}
was to construct an interesting class of toy models to which the techniques of Gaiotto-Moore-Neitzke
may be applied. We plan to investigate this in future work.




\section{Summary of Results}

In Chapter \ref{ch-morse}, we review the basic theory of Morse theory with moment maps.
We develop some basic techniques to organize Morse theory calculations, and especially
to streamline index calculations.
The main new result that we prove is the {\L}ojasiewicz estimate of
Theorem \ref{GlobalEstimate}. This gives good analytic control on the gradient flow
of $|\mu|^2$ on $\BC^n$. In the case that the group is abelian, this theorem
also applies to $|\muc|^2$ and $|\muhk|^2$. This yields a new proof of hyperk\"ahler
Kirwan surjectivity for hypertoric varieties, and Morse theory provides a very
short calculation of their cohomology rings (both equivariant and ordinary).
Much of the material in Sections \ref{sec-noncompact-quotient} and \ref{sec-toric} 
appeared in an earlier article \cite{FisherMorse}.

In Chapter \ref{ch-quiver}, we turn our attention to quiver varieties.
In Theorem \ref{thm-inductive-procedure}, we give a general procedure to compute the Poincar\'e
polynomials of subvarieties of quiver varieties defined by relations in the path
algebra. Nakajima quiver varieties are special cases, but our procedure applies
to a larger class of non-compact varieties. In \S\ref{sec-rank-3-star} 
we give a detailed example of this procedure.

In Chapter \ref{ch-residue}, we re-examine the Jeffrey-Kirwan residue formula in the
context of hyperk\"ahler quotients. We show that the residue formula can be interpreted 
as a procedure for producing cogenerators of cohomology rings of hyperk\"ahler quotients.
Theorem \ref{thm-linear-residue-formula} provides a specialization of the residue
formula applicable to quotients of $\BC^n$ by linear group actions. 
As a corollary,
this gives an algorithmic procedure to compute the intersection pairings and 
cohomology ring of \emph{any} compact symplectic quotient of the form $\BC^n \red G$.
Using the residue formula, we define a natural ring associated to any compact
symplectic quotient, which we call the \emph{VGIT ring}. 
Conjecture \ref{conjecture-vgit-kirwan} states that (under appropriate conditions),
the VGIT ring is equal to the image of the hyperk\"ahler Kirwan map. We give some
explicit examples in \S\ref{sec-residue-examples}.
Theorem \ref{thm-residue-hyperkahler-algorithm} provides an alternative procedure
to compute a complete set of cogenerators for any Nakajima quiver variety, which is
valid even if Conjecture \ref{conjecture-vgit-kirwan} fails.

Finally, in Chapter \ref{ch-integrable} we turn our attention away from topology
and focus on the symplectic geometry of hyperk\"ahler quotients.
In Theorem \ref{thm-integrable-parabolic}, we identify a certain class of hyperk\"ahler quiver
varieties with moduli spaces of parabolic Higgs bundles on $\BP^1$. Using
this identification, together with standard arguments about Hitchin systems,
we prove in Theorem \ref{thm-integrable-star-quiver} 
that these quiver varieties are algebraic completely integrable systems.
This gives an interesting family of toy models of Hitchin systems for which the hyperk\"ahler
metric may be computed explicitly.


\chapter{Morse Theory with Moment Maps} \label{ch-morse}

\section{Hamiltonian Group Actions}

\subsection{Symplectic Reduction}
We begin by reviewing basic notions in equivariant symplectic geometry. A classic reference
is \cite{GS84}.
A \emph{symplectic manifold} is a pair $(X, \omega)$ where $X$ is a smooth manifold and
$\omega$ is a symplectic form on $X$. Let $G$ be a compact Lie group with Lie algebra $\fg$, and
suppose that $G$ acts on a symplectic manifold $X$ by symplectomorphisms. For each
$\xi \in \fg$, we have the \emph{fundamental vector field} $v_\xi$ defined by
\begin{equation}
v_\xi(x) :=  \left. \frac{d}{dt} \right|_{t=0} e^{t\xi} \cdot x.
\end{equation}
The action is said to be \emph{Hamiltonian} if there exists a $G$-equivariant map
$\mu: X \to \fgd$ satisfying
\begin{equation}
d\ev{\mu,\xi} = i_{v_\xi} \omega,
\end{equation}
for all $\xi \in \fg$. The map $\mu$ is called the \emph{moment map}, and is unique
(if it exists) up to the addition of a central constant.
If $\alpha$ is a central value of $\mu$, then the \emph{symplectic reduction}
$X \reda{\alpha} G$ is defined to be
\begin{equation}
  X \reda{\alpha} G := \mu^{-1}(\alpha) / G.
\end{equation}
When the group $G$ is understood, we will often denote $X \reda{\alpha} G$ by $\FX_\alpha$.

\begin{theorem}
If $\alpha$ is a regular central value of $\mu$, then $\FX_\alpha$ is a symplectic orbifold,
and a symplectic manifold if $G$ acts freely on $\mu^{-1}(\alpha)$.
Now suppose that $(g, I, \omega)$ is a K\"ahler triple on $X$, and that the $G$-action
preserves this K\"ahler structure. Then this K\"ahler structure descends to the symplectic
reduction $X \red G$, making it into a complex orbifold.
\end{theorem}



\begin{remark} \label{rmk-reduction-in-stages}
Suppose that $N \subset G$ is a normal subgroup, and let $H = G/N$. Then for appropriately
chosen moment map levels, there is a natural identification
\begin{equation}
  X \red G \iso (X \red N) \red H.
\end{equation}
This is called \emph{reduction in stages}.
\end{remark}

\subsection{Hyperk\"ahler Reduction}

A \emph{hyperk\"ahler manifold} is a tuple $(M, g, I, J, K)$ such that $g$ is a Riemannian
metric on $M$ and $I,J,K$ are skew-adjoint parallel endomorphisms of $TM$ satisfying the
quaternion relations. This endows $M$ with a triple $(\omega_I, \omega_J, \omega_K)$ of
symplectic forms, defined by
\begin{equation}
  \omega_I(u,v) = g(Iu,v), \   \omega_J(u,v) = g(Ju,v), \   \omega_K(u,v) = g(Ku,v). \ 
\end{equation}
The action of a compact Lie group $G$ on $M$ is called
\emph{hyperhamiltonian} if it is Hamiltonian  with respect to
each of these symplectic forms. If we denote the corresponding moment maps by
$\mu_I, \mu_J, \mu_K$ then we may package these together into a
\emph{hyperk\"ahler moment map} $\muhk: M \to \fg \otimes \fsp_1$ given by
\begin{equation}
  \muhk(x) =  \mu_I(x) \otimes i + \mu_J(x) \otimes j + \mu_K(x) \otimes k
\end{equation}
where $\{i,j,k\}$ is the standard basis of $\fsp_1 \iso \fsu_2 \iso \BR^3$.

Given a hyperk\"ahler manifold, it is convenient to think of it as having 
a pair $(\omegar, \omegac)$ of symplectic forms, where $\omegar := \omega_I$
is the \emph{real symplectic form} and $\omegac := \omega_J + i \omega_K$
is the \emph{complex symplectic form}. We denote the corresponding moment
maps by $\mur$ and $\muc$, respectively. The \emph{hyperk\"ahler quotient}
of $M$ by $G$, denoted $M \rreda{\alpha} G$, is
\begin{equation}
  M \rreda{\alpha} G := \muhk^{-1}(\alpha) / G.
\end{equation}
As in the symplectic case, when no confusion should arise, we will often denote
$M \rreda{\alpha} G$ by $\FM_\alpha$.

\begin{theorem}[\cite{HKLR}] If $\alpha$ is a regular central value, then $M \rreda{\alpha} G$
is a hyperk\"ahler orbifold, and a hyperk\"ahler manifold provided that the action of $G$ on
$\muhk^{-1}(\alpha)$ is free. The metric on $M \rreda{\alpha} G$ is complete provided
that the metric on $M$ is.
\end{theorem}

Denote the real and complex parts of $\alpha$ by $\alpha_\BR$ and $\alpha_\BC$, respectively.
If $\alpha_\BC$ is regular, then $\muc^{-1}(\alpha_\BC)$ is a complex submanifold of $M$,
and hence is K\"ahler. Then we have
\begin{equation}
  M \rreda{\alpha} G = \mur^{-1}(\alpha_\BR) \cap \muc^{-1}(\alpha_\BC) / G
  = \muc^{-1}(\alpha_\BC) \reda{\alpha_\BR} G.
\end{equation}
Since the K\"ahler quotient $\muc^{-1}(\alpha_\BC) \reda{\alpha_\BR} G$ may be identified
with the geometric invariant theory quotient $\muc^{-1}(\alpha_\BC) \red G_\BC$, we may think
of the hyperk\"ahler quotient as a \emph{holomorphic symplectic reduction} with respect to
the action of $G_\BC$.

\subsection{Cuts and Modification} \label{sec-cuts-modification}

We review two basic constructions which we will be useful for our purposes.
Let $X$ be a symplectic manifold with Hamiltonian circle action.

\begin{definition} \label{dfn-cut-modification}
For any $\sigma \in \BR$ we define the \emph{symplectic cut} \cite{Lerman} 
$X_\sigma$ to be
\begin{equation}
  X_\sigma := \left( X \times \BC \right) \reda{\sigma} S^1,
\end{equation}
where the $S^1$ action is given by $s \cdot (x, z) = (s\cdot x, sz)$. Similarly, if
$M$ is hyperk\"ahler with a hyperhamiltonian circle action, we define its
\emph{hyperk\"ahler modification} \cite{DancerSwannModifying} $M_\sigma$ to be
\begin{equation}
  M_\sigma := \left( M \times T^\ast\BC \right) \rreda{\sigma} S^1.
\end{equation}
\end{definition}

Note that if $M$ is a hyperk\"ahler analogue of $X$, then there is a natural inclusion
$X_\sigma \into M_\sigma$.

We now consider the special case of $X = \BC^n$, with the standard circle action given
by $s \cdot z = sz$. For $\sigma > 0$, $\BC_\sigma^n \iso \BP^n$ as a variety,
but its K\"ahler metric and symplectic form depend on $\sigma$. If $G \subset U(n)$,
then $G$ acts on $\BC^n$ as well as on $\BC^n_\sigma$.
We will need the following lemma.
\begin{lemma} \label{lemma-cut}
Suppose that $G$ acts linearly on $X = \BC^N$ with proper moment map. Then for a sufficiently
large cut parameter $\sigma$, we have a natural identification
  \begin{equation}
    X \reda{\alpha} G \iso X_\sigma \reda{\alpha} G
  \end{equation}
as K\"ahler manifolds.
\end{lemma}
\begin{proof} This is easy to see by reduction in stages. By the hypothesis that the
$G$-moment map is proper, we can find an embedding $\chi: S^1 \into G$ so that the
$S^1$-moment map is proper. We define an action of $G \times S^1$ on $X \times \BC$ by
\begin{equation} (k, s) \cdot (x, z) = (\chi(s) k x, sz) \end{equation}
Note that $(X \times \BC) \reda{\alpha} G \iso (X \reda{\alpha} G) \times \BC$ and the residual
$S^1$ action acts only on the second factor, so we have
\begin{equation} (X \times \BC) \reda{(\alpha,\sigma)} (G \times S^1)
    \iso (X \reda{\alpha} G ) \times (\BC \reda{\sigma-\alpha'} S^1) 
    \iso (X \reda{\alpha} G) \times \{\mathrm{pt}\}
    \iso X \reda{\alpha} G \end{equation}
provided that $\sigma - \alpha' > 0$, where $\alpha'$ is the projection of
$\alpha$ to the image of $d\chi$. On the other hand, we can perform the reductions
in the opposite order. Again, assuming $\sigma > \alpha'$, we have
\begin{equation} (X \times \BC) \reda{(\alpha,\sigma)} (G \times S^1)
    \iso ((X \times \BC) \reda{\sigma} S^1 ) \reda{\alpha} G
    \iso X_\sigma \reda{\alpha} G \end{equation}
\end{proof}

\subsection{Equivariant Cohomology and the Kirwan Map}
We will use equivariant cohomology throughout this thesis, so we review the basic constructions.
A basic reference for equivariant cohomology and localization is \cite{GS99}.

First we recall the Borel model of equivariant cohomology. Let $M$ be a $G$-space, and let
$EG$ be the classifying space of $G$. The equivariant cohomology\footnote{We will always take cohomology with real coefficients.}
$H_G^\ast(M)$ is defined to be
\begin{equation}
  H_G^\ast(M) := H^\ast(EG \times_G M)
\end{equation}
where $EG$ is the classifying space, and $EG \times_G M = (EG \times M) / G$ is the quotient
by the diagonal $G$-action. The functor $H_G^\ast(\cdot)$ is an extraordinary cohomology theory
on the category of $G$-spaces.

When $G$ is a Lie group acting smoothly on a manifold $M$, there is an alternative construction
of equivariant cohomology called the Cartan model. One simply defines $H_G^\ast(M)$ to be
the cohomology of the complex
\begin{equation}
  \Omega_G^\ast(M) := \left(\Omega^\ast(M) \otimes \Sym\ \fg^\ast \right)^G,
\end{equation}
where the equivariant differential $D$ is defined to be
\begin{equation}
  D = d \otimes 1 - \sum_a i_{a} \otimes \phi_a
\end{equation}
where $\{\phi_a\}$ is a basis of $\fg^\ast$ and $i_a$ denotes the interior product with the fundamental
vector field dual to the basis element $\phi_a$. In all of the cases that we will consider, the
Borel and Cartan models of equivariant cohomology are naturally isomorphic.

\begin{definition} Let $M$ be a symplectic manifold with Hamiltonian $G$-action.
The \emph{Kirwan map} is the natural map
\begin{equation}
 \kappa: H_G^\ast(M) \to H^\ast(M \red G)
\end{equation}
induced by the inclusion $\mu^{-1}(0) \into M$. Similarly, if $M$ is hyperk\"ahler
with a hyperhamiltonian $G$-action, the \emph{hyperk\"ahler Kirwan map} is the natural map
\begin{equation}
  \khk: H_G^\ast(M) \to H^\ast(M \rred G).
\end{equation}
\end{definition}

It is well-known that the Kirwan map is surjective for a large class of symplectic quotients \cite{Kirwan}.
However, the hyperk\"ahler Kirwan map is not so well-understood.

\begin{conjecture} The hyperk\"ahler Kirwan map is surjective.
\end{conjecture}

\begin{remark} This conjecture is definitely false if one considers certain infinite-dimensional
quotients. The moduli space of rank 2, fixed-determinant Higgs bundles provides an explicit
counterexample \cite{DWWW}.
\end{remark}

\begin{remark} While preparing this thesis, we became aware of the work of
McGerty and Nevins \cite{McGertyNevins}. By studying the Kirwan-Ness decomposition
on the quotient stack $\CY = X/G$, they obtain a modified version of Kirwan surjectivity for
the algebraic symplectic quotient $T^\ast X \rred G$. 
\end{remark}

\section{Morse Theory}

\subsection{The Moment Map as a Morse-Bott Function}

We begin by recalling (without proof) basic facts about Morse theory with moment maps. 
Standard references include Atiyah-Bott \cite{YangMillsRiemannSurface, AtiyahBottMomentMap}, 
Atiyah \cite{AtiyahConvexity}, and Kirwan \cite{Kirwan}. Let $G$ act on $(X, \omega)$ 
with moment map $\mu$, and let $T \subseteq G$ be a maximal torus. 
For $\beta \in \ft$, we define $\mu^\beta$ to be the pairing $\ev{\mu, \beta}$.
Denote by $T_\beta$ the closure in $T$ of 
$\{ e^{t\beta} \suchthat t \in \BR \}$. Throughout, we will assume that we have fixed some
compatible metric and almost complex structure.

\begin{proposition} For any $\beta \in \ft$, the function $\mu^\beta$ is Morse-Bott.
The critical set of $\mu^\beta$ is equal to the fixed-point set $X^{T_\beta}$. 
\end{proposition}

Let $C$ be a connected component of $X^{T_\beta}$. Since any $x \in C$
is fixed by $T_\beta$, we obtain an isotropy representation on $T_x X$. If we fix a compatible
metric and almost complex structure, then we may decompose $T_x X$ into a direct sum of
irreducibles
\begin{equation}
  T_x X \iso \bigoplus_{i \in I} \BC(\nu_i),
\end{equation}
where $I$ is some indexing set and $\nu_i$ are the weights of the isotropy representation,
counted with multiplicity. We call a weight $\nu$ \emph{positive} if $\ev{\beta, \nu} >0$
and \emph{negative} if $\ev{\beta, \nu} < 0$. Let $N^+ C$ and $N^- C$ be the 
subbundles of $TX|_{C}$ consisting of the positive
and negative weight spaces of the isotropy representation, respectively.

\begin{proposition} The bundles $N^+ C$ and $N^- C$ are respectively the
positive and negative eigenbundles of the Hessian of $\mu^\beta$ on $C$.
In particular, the Morse index of $\mu^\beta$ on $C$ is equal to twice the number
of negative weights, counted with multiplicity.
\end{proposition}

For the remainder of this section, we will assume that $\mu^\beta$ is proper and bounded below,
and that $X^{T_\beta}$ is compact. Let $C \subseteq X^{T_\beta}$ be a connected component of
the fixed-point set. Denote by $X^\pm_C$ the sublevel sets
\begin{equation}
  X^\pm_C = \{ x \in X \suchthat \mu^\beta(x) \leq \mu^\beta(C) \pm \epsilon \}
\end{equation}
where $\epsilon$ is some positive number chosen small enough so that the only
critical value occurring in $[\mu^\beta(C)-\epsilon, \mu^\beta(C)+\epsilon]$ is $\mu^\beta(C)$ itself.

\begin{proposition} \label{prop-AB-lemma}
The equivariant Thom-Gysin sequence for the Morse stratification with respect to
$\mu^\beta$ splits into short exact sequences
\begin{equation}
  0 \to H_T^{\ast-\lambda}(C) \to H_T^\ast(X^+_C) \to H_T^\ast(X^-_C) \to 0.
\end{equation}
The composition of $H_T^{\ast-\lambda}(C) \to H_T^\ast(X^+_C)$ with the restriction
$H_T^\ast(X^+_C) \to H_T^\ast(C)$ is given by multiplication by the equivariant
Euler class of the negative normal bundle to $C$. In particular, the function $\mu^\beta$
is equivariantly perfect.
\end{proposition}

\begin{theorem} \label{thm-formal-mu-perf}
Suppose that $X^T$ is compact and there exists some component $\beta$ such that $\mu^\beta$ is 
proper and bounded below. Then $X$ is equivariantly formal, i.e.
 $H_T^\ast(X) \iso H^\ast(BT) \otimes H^\ast(X)$ as a $H^\ast(BT)$-module. As a consequence,
the function $\mu^\beta$ is perfect, and the Poincar\'e polynomial of $X$ is given by
\begin{equation}
  P_t(X) = \sum_{C \subseteq X^{T_\beta}} t^{\lambda} P_t(C),
\end{equation}
where the sum is over connected components $C$ of the fixed-point set.
\end{theorem}

\begin{theorem} Under the same hypotheses as above, the restriction 
$H_T^\ast(X) \to H_T^\ast(X^T)$ is an injection.
\end{theorem}

\subsection{Computing Isotropy Representations}
We saw in the previous subsection that Morse theory with a moment map amounts to
computing the fixed-point set as well as the isotropy representations. Not only does this allow
us to compute the Poincar\'e polynomial of our manifold, but in addition we may use
Proposition \ref{prop-AB-lemma} to inductively construct the cohomology ring.

The most important special case that we consider is the following. Suppose that $X$ is
a symplectic manifold that we understand well---in most cases, a vector space, flag variety,
or product thereof. Suppose that a compact group $H$ acts on $X$, and $G \subset H$
is a normal subgroup. Then $K := H/G$ has a residual action  on the symplectic reduction 
$\FX = X \red G$. The goal of this section is to give a description of Morse theory
on $\FX$ with respect to the $K$ action in terms of the $H$-action on $X$. We assume throughout
this section that $G$ acts freely on $\mu_G^{-1}(0)$, so that $\FX$ is smooth.

\begin{proposition} \label{prop-fixed-lift}
$x \in \FX^K$ if and only if there is a lift $\tx \in X$
and a unique map $\phi: H \to G$ (depending on the lift $\tx$) such that for all $h \in H$, 
$h \cdot \tx = \phi(h) \cdot \tx$.
Furthermore, the map $\phi$ satisfies the properties
$\phi|_G = 1_G$ and $\phi(gh) = g \phi(h) g^{-1} \phi(g)$ for all $g,h \in H$. 
\end{proposition}
\begin{proof} Suppose that $x \in \FX^K$. Then for any representative $\tx$,
$h \cdot \tx$ must represent the same point in $\FX$, and hence differs from $\tx$ by an
element of $G$. Since the action of $G$ on $\mu_G^{-1}(0)$ is free,
they differ by a \emph{unique} element of $G$, which we will denote by $\phi(h)$.
This defines the map $\phi: H \to G$.
Now consider two elements $g,h \in H$. We compute
\begin{equation}
  \phi(gh) \cdot \tx = gh \cdot \tx = g \phi(h) \tx 
    = g \phi(h) g^{-1} g \tx = g \phi(h) g^{-1} \phi(g) \tx,
\end{equation}
and hence $\phi(gh) =g \phi(h) g^{-1} \phi(g)$. 
\end{proof}

\begin{proposition} Let $\phi$ be as in the preceding proposition, for some
lift $\tx \in X$ of a fixed point $x \in \FX^K$. Then the set $\ker \phi \subseteq H$ 
is a subgroup of $H$ and is naturally
isomorphic to $K$. Consequently, we have $H \iso K \ltimes G$, where the semidirect
product structure is given by $(k_1, g_1)(k_2, g_2) = (k_1 k_2, k_1 g_2 k_1^{-1} g_1)$
\end{proposition}
\begin{proof} A short calculation using the property $\phi(gh) = g\phi(h) g^{-1} \phi(g)$
verifies that $\ker \phi$ is indeed a subgroup of $H$.
We define a map $\psi: K \to \ker \phi$ by $\psi(k) = \phi(\tilde{k})^{-1} \tilde{k}$,
for any representative $\tilde{k} \in H$ of $k$. This is well-defined, since any other
representative of $k$ is of the form $\tilde{k} g$ with $g \in G$, and a short calculation
shows that $\phi(\tilde{k} g)^{-1} \tilde{k} g = \phi(\tilde{k})^{-1} \tilde{k}$.
It is easy to check that this is a group homomorphism, and is inverse to the quotient
map $\ker \phi \to K$. The isomorphism $H \stackrel{\iso}{\to} K \ltimes G$ is given by the map
$h \mapsto ([h], \phi(h))$.


\end{proof}

Now suppose that $x \in \FX^K$ is some particular fixed point with lift $\tx \in X$,
and let $\phi: H \to G$ be the induced map as in the preceding propositions.
Using $\phi$, we define a new $K$-action on $X$, denoted $\ast$, by
\begin{equation}
  k \ast y := \phi(k)^{-1} k \cdot y.
\end{equation}
The property $\phi(k_1k_2) = k_1 \phi(k_2) k_1^{-1} \phi(k_1)$ ensures that $k_1 \ast ( k_2 \ast y ) = (k_1k_2) \ast y$,
so this is indeed an action. This new action has the property that $\tx$ will be fixed
by all of $K$. Hence we have an isotropy representation of $K$ on $T_{\tx} X$. Moreover, this new $K$-action
on $X$ induces the original $K$-action on $\FX$, since for any $y \in X$, $k \cdot y$ and $\phi^{-1}(k) k \cdot y$
lie on the same $G$-orbit.

\begin{proposition} Let $x \in \FX^K$, $\tx \in X$ a representative, and $\phi$
as above. We have the following complex of $K$-representations
\begin{equation}
  0 \to \fg \stackrel{d\rho}{\longrightarrow} T_\tx X \stackrel{d\mu_G}{\longrightarrow} \fg^\ast \to 0, 
\end{equation}
where $\rho$ is the action map and $K$ acts on $\fg \iso \fg^\ast$ by conjugation.
Consequently, $T_x \FX$ is isomorphic as a $K$-representation to the cohomology of this complex.
\end{proposition}
\begin{proof} By the preceding proposition, we identify $K$ with $\ker\phi \subset H$. Then the twisted
$K$-action is nothing more than the restriction of the $H$-action to $K$ under this identification.
Since $G$ is a normal subgroup of $H$, then $\fg$ is invariant under the adjoint action of $H$,
so we may restrict this to a $K$-action on $\fg$. Since the action map and moment map are $H$-equivariant,
this gives us an equivariant complex
\begin{equation}
  0 \to \fg \stackrel{d\rho}{\longrightarrow} T_\tx X \stackrel{d\mu_G}{\longrightarrow} \fg^\ast \to 0
\end{equation}
as claimed. At the level of vector spaces, we have $T_x \FX \iso T_{\tx} \mu^{-1}(0) / T_\tx \CO_\tx$,
where $\CO_\tx$ is the $G$-orbit through $\tx$. Since $T_\tx \mu^{-1}(0) = \ker d\mu_G(\tx)$ and
$T_\tx \CO_\tx = \im \rho$, we find that the cohomology of this complex is isomorphic to $T_x\FX$.
\end{proof}

Since knowledge of the isotropy representation allows us to immediately compute the
Morse index, the above proposition gives an effective method to compute Morse indices
for a large class of Hamiltonian group actions. However, it will be convenient to have
a slightly more general statement for invariant subvarieties. Suppose that $E$ is a linear $H$-representation and 
that $f: X \to E$ is an $H$-equivariant map. Then $f^{-1}(0)$ is $H$-invariant and induces a 
subset\footnote{In all cases that we consider, $\FX$ will be a quasiprojective variety
and $V(f)$ will in fact be an algebraic subvariety.} $V(f) \subset \FX$. 

\begin{proposition} Let $f: X \to E$ an $H$-equivariant map to a linear representation of $H$,
$x \in V(f)^K$, $\tx \in X$ a representative, and $\phi$ as above. 
We have the following complex of $K$-representations
\begin{equation}
  0 \to \fg \to T_\tx X \stackrel{d\mu_G \oplus df}{\longrightarrow} \fg \oplus E \to 0,
\end{equation}
and $T_x V(f)$ is isomorphic as a $K$-representation to the cohomology of this complex.
\end{proposition}
\begin{proof} The argument is identical to that of the previous proposition,
with $d\mu_G \oplus df$ in place of $d\mu_G$.
\end{proof}

Our basic toolkit is nearly complete. However, since we only care about the weights
and multiplicities of isotropy representations of maximal tori, it is most convenient to work at the level
of the representation ring\footnote{Recall that the representation ring of a group is 
the ring consisting of formal $\BZ$-linear combinations of isomorphism classes of representations,
with addition and multiplication given by direct sum and tensor product, respectively.
For a torus $T$ of rank $r$, we have $\Rep(T) \iso \BZ[t_1^{\pm 1}, \dots, t_r^{\pm 1}]$.}
of $K$. We summarize the preceding arguments in the following theorem.
\begin{theorem}  \label{thm-rep-ring}
In the representation ring of $K$, we have the following equality
\begin{equation}
  [T_x \FX] = [ T_\tx X ] - [ \fg_\BC ].
\end{equation}
If $f: X \to E$ is an $H$-equivariant map to a linear representation of $H$, then we also
have
\begin{equation}
  [T_x V(f)] = [T_\tx X ] - [\fg_\BC] - [E]
\end{equation}
\end{theorem}



\subsection{Equivariant Morse Theory} \label{morse-sec-eq-morse}

We now recall some of the basic results of equivariant Morse theory as developed by
Atiyah-Bott \cite{YangMillsRiemannSurface} and Kirwan \cite{Kirwan}.
Let $M$ be a symplectic manifold with a Hamiltonian action of
a compact group $G$, and assume that $0$ is a regular value of the moment 
map.\footnote{This assumes that we wish to study the symplectic reduction at moment map
level $0$. We can always assume that this is the case, by making the replacement
$\mu \mapsto \mu-\alpha$.} Let $T \subset G$ be a maximal torus, and let
$\mu_T$ denote the restriction of the moment map. For $\beta \in \ftd$, we define
\begin{align}
  T_\beta &:= \overline{ \{\exp(t\beta) \suchthat t \in \BR \} }, \\
  Z_\beta &:= \mu^{-1}(\beta), \\
  C_\beta &:= G \cdot \left( Z_\beta \cap M^{T_\beta} \right), \\
  G_\beta &:= \Stab_G(\beta).
\end{align}
Let $\ft_+$ be a positive Weyl chamber and let $\CB$ be the collection of $\beta \in \ft_+$
such that $C_\beta$ is non-empty. Kirwan calls these \emph{minimal combinations of weights}.

\begin{proposition} \label{morse-prop-eq-critical}
The critical set of $|\mu|^2$ is equal to the union $\bigcup_{\beta \in \CB} C_\beta$.
\end{proposition}

\begin{theorem} The function $|\mu|^2$ is Morse-Bott-Kirwan.
\end{theorem}

\begin{proposition}
\label{prop-hessian}
Let $x \in Z_\beta \cap M^{T_\beta}$. Then the negative normal space to $Z_\beta \cap M^{T_\beta}$
at $x$ is isomorphic to $N_x^-(M^{T_\beta}) / T_x ( G \cdot x )$.
In particular, the Morse index at $x$ is equal to
\begin{equation}
  \lambda_\beta = \lambda(\mu^\beta; C_\beta) - \dim G + \dim G_\beta.
\end{equation}
\end{proposition}

We would like to do Morse theory with $|\mu|^2$, and to avoid analytic subtleties
one usually assumes that  $M$ is compact or that $\mu$ is proper. However, we will
need to work in a more general setting. As remarked in \cite[\S 9]{Kirwan}, 
it is enough to assume that the path of steepest descent through any point of $M$
is contained in a compact subset. This notion is extremely useful, so we adopt
the following definition.
\begin{definition} \label{DefnFlowClosed}
A function $f: M \to \BR$ is said to be \emph{flow-closed} if every
positive time trajectory of $-\grad f$ is contained in a compact set.
\end{definition}

For each component $C_\beta$ of the critical set, denote by $M^\beta_\pm$ sub-level sets
\begin{equation}
  M^\beta_\pm := \{ x \in M \suchthat |\mu|^2 \leq |\beta|^2\pm\epsilon \}
\end{equation}
where $\epsilon$ is taken sufficiently small so that the interval
$[|\beta|^2-\epsilon, |\beta|^2+\epsilon]$ contains no critical values other than 
$|\beta|^2$. 

\begin{theorem} \label{KirwanSurjectivity}
If $|\mu|^2$ is flow-closed,
then the stable manifolds $S_\beta$ of its gradient flow form a smooth stratification
of $M$. The Thom-Gysin sequence splits into short exact sequences
\begin{equation}
0 \to H^{\ast-\lambda_\beta}_G(C_\beta) \to H^\ast_G(M^\beta_+) \to H^\ast_G(M^\beta_-) \to 0
\end{equation}
for each $\beta \in \CB$. For any $\beta \in \CB$, the restriction
$H_G^\ast(M) \to H_G^\ast(M^\beta_\pm)$ is surjective. In particular, the Kirwan map is
surjective, and its kernel is generated by those elements of $f \in H_G^\ast(M)$ such that
for some $\beta$, the restriction of $f$ to $C_\beta$ is in the ideal generated by
the equivariant Euler class of the negative normal bundle to $C_\beta$.
\end{theorem}

\begin{corollary} \label{thm-equivariant-poincare}
We have the following equality of Poincar\'e series
\begin{equation}
P_t^G(M) = \sum_\beta t^{\lambda_\beta} P_t^G(C_\beta),
\end{equation}
or equivalently,
\begin{equation}
P_t(M \red G) =  P_t^G(M) - \sum_{\beta\neq 0} t^{\lambda_\beta} P_t^G(C_\beta).
\end{equation}
\end{corollary}



By the above theorems, we can understand the cohomology of the quotient
in terms of the $G$-equivariant cohomology of the connected components of the critical
set. Luckily, we have the following.

\begin{proposition} 
Let $\beta \in \CB$. Then we have
  $H_G^\ast(C_\beta) \iso H_{G_\beta}^\ast\left( Z_\beta \cap M^{T_\beta} \right)$.
\end{proposition}

This proposition yields an inductive procedure as follows. By standard results
in Hamiltonian group actions, the connected components of $M^{T_\beta}$
are symplectic submanifolds. The group $G_\beta=\Stab(\beta)$ acts on $M^{T_\beta}$ with
moment map given by the restriction of $\mu$. Let $H_\beta \subseteq G_\beta$ be
the kernel of the action map $G_\beta \to \Diff(M^{T_\beta})$, and let
$K_\beta = G_\beta / H_\beta$.

\begin{proposition} 
Suppose that $K_\beta$ acts locally freely on $Z_\beta \cap M^{T_\beta}$. Then we have
\begin{equation}
  H_{G_\beta}^\ast(Z_\beta \cap M^{T_\beta}) 
    \iso H^\ast(BH_\beta) \otimes H^\ast(M^{T_\beta} \red K_\beta).
\end{equation}
\end{proposition}
\begin{proof} Simply compute:
\begin{align}
  H_{G_\beta}^\ast( Z_\beta \cap M^{T_\beta} ) 
    &= H^\ast( EG_\beta \times_{G_\beta} (Z_\beta \cap M^{T_\beta}) ) \\
    &= H^\ast( ((EG_\beta/H_\beta) \times (Z_\beta \cap M^{T_\beta})) / K_\beta ) \\
    &= H^\ast( (EG_\beta/H_\beta) \times ((Z_\beta \cap M^{T_\beta}) / K_\beta) ) \\
    &\iso H^\ast(BH_\beta) \otimes H^\ast( (Z_\beta \cap M^{T_\beta}) / K_\beta ) \\
    &= H^\ast(BH_\beta) \otimes H^\ast( M^{T_\beta} \red K_\beta).
\end{align}
\end{proof}
\begin{remark} \label{rmk-symplectic-git}
There is a well-known equivalence between symplectic reduction on K\"ahler varieties
and geometric invariant theory \cite[Chapter 8]{MFK}. The fundamental result is that
the Morse stratification with respect to $|\mu|^2$ agrees with the algebraic Kirwan-Ness
stratification. In particular, the stable manifold of $\mu^{-1}(0)$ consists precisely of
the semistable points, as defined by GIT. We will tend to take the symplectic
point of view, but we will occasionally take advantage of the algebro-geometric description when
it is convenient to do so.
\end{remark}

\subsection{Examples}
Below we present some standard examples to illustrate Morse theory with moment maps,
and in particular to elucidate our method of computing Morse indices. In Section \ref{sec-rank-3-star} 
we present a very involved calculation which assumes familiarity with these techniques.
\begin{example} Complex projective space $\BP^n$ may be identified with the symplectic quotient
$\BC^{n+1} \red S^1$. We use homogeneous coordinates $[z_0: \cdots : z_n]$ on $\BP^n$.
Define an $S^1$ action on $\BC^{n+1}$ by
\begin{equation}
  s \cdot (z_0, \ldots, z_n) = (z_0, s z_1, s^2 z_2, \ldots, s^n z_n).
\end{equation}
This induces an $S^1$ action on $\BP^n$. Suppose that $[z_0: \cdots : z_n]$ is fixed.
Then according to Proposition \ref{prop-fixed-lift}, there is a map $\phi: S^1 \to S^1$
such that $s \ast z = \phi(s) \cdot z$ for all $s \in S^1$. Since $S^1$ is abelian,
$\phi$ is must be a genuine homomorphism, and hence is of the form $s \mapsto s^k$
for some $k$. The fixed-point set is non-empty only if $k \in \{0, \ldots, n\}$.
Given such a $\phi$, the twisted action is given by
\begin{equation}
  s \ast (z_0, \ldots, z_n) = (s^{-k} z_0, s^{1-k} z_1, \ldots, s^{n-k} z_n).
\end{equation}
Hence the fixed-points are isolated and of the form $[0 : \cdots : 0 : 1: 0 : \cdots : 0 ]$.
According to Theorem \ref{thm-rep-ring}, the isotropy representation then is given by
\begin{equation}
  [T_x \BP^n ] = \sum_{i = 0}^n s^{i-k} - 1,
\end{equation}
hence the Morse index is $\lambda_k = 2\# \{ i \suchthat i-k < 0 \} = k$. Hence
\begin{equation}
  P_t(\BP^n) = \sum_{k=0}^n t^{2k} = 1 + t^2 + \cdots + t^{2n}.
\end{equation}
\end{example}
\begin{example} The hyperk\"ahler ALE spaces may be constructed as hyperk\"ahler quotients \cite{KronheimerALE}.
We consider the type $A_n$ ALE space $\FM$, which is the hyperk\"ahler quotient of $T^\ast \BC^{n+1}$
by $(S^1)^n$, where the action of $(S^1)^n$ on $\BC^{n+1}$ is given by
\begin{equation}
  (s_1, \ldots, s_n) \cdot (x_1, \ldots, x_n, x_{n+1})
    = (s_1 x_1, \ldots, s_n x_n, (s_1 \cdots s_n)^{-1} x_{n+1}).
\end{equation}
The moment map equations are
\begin{align}
  |x_j|^2 + |y_{n+1}|^2 &= \alpha_j + |y_j|^2 + |x_{n+1}|^2, \\
  x_j y_j - x_{n+1} y_{n+1} &= 0,
\end{align}
for $j=1, \dots, n$. Let us take $0 < \alpha_1 < \cdots < \alpha_n$.
The $S^1$-action on $\FM$ is given by $s \cdot (x,y) = (sx, sy)$. As in the previous example,
a point $(x,y)$ represented by $(\tx, \ty)$ is fixed if and only if there is a
homomorphism $\phi: S^1 \to (S^1)^n$ such that
\begin{equation}
  \phi(s)^{-1} s \cdot (\tx, \ty) = (\tx, \ty)
\end{equation}
for all $s \in S^1$. Such a homomorphism is necessarily of the form
  $\phi(s) = (s^{w_1}, \ldots, s^{w_n})$,
and this must solve $(s\tx, s\ty) = \phi(s) \cdot (\tx, \ty)$ which gives the system
of equations
\begin{align}
  s x_j &= s^{w_j} x_j, \ j = 0, \ldots, n \\
  s y_j &= s^{-w_j} y_j, \ j = 0, \ldots, n \\
  s x_{n+1} &= s^{-w_1-\cdots-w_n} x_{n+1}, \\
  s y_{n+1} &= s^{w_1+\cdots+w_n} y_{n+1}.  
\end{align}
There are two possible cases: (i) $y_{n+1} = 0$ and (ii) $y_{n+1} \neq 0$.
In the first case, the moment map equations force that $x_j \neq 0$ for $j=1, \dots, n$,
so that $w_1 = \dots w_n = 1$. Then we must have $y_j = 0$ for $j=1, \dots, n$ as
well as $x_{n+1}=0$. Hence this corresponds to a unique fixed point. 
The isotropy representation is given by $s^{n-1} + s^{1-n}$,
and hence the Morse index is equal to 2 if $n>1$ and $0$ if $n=1$.

In the second case, since $y_{n+1} \neq 0$ we have $\sum w_i = 1$, and hence $x_{n+1}=0$.
For each $j$, the complex moment map equations force at least one of $x_j$, $y_j$ to be zero.
Subtracting the real moment map equation for index $k$ from that of index $j$ gives
\begin{equation}
  |x_j|^2 + |y_k|^2 = \alpha_j - \alpha_k + |y_j|^2 + |x_k|^2.
\end{equation}
For $j>k$, at least one of $x_j$ or $y_k$ must be non-zero. Let $\ell$ be the smallest
index such that $y_\ell = 0$. Then $x_{\ell+1} \neq 0$, and therefore $y_{\ell+1} = 0$, hence
$x_{\ell+2} \neq 0$, and so on. Hence we find
  $x_j = 0, y_j \neq 0, j < \ell,
  x_j \neq 0, y_j = 0, j \geq \ell, y_\ell = 0$.
Now in order to have $\sum w_i = 1$, we have to choose $w_\ell = 2\ell-n$.
Hence the fixed-points are parameterized by the integer $\ell \in \{1, \dots, n\}$,
and the isotropy representation is $s^{1+n-2\ell} + s^{1-n+2\ell}$. Hence the
Morse index is $0$ if $\ell \in [(n-1)/2, (n+1)/2]$ and $2$ otherwise.
Combining this with the previous result, we find that the Poincar\'e polynomial
of the $A_n$ ALE space is 
\begin{equation}
  P_t(\FM) = 1+nt^2.
\end{equation}
\end{example}
\begin{example} In this case we use equivariant rather than ordinary Morse theory. The Grassmannian
$\Gr(2,n)$ may be constructed as a symplectic quotient of the space of $2 \times n$ 
matrices by the left multiplication by $U(2)$, i.e. $\Gr(2,n) \iso \BC^{2n} \reda{\alpha} U(2)$,
where the reduction is taken at the matrix $\diag(\alpha, \alpha)$. Up to the action of
the Weyl group, the possible values $\beta$ that can occur as minimal combinations of
weights are $\beta = \diag(0,0), \diag(0, -\alpha), \diag(-\alpha,-\alpha)$.
The critical set corresponding to $\beta = \diag(0,0)$ is $\mu^{-1}(0)$, i.e. the
set whose ($U(2)$-equivariant) cohomology we would like to compute, so there is
nothing to be done.
The critical set corresponding to $\beta = \diag(-\alpha, -\alpha)$ is simply the
origin in $\BC^{2n}$. The negative normal bundle is equal to all of $\BC^n$, 
In particular, the Morse index is equal to $4n$.

The non-trivial critical set is the one corresponding to $\beta = \diag(0, -\alpha)$.
We have $G_\beta \iso U(1) \times U(1)$, with $T_\beta \iso \{1\} \times U(1)$.
$(\BC^{2n})^{T_\beta}$ consists of matrices of the form
\begin{equation}
  \left(
    \begin{array}{cccc}
      \ast & \ast & \cdots & \ast \\
      0    &  0   & \cdots &   0
    \end{array}
  \right)
\end{equation}
and $Z_\beta$ consists of those matrices of this form such that the norm-square of the
top row is equal to $\alpha$. Hence $Z_\beta / U(1) \times \{1\} \iso \BP^{n-1}$.
The index of $\mu^\beta$ is equal to $2n$, and subtracting the dimension of
$U(2) / \Stab(\beta)$ we find that the Morse index is equal to $2n-2$. Hence the 
Poincar\'e polynomial of $\Gr(2,n)$ is equal to
\begin{equation}
  P_t(\Gr(2,n)) = \frac{1-t^{4n}}{(1-t^2)(1-t^4)} - \frac{t^{2n-2}(1+t^{2n})}{(1-t^2)^2}.
\end{equation}
\end{example}
\section{Noncompact Quotients} \label{sec-noncompact-quotient}

\subsection{Complex and Hyperk\"ahler Moment Maps}
One might expect to be able to develop an analogous Morse theory for $|\muhk|^2$. However,
this appears not to be the case, except when $G$ is a 
torus \cite{KirwanHyperkahlerManuscript}. The problem is that in the
nonabelian situation, the norm of the gradient of $|\muhk|^2$ contains 
a term proportional to
the structure constants of the group that is difficult to understand
(see Remark \ref{CrossTermRemark}).
In \cite{JKK}, it was found instead that the 
function $|\muc|^2$ is better behaved, owing to the fact that $\muc$
is $I$-holomorphic. 

\begin{theorem}[{Jeffrey-Kiem-Kirwan 
\cite{JKK}, Kirwan \cite{KirwanHyperkahlerManuscript}}]
\label{SurjectivityCriterion} The function $f = |\muc|^2$ is minimally
degenerate. If $f$ is flow-closed, then the stable manifolds $S_C$
form a $K$-invariant stratification of $M$ and the equivariant Thom-Gysin
sequence splits into short exact sequences. Consequently, the restriction
$H_G^\ast(M) \to H_G^\ast(\muc^{-1}(0))$ is surjective.
If $G$ is a torus, the same 
conclusions hold for $|\muhk|^2$ provided that it is flow-closed.
\end{theorem}  

Note that $\muc^{-1}(0)$ is a $G$-invariant complex submanifold of $M$ (provided $0$ is regular),
so that $M \rred G$ $\iso$  $\muc^{-1}(0) \red G$. Hence surjectivity of map
$H_G^\ast(\muc^{-1}(0)) \to H^\ast(M \rred G)$ can be studied by
the usual methods (but see Corollary \ref{HomotopyCorollary}).
In principle, this reduces the question of Kirwan surjectivity for
hyperk\"ahler quotients to the following conjecture.

\begin{conjecture}[\cite{JKK}]
\label{FlowClosedConjecture}
If $\muc$ is a complex moment map associated to a linear action by a compact 
group $G$ on a vector space, then $|\muc|^2$ is flow-closed. 
\end{conjecture}

This was proved for the special case of $S^1$ actions in 
\cite{JKK}, but the method of proof does not
admit any obvious generalization. We will prove the following.

\begin{theorem} Conjecture \ref{FlowClosedConjecture} is true when $G$
is a torus.
\end{theorem}

This is an immediate consequence of Proposition \ref{FlowClosedness}
and Theorem \ref{GlobalEstimate}.
To appreciate why this result requires some effort, let us contrast it with
the analogous statement for $|\mu|^2$, where $\mu$ is an ordinary moment
map associated to a linear action on a Hermitian vector space $V$. It 
is immediate from the definitions that $\nabla |\mu|^2 = 2 I v_{\mu}$,
where $v_\mu$ denotes the fundamental vector field of $\mu$.
Hence, the gradient trajectory through a point $x$ always remains in the
$G_\BC$-orbit through $x$. Thus it suffices to restrict attention to the
$G_\BC$-orbits. Flow-closedness is then a consequence of the following
proposition, which follows easily from the $KAK$ decomposition of
reductive Lie groups \cite{Knapp}.

\begin{proposition}[{Sjamaar \cite[Lemma 4.10]{SjamaarConvexity}}] 
\label{SjamaarTrick}
Suppose $G_\BC$ acts linearly
on $V$, let $\{g_n\}$ be a sequence of points in $G_\BC$, and 
let $\{x_n\}$ be a bounded sequence in $V$. Then the sequence
$\{g_n x_n\}$ is bounded if $\{\mu(g_n x_n)\}$ is bounded.
\end{proposition}

Now consider the gradient flow of $|\muc|^2$ on $T^\ast V$. Since 
$|\muc|^2 = |\mu_J|^2 + |\mu_K|^2$, we see that
$\nabla |\muc|^2 = 2 J v_{\mu_J} + 2 K v_{\mu_K}$.
The problem is immediate: due to the simultaneous appearance of $J$ and $K$,
the gradient trajectories appear to lie
on the orbits of a ``quaternification'' of $G$,
but in general no suitable quaternification of $G$ exists.
More precisely, if we let  $\mathcal{D}$ denote
the distribution on $V$ generated by the $G$ action
(i.e.\ the integrable distribution whose leaves are the $G$-orbits), then
$\mathcal{D}_\BC := \mathcal{D} + I \mathcal{D}$ is integrable,
whereas the distribution
$\mathcal{D}_\BH := \mathcal{D} + I \mathcal{D} 
 + J \mathcal{D} + K \mathcal{D}$ on $T^\ast V$ is not integrable in general.

In a few specific examples, a detailed study of $\mathcal{D}_\BH$ leads to
definite conclusions about the gradient flow, but at present
we cannot prove a general result in this direction. In any case,
we will not pursue this approach in the present work (but see 
Remarks \ref{CrossTermRemark} and \ref{SubriemannianRemark}, which are related
to this problem).

With these considerations in mind, it would be illuminating to have a proof
of the flow-closedness of $|\mu|^2$ that relies neither on arguments
involving $G_\BC$ nor on properness, as such a proof might
generalize to the hyperk\"ahler setting. 
This is exactly the content the main result of this chapter, Theorem \ref{GlobalEstimate}. The key idea is to
relax the assumption of properness to the weaker condition of satisfying
a certain gradient inequality. 

In the case of ADHM spaces \cite{ADHM}, a gradient-like flow
was studied by Boyer and Mann \cite{BoyerMann}, following an idea of Taubes
\cite{TaubesModuli} to perturb $\nabla |\muhk|^2$ in such a way that its
flow is better behaved analytically. However, this approach
does not seem to be applicable to reduction at non-zero values of the moment
map. The idea of using gradient estimates to study the flow of $\grad|\muhk|^2$
directly was inspired in part by \cite{WilkinHiggs}.

\subsection{\loj\ Inequalities} \label{LojasiewiczInequalities} 

We begin by giving a precise definition of the type of inequality we
wish to consider, as well as its most important consequence.

\begin{definition} Let $(M, g)$ be a complete Riemannian manifold. A smooth
real-valued function $f$ on $M$ is said to satisfy a
\emph{global \loj\ inequality} if for any real number $f_c$ in the closure of
the image of $f$, there exist constants $\epsilon > 0$, $k > 0$, and
$0 < \alpha < 1$, such that
\begin{equation} |\grad f(x)| \geq k|f(x) - f_c|^\alpha, \end{equation}
for all $x \in M$ such that $|f(x) - f_c| < \epsilon$.
\end{definition}

\begin{remark} The term global \loj\ inequality is borrowed from
\cite{JKSGlobalLojasiewicz}; however, we use it in a different way, as
we are concerned specifically with bounding the gradient of $f$.
\end{remark}

\begin{proposition} 
\label{FlowClosedness}
Suppose $f$ satisfies a \gli\ and is bounded below.
Then $f$ is flow-closed.
\end{proposition}
\begin{proof} Let $x(t)$ be a trajectory of $-\grad f$.
Since $f(x(t))$ is decreasing and bounded below,
$\lim_{t \to \infty} f(x(t))$ exists. Call this limit $f_c$.
Let $\epsilon, k,$ and $\alpha$ be the constants appearing in the global
\loj\ inequality for $f$ with limit $f_c$.
For large enough $T$, we have $|f(x(t)) - f_c| < \epsilon$ 
whenever $t > T$. Consider $t_2 > t_1 > T$, and let $f_1 = f(x(t_1))$ and 
$f_2 = f(x(t_2))$. Since $\dot{x} = - \grad f$, we have that
\begin{equation} d(x(t_1), x(t_2)) \leq \int_{t_1}^{t_2} |\grad f(x(t))| dt, \end{equation}
where $d(x,y)$ denotes the Riemannian distance.
By the change of variables $t \mapsto f(x(t))$, we obtain
\begin{align}
d(x(t_1), x(t_2)) &\leq \int_{f_2}^{f_1} |\grad f|^{-1} df \\
&\leq k^{-1} \int_{f_2}^{f_1} |f - f_c|^{-\alpha} df \\
&= k^{-1} (1-\alpha)^{-1} \left(|f_1-f_c|^{1-\alpha} - |f_2-f_c|^{1-\alpha} \right)\\
&< k^{-1} (1-\alpha)^{-1} |f(x(T)) - f_c|^{1-\alpha}.
\end{align} 
Since $\alpha < 1$, the last expression can be made arbitrarily small by
taking $T$ sufficiently large, so we see that $\lim_{t \to \infty} x(t)$
exists,
and in particular the gradient trajectory is contained in a compact set.
\end{proof}

This argument establishes flow-closedness directly from the
{\L}ojasiewicz inequality, without appealing to compactness.
Thus it would be sufficient to show that $|\mu|^2$ satisfies such an inequality.
To motivate why we might expect this to be case,
we recall the classical \loj\ inequality.
\begin{theorem}[\loj\ Inequality \cite{BierstoneMilman,Lojasiewicz}]
\label{LojasiewiczInequality} 
Let $f$ be a real analytic function on a domain in $\BR^N$, and
let $c$ be a critical point of $f$. Then there is an open neighbourhood $U$ of $c$
such that for any compact subset $K \subset U$,
there are constants $k > 0$ and $0 < \alpha < 1$ such that the inequality
\begin{equation} \norm{\nabla f(x)} \geq k \left|f(x) - f(c)\right|^\alpha \end{equation}
holds for all $x \in K$.
\end{theorem}

If $f$ is a \emph{proper} real analytic function, then this immediately
implies that $f$ satisfies a global \loj\ inequality as defined above.
Since our primary concern is Morse theory, we can relax the assumption
of analyticity as follows.
\begin{proposition} Suppose $f$ is a proper Morse function. Then $f$ 
satisfies a global \loj\ inequality.
\end{proposition}
\begin{proof} By the Morse lemma, near each critical point we can choose
coordinates in which $f$ is real analytic. Hence $f$ satisfies the classical
\loj\ inequality near each critical point, and since $f$ is proper this can
be extended to a global inequality.
\end{proof}

For a moment map $\mu$, the function $|\mu|^2$ is in general
neither Morse nor Morse-Bott, but minimally degenerate
(in the sense of \cite{Kirwan}). Nonetheless, we can still
obtain a global \loj\ inequality whenever $\mu$ is proper.
\begin{proposition} Suppose $\mu$ is a moment map associated to an action 
of a compact Lie group, and suppose furthermore that $\mu$ is proper.
Then $|\mu|^2$ satisfies a \gli.
\end{proposition}
\begin{proof} 
Lerman \cite{LermanGradientFlow} proves that the moment map is locally real analytic
with respect to the coordinates induced by the local normal form, and uses this fact to
show that $|\mu|^2$ satisfies the classical \loj\ inequality.
Since $\mu$ is proper, this can be extended to a global inequality.
\end{proof}

Since we would like to drop the assumption of properness,  it is natural to ask
whether there are examples of moment maps which are \emph{not} proper
but nevertheless satisfy a global \loj\ inequality. 
The answer to this question is in the affirmative, at least when the action 
is linear. The following is the main theorem of this chapter, which we prove in 
Section \ref{Proof}.
\begin{theorem} 
\label{GlobalEstimate}
Let $\mu$ be a moment map associated to a unitary representation
of a compact group $G$ on a Hermitian vector space $V$.
Then $f = |\mu|^2$ satisfies a 
\gli. In detail, for every $f_c \geq 0$, there exist constants 
$k > 0$ and $\epsilon > 0$ such that
\begin{equation} \label{GlobalInequality}
|\nabla f(x)| \geq k|f(x) - f_c|^\frac{3}{4}
\end{equation}
whenever $|f(x) - f_c| < \epsilon$.
If $G$ is a torus, then the same holds for the functions $|\muc|^2$
and $|\muhk|^2$ associated to the action of $G$ on $T^\ast V$.
\end{theorem}

\begin{remark} \label{FirstNeemanRemark}
This theorem is a generalization of
\cite[Theorem A.1]{NeemanQuotientVarieties}.
However, in \cite{NeemanQuotientVarieties}, it is assumed that the 
constant term in the moment map is chosen so that $f$ is homogeneous
(see equation (\ref{OrdinaryMomentMapFormula})),
leading to an inequality of the form
\begin{equation} |\grad f| \geq k f^\frac{3}{4}, \end{equation}
which holds on all of $V$. Both sides are homogeneous of the same degree,
so that it suffices to prove the inequality on the unit sphere, allowing
the use of compactness arguments. In Theorem \ref{GlobalEstimate}, we
make no such assumption, and this complicates several steps of the proof.
\end{remark}

\begin{remark} If $X \subset V$ is a $G$-invariant subvariety, then it is easy to
see that Theorem \ref{GlobalEstimate} implies that the restriction
of $f$ to $X$ satisfies a global \loj\ inequality. Similarly,
if $X \subset T^\ast V$ is a $G$-invariant hyperk\"ahler subvariety, we deduce the inequality
for the restrictions of $|\muc|^2$ and $|\muhk|^2$.
\end{remark}

\begin{remark} In \cite{GradientSemialgebraic} and \cite{KurdykaGradients} it
is shown that certain classes of functions satisfy similar global {\L}ojasiewicz 
inequalities; indeed this was the motivation to consider such inequalities. 
However, these general theorems cannot rule out the possibility $\alpha \geq 1$,
which is not sharp enough to prove the boundedness of all gradient
trajectories. In this sense, the real content of 
Theorem \ref{GlobalEstimate} is the bound on the exponent. 
\end{remark}


In what follows,
let $G$ be a compact group acting unitarily on a Hermitian vector space $V$,
with moment map $\mu$. Note that since the action is linear, we have
$H_G^\ast(V) = H_G^\ast(\mathrm{point}) =: H_G^\ast$. For $\alpha_\BR \in \fgd$
and $\alpha_\BC \in \fgd_\BC$, we denote
\begin{align}
\FX(\alpha_\BR) &:= V \reda{\alpha_\BR} G, \\
\FM(\alpha_\BR, \alpha_\BC) &:= T^\ast V \rreda{(\alpha_\BR, \alpha_\BC)} G.
\end{align}

\begin{corollary}
\label{DeformationRetract} Let $f = |\mu|^2$.
Then for each component $C$ of the 
critical set of $f$, the gradient flow defines a $G$-equivariant homotopy equivalence from 
the stable manifold $S_C$ to the critical set $C$. If $G$ is abelian, then
the same holds for the functions $|\muc|^2$ and $|\muhk|^2$.
\end{corollary}

\begin{remark}
If $\mu$ is assumed to be proper then this is a special case of the main 
theorem of \cite{LermanGradientFlow}. However, the essential
ingredient of the proof is the \loj\ inequality.
We give the proof below for completeness, but the details
do not differ significantly from \cite{LermanGradientFlow}.
\end{remark}

\begin{proof}
We define a continuous map $F: S_C \times [0, \infty) \to S_C$ by
$(x, t) \mapsto x(t)$, where $x(t)$ is the trajectory of $-\grad f$ beginning
at $x$, evaluated at time $t$.
By Proposition \ref{FlowClosedness}, we can extend
this to a map $F: S_C \times [0, \infty] \to S_C$ by
$(x, \infty) \mapsto \lim_{t \to \infty} x(t)$. This map is the identity
when restricted to $C$, and maps $S_C \times \{\infty\}$ to $C$.
We must verify that this extended map is continuous.

We must show that for any $x_0 \in S_C$ and any sequence $\{(x_n, t_n)\}$ of 
points in $S_C \times [0, \infty]$ satisfying 
$\lim_{n\to\infty} x_n = x_0 \in S_C$ and
$\lim_{n\to\infty} t_n = \infty$ that $\lim_{n\to\infty} x_n(t_n)$ exists.
We will show that it
is equal to $x_c := \lim_{t\to\infty} x_0(t)$. 
Let $f_c$ be the value of $f$ on $C$ and let $\epsilon, k$ be the 
constants appearing in Theorem \ref{GlobalEstimate}.
Given such a sequence,
let $\eta > 0$ and assume $\eta < \epsilon$ but is otherwise arbitrary;
let $T > 0$ be chosen large enough so that $|f(x_0(t)) - f_c| < \eta$
for $t > T$; and let $N > 0$ be chosen to that $t_n > T$ for $n > N$.
The map $S_C \to S_C$ given by $x \mapsto x(T)$ is continuous, 
so we can find $\delta > 0$ such that
\begin{equation} |x - x_0| < \delta \implies |x(T) - x_0(T)| < \eta. \end{equation}
Since $x \mapsto f(x(T))$ is continuous, we can shrink $\delta$ if 
necessary so that
\begin{equation} |x - x_0| < \delta \implies |f(x(T)) - f(x_0(T))| < \eta. \end{equation}
Choose $N$ larger if necessary such that $|x_n - x_0| < \delta$ for $n > N$.
We would like to show that $|x_n(t_n) - x_c| \to 0$ as $n \to \infty$.
For $n > N$ we have
\begin{align}
|x_n(t_n) - x_c|  &= |x_n(t_n) - x_n(T) + x_n(T) - x_0(T) + x_0(T) - x_c| \\
&\leq |x_n(t_n) - x_n(T)| + |x_n(T) - x_0(T)| + |x_0(T) - x_c|.
\end{align}
By our choice of $N$, the second term is bounded by $\eta$, and 
we may apply the argument in the proof of Proposition \ref{FlowClosedness}
to show that the third term is bounded by
$k^{-1} |f(x_0(T)) - f_c|^\frac{1}{4} < 4k^{-1} \eta^\frac{1}{4}$.
Finally, to bound the first term we again apply the argument of Proposition
\ref{FlowClosedness} to obtain the bound
\begin{equation} 4k^{-1} \left(|f(x_n(T)) - f_c|^\frac{1}{4} 
 - |f(x_n(t_n)) - f_c|^\frac{1}{4} \right) < 
 4k^{-1} |f(x_n(T)) - f_c|^\frac{1}{4}. \end{equation}
Since $|x_n - x_0| < \delta$, we have
\begin{align}
|f(x_n(T)) - f_c| &= |f(x_n(T)) - f(x_0(T)) + f(x_0(T)) - f_c| \\
&\leq |f(x_n(T)) - f(x_0(T))| + |f(x_0(T)) - f_c| \\
&< 2 \eta.
\end{align}
Hence the first term is bounded by 
$4 k^{-1} (2\eta)^\frac{1}{4} < 8k^{-1} \eta^\frac{1}{4}$, and we obtain
\begin{equation} |x_n(t_n) - x_c| \leq 12k^{-1} \eta^\frac{1}{4} + \eta. \end{equation}
Since we may take $\eta$ arbitrarily small, we see that
$|x_n(t_n) - x_c| \to 0$.
\end{proof}

\begin{corollary} \label{HomotopyCorollary}
If $\alpha_\BC \in \fg_\BC^\ast$ is regular central, then for any central
$\alpha_\BR \in \fgd$, the set $\mur^{-1}(\alpha_\BR) \cap \muc^{-1}(\alpha_\BC)$ is
a $G$-equivariant deformation retract of $\muc^{-1}(\alpha_\BC)$, and
in particular
\begin{equation}
  H_G^\ast(\muc^{-1}(\alpha_\BC)) \iso H^\ast( \FM(\alpha_\BR, \alpha_\BC) ). 
\end{equation}
\end{corollary}
\begin{proof} Since $\alpha_\BC$ is regular central, $\muc^{-1}(\alpha_\BC)$ is
a $G$-invariant complex submanifold of $T^\ast V$, and furthermore $G$
acts on $\muc^{-1}(\alpha_\BC)$ with at most discrete stabilizers. Hence the
only component of the critical set of $|\mur-\alpha_\BR|^2$ that intersects
$\muc^{-1}(\alpha_\BC)$ is the absolute minimum, which occurs on $\mur^{-1}(\alpha_\BR)$.
By Corollary \ref{DeformationRetract}, the gradient flow of $-|\mur-\alpha_\BR|^2$
gives the desired $G$-equivariant deformation retract from $\muc^{-1}(\alpha_\BC)$ to
$\mur^{-1}(\alpha_\BR) \cap \muc^{-1}(\alpha_\BC)$. Thus
\begin{equation} H_G^\ast(\muc^{-1}(\alpha_\BC)) 
\iso H_G^\ast(\mur^{-1}(\alpha_\BR) \cap \muc^{-1}(\alpha_\BC))
\iso H^\ast( \FM(\alpha_\BR, \alpha_\BC) ). \end{equation}
\end{proof}

\begin{corollary} \label{SurjectivityForTori}
If $\alpha$ is a regular central value of $\mu$, then the
Kirwan map $H_G^\ast \to H^\ast(\FX(\alpha))$ is surjective. If $G$ is
a torus and $(\alpha_\BR, \alpha_\BC)$ is a regular value of the hyperk\"ahler
moment map, then the hyperk\"ahler Kirwan map $H_G^\ast \to H^\ast(\FM(\alpha_\BR,\alpha_\BC))$ is surjective.
\end{corollary}
\begin{proof} Theorem \ref{GlobalEstimate} and Proposition \ref{FlowClosedness}
show that $|\mu-\alpha|^2$ is flow-closed, so we obtain surjectivity of 
$H_G^\ast \to H^\ast(\FX(\alpha))$ by Theorem \ref{KirwanSurjectivity}.

In the hyperk\"ahler case, if $G$ is a torus
then $|\muc-\alpha_\BC|^2$ is flow-closed, so by Theorem \ref{SurjectivityCriterion}
we obtain surjectivity of $H_G^\ast \to H^\ast_G(\muc^{-1}(\alpha_\BC))$.
By Corollary \ref{HomotopyCorollary}, the map
$H_G^\ast(\muc^{-1}(\alpha_\BC)) \to H^\ast(\FM(\alpha_\BR, \alpha_\BC))$ is an
isomorphism, so we have surjectivity of $H_G^\ast \to H^\ast(\FM(\alpha_\BR, \alpha_\BC))$.
\end{proof} 

\begin{remark}
Konno proved surjectivity of the map $H_G^\ast \to H^\ast(\FM(\alpha_\BR,\alpha_\BC))$ when $G$ is a
torus using rather different arguments \cite{KonnoToric}.
Konno also computed the kernel of the Kirwan map, giving an explicit description
of the cohomology ring $H^\ast(\FM(\alpha_\BR, \alpha_\BC))$. We study this case in 
detail in Section \ref{sec-toric}, and we will see that Morse theory allows us to compute the kernel very easily. 
\end{remark}

\subsection{Proof of the {\L}ojasiewicz Inequality} \label{Proof}
Let $G$ be a compact Lie group with Lie algebra $\fg$, and
suppose $G$ acts unitarily on a Hermitian vector space $G$. 
Without loss of generality we will regard $G$ as a subgroup of $U(V)$
and identify $\fg$ with a Lie subalgebra of $\mathfrak{u}(V)$,
which we identify with the Lie algebra of skew-adjoint matrices. 
We will use the trace norm on $\mathfrak{u}(V)$ to induce an invariant
inner product on $\fg$, and use this to identify $\fg$ with its dual.
For any $\xi \in \fg$ there is
a fundamental vector field $v_\xi$ which is given by $v_\xi(x) = \xi x$.
We denote by $\stab(x)$ the Lie algebra of the stabilizer of a point
$x \in V$; i.e.
\begin{equation}
\stab(x) = \{ \xi \in \fg \suchthat v_\xi(x) = 0\}.
\end{equation}
If we fix an orthonormal basis $\{e_a\}$ of $\fg$, then a moment map
is given by
\begin{equation} \label{OrdinaryMomentMapFormula}
\mu(x) = \sum_a \frac{1}{2} \ev{i e_a x,x} - \alpha,
\end{equation}
where $i = \sqrt{-1}$ is the complex structure on $V$ and $\alpha$
is any central element of $\fg$. Then for $f = |\mu|^2$,
we have
\begin{equation}
\label{GradFormula}
\nabla f(x) = 2i v_{\mu(x)}(x) = \sum_a 2 i \mu^a(x) e_a x,
\end{equation}
and since $i$ is unitary, we have that
\begin{equation}
\label{GradEquality}
|\nabla f| = 2|v_\mu|. 
\end{equation}

\begin{lemma} \label{CrossTermVanish}
Suppose $G$ is abelian, and consider its action on
$T^\ast V$. Let $f_i = |\mu_i|^2$ for $i = I,J,K$. Then
$\ev{\nabla f_i,\nabla f_j} = 0$ for $i \neq j$.
\end{lemma}
\begin{proof} We compute:
\begin{align}
\ev{\nabla f_J,\nabla f_K} 
&= 4 \ev{J v_{\mu_J},K v_{\mu_K}} \\
&= 4 \ev{I v_{\mu_J},v_{\mu_K}} \\
&= 4 \omega_I(v_{\mu_J}, v_{\mu_K}) \\
&= 4 \ev{\mu_J,d\mu_I(v_{\mu_K})} \\
&= 4 \ev{\mu_J,[\mu_K, \mu_I]} \\
&= 0.
\end{align}
Similar computations show that the other two cross terms vanish.
\end{proof}
\begin{remark} \label{CrossTermRemark}
The proof of this lemma makes it clear why the assumption
that $G$ is abelian is so useful in the hyperk\"ahler setting.
The cross term 
\begin{equation} \ev{\nabla f_J,\nabla f_K} = 4\ev{\mu_I,[\mu_J, \mu_K]} \end{equation}
is exactly the obstruction to proving an estimate in the general nonabelian case.
The function on the right hand side is very natural, and seems to be
genuinely hyperk\"ahler, having no analogue in
symplectic geometry. Numerical experiments suggest
that it is small in magnitude compared to $|\nabla f_J| + |\nabla f_K|$,
but we do not know how to prove this. A theorem in this direction might
be enough to prove flow-closedness (and hence Kirwan surjectivity) in 
general. 
\end{remark}

\begin{remark} \label{SubriemannianRemark}
In light of the discussion following Proposition
\ref{SjamaarTrick}, the cross term $\ev{\mu_I,[\mu_J, \mu_K]}$
should have an interpretation in terms of the geometry of the 
non-integrable distribution $\mathcal{D}_\BH$.
Subriemannian geometry may have a key role to play in proving
Kirwan surjectivity for hyperk\"ahler quotients by nonabelian groups.
\end{remark}

\begin{proposition} 
\label{ReduceToTorus} 
Let $T \subset G$ be a maximal torus of $G$, and let $\mu_G$ and $\mu_T$ be
the corresponding moment maps. Let $f_G = |\mu_G|^2$ and $f_T = |\mu_T|^2$.
Suppose that for any $f_c \geq 0$, $f_T$ satisfies a global \loj\ inequality. 
Then $f_G$ satisfies a global \loj\ inequality with the same constants and 
exponent.
\end{proposition}
\begin{proof}
Since $G$ is compact, for each
$x \in V$ we can find some $g \in G$ so that $\Ad_{g} \mu_G(x) \in \ft$. 
By equivariance of the moment map, we have $\Ad_{g} \mu_G(x) = \mu_G(g x) \in \ft$.
Hence $\mu_G(g x) = \mu_T(g x)$, so that $|v_{\mu_G(gx)}| = |v_{\mu_T(gx)}|$.
Using equality (\ref{GradEquality}), this tells us that
$|\nabla f_G(gx)| = |\nabla f_T(gx)|$. 
Since $f_G(gx) = |\mu_G(gx)|^2 = |\mu_T(gx)|^2 = f_T(gx)$,
we deduce the \loj\ inequality for $f_G$ from the inequality for $f_T$.
\end{proof}

We assume for the remainder of this section that $G$ is a torus.

\begin{proposition}
\label{StrongerCondition} 
Fix $f_c \geq 0$, and suppose that for each 
$\mu_c \in \fg$ satisfying $|\mu_c|^2 = f_c$, there exist
constants $\epsilon' > 0$ and $c' > 0$ (depending on $\mu_c$) such that 
\begin{equation} \label{GlobalInequalityWeaker}
|\nabla f(x)| \geq c'|f(x) - f_c|^\frac{3}{4}
\end{equation}
whenever $|\mu(x) - \mu_c| < \epsilon'$. Then $f$ satisfies a global
\loj\ inequality, i.e., there exist constants $\epsilon > 0$ and
$c > 0$ so that the inequality (\ref{GlobalInequalityWeaker}) holds whenever 
$|f(x) - f_c| < \epsilon$. 
\end{proposition}
\begin{proof} Suppose that for each $\mu_c$ as above we can find
constants $\epsilon(\mu_c)$ and $c(\mu_c)$ so that inequality
(\ref{GlobalInequalityWeaker}) holds. Let
$U(\mu_c)$ be the $\epsilon(\mu_c)$-ball in $\fg$ centered at $\mu_c$.
These open sets cover the sphere $S$ of radius $\sqrt{f_c}$ in $\fg$, and by
compactness we can choose a finite subcover. Denote this finite subcover by
$\{U_i\}_{i=1}^n$, with centers $\mu_i$ and constants $c_i$, and
let $c = \min_i c_i$. The finite union $\cup_i U_i$ 
contains an $\epsilon$-neighborhood of $S$ for some sufficiently small
$\epsilon$.  Since $f(x) = |\mu(x)|^2$, if we choose
$\epsilon' > 0$ sufficiently small then $|f(x) - f_c| < \epsilon'$ implies
that $\left| |\mu(x)| - \sqrt{f_c} \right| < \epsilon$, so that 
$\mu(x) \in \cup_i U_i$. In particular, there is some $j$ such that
$\mu(x) \in U_j$, and by inequality (\ref{GlobalInequalityWeaker}) we have
\begin{equation} |\nabla f(x)| \geq c_j |f(x) - f_c|^\frac{3}{4} 
\geq c|f(x) - f_c|^\frac{3}{4}, \end{equation}
as desired.
\end{proof} 

Before giving the proof of Theorem \ref{GlobalEstimate}, we isolate some of
the main steps in the following lemmas.
Let us introduce the following notation. For $x \in V \setminus \{0\}$,
let $\hat{x}$ denote its projection to the unit sphere, i.e.\
$\hat{x} = x/|x|$ or equivalently $x = |x|\hat{x}$.
Since the action is linear, we have that $v_\xi(x) = |x| v_\xi(\hat{x})$
and $\stab(x) = \stab(\hat{x})$.

\begin{lemma} 
\label{KeyEstimate1} 
Fix $\hat{y}$ 
in the unit sphere in $V$. Let $P$ be the orthogonal projection from $\fg$ to 
$\stab(\hat{y})^\perp$, and $Q = 1-P$. Then there is a neighborhood 
$U$ of $\hat{y}$ such that for any $\xi \in \fg$, inequalities
\begin{align}
\label{KeyInequality1}
 |v_\xi(x)| &\geq c |x| | P \xi | \\
\label{KeyInequality2}
 |v_\xi(x)| &\geq c' |v_{P \xi}(x)| \\
\label{KeyInequality3}
 |v_\xi(x)| &\geq c'' \left(|v_{P\xi}(x)| + |v_{Q\xi}(x)| \right)
\end{align}
hold for all $x$ such that $\hat{x} \in U$.
The constants $c,c',c''$ are 
positive and depend only on $\hat{y}$ and $U$ but not on $x$ or $\xi$.
\end{lemma}

\begin{remark} A version of this lemma appears as part of the proof of
\cite[Theorem A.1]{NeemanQuotientVarieties}, though it is not stated
exactly as above. We repeat the argument below so that our proof of Theorem
\ref{GlobalEstimate} is self-contained.
\end{remark}

\begin{proof}
Fix $\hat{y}$ and let $P$ and $Q$ be as above. 
Let $W$ be the smallest $G$-invariant subspace of $V$ containing $\hat{y}$,
and let $P_W: V \to W$ be the orthogonal projection. Note that $W$ is
generated by vectors of the form $\xi_1 \cdots \xi_l \hat{y}$, with
$\xi_i \in \fgc$. Since $P_W$ is a projection, $|v_\xi| \geq |P_W v_\xi|$, 
so to establish inequality 
(\ref{KeyInequality1}) it suffices to show that 
$|P_W v_\xi| \geq c|P \xi|$.
Note that $P_W$
is equivariant, i.e.\ $\xi P_W = P_W \xi$ for all $\xi \in \fg$. Note also
that since $K$ is abelian, if $\xi \in \stab(\hat{y})$ then $\xi \in \ann(W)$,
since $\xi \xi_1 \cdots \xi_l \hat{y} = \xi_1 \cdots \xi_l \xi \hat{y} = 0$.
For any orthonormal basis $\{e_a\}_{i=1}^d$ of $\fg$ chosen so that
$\{e_a\}_{i=1}^n$ is an orthonormal basis of $\stab(\hat{y})^\perp$ and
$\{e_a\}_{i=n+1}^d$ is an orthonormal basis of $\stab(\hat{y})$, we have
$P\xi = \sum_{a=1}^d \xi^a Pe_a = \sum_{a=1}^n \xi^a e_a$.
Similarly, we find
\begin{equation} P_W v_\xi(x) = \sum_{a=1}^d P_W \xi^a e_a x 
= \sum_{a=1}^d \xi^a e_a P_W x
= \sum_{a=1}^n \xi^a e_a P_W x = P_W v_{P\xi}(x). \end{equation}
Taking norms, we see that
\begin{equation} |P_W v_\xi(x)|^2 
= \sum_{a=1}^n \sum_{b=1}^n \xi^a \xi^b \ev{P_W e_a x,P_W e_b x}
 = (P\xi)^T H(x) (P\xi), \end{equation}
where $H(x)$ is the matrix with entries 
$H_{ab}(x)  = \ev{P_W e_a x,P_W e_b x}$ for $a,b = 1, \ldots, n$.
By construction,
this matrix is positive definite at $\hat{y}$, so for a sufficiently 
small neighborhood $U$ of $\hat{y}$, we obtain
$|P_W v_\xi(\hat{x})|^2 \geq c |P\xi|^2$,
with the positive constant $c$ depending only on $\hat{y}$ and the choice of
neighborhood $U$. For any $x$ with $\hat{x} \in U$, we obtain
$|v_\xi(x)| = |x| |v_\xi(\hat{x})| \geq c |x| |P\xi|$,
which is inequality (\ref{KeyInequality1}).

We can deduce inequality (\ref{KeyInequality2}) from inequality 
(\ref{KeyInequality1}) as follows. We have
\begin{equation} |v_{P\xi}(\hat{x})| = |\sum_{a=1}^n \xi^a e_a \hat{x}|
 \leq |P\xi|\sum_{a=1}^n |e_a \hat{x}|. \end{equation}
Shrinking $U$ if necessary, we can assume that the functions
$|e_a \hat{x}|$ are bounded on $U$, and so we obtain
$|v_{P\xi}(\hat{x})| \leq c'|P\xi| \leq c^{-\frac{1}{2}} c' |v_\xi(\hat{x})|$.
Hence for any $x$ with $\hat{x} \in U$ we have
\begin{equation}
  |v_{P\xi}(x)| = |x| |v_{P\xi}(\hat{x})| \leq c^{-\frac{1}{2}} c' |x| |v_\xi(\hat{x})| = c^{-\frac{1}{2}} c' |v_\xi(x)|,
\end{equation}
which is inequality (\ref{KeyInequality2}).

To establish inequality (\ref{KeyInequality3}), first note the following
consequence of the triangle inequality. If $v,w$ are vectors in some
normed vector space, and $|v + w| \geq a |v|$ with $a > 0$, then we have
\begin{equation} \label{MinorUsefulInequality}
 |v| + |w| = |v| + |v+w - v| \leq 2|v| + |v+w| 
\leq \left(1 + \frac{2}{a} \right)|v+w|.
\end{equation}
Since $\xi = P\xi + Q\xi$, we have
$v_\xi(x) = v_{P\xi}(x) + v_{Q\xi}(x)$, so applying inequality
(\ref{KeyInequality2}) together with the inequality
(\ref{MinorUsefulInequality}) above, we obtain
\begin{equation} |v_\xi(x)| \geq c''' \left( |v_{P\xi}(x)| + |v_{Q\xi}(x)| \right), \end{equation}
with the constant $c'''$ depending only on $\hat{y}$ and the neighborhood
$U$.
\end{proof}

\begin{lemma}
\label{KeyEstimate2}
Let $f = |\mu|^2$ and fix $\mu_c \in \fg$. Then there exist
constants $c > 0$ and $\epsilon > 0$ (depending on $\mu_c$) such that whenever
$|\mu(x) - \mu_c| < \epsilon$, we have 
\begin{equation} \label{KeyInequality4}
|x|^2 |f(x)| \geq c \left|f(x) - f_c \right|^\frac{3}{2}.
\end{equation}
\end{lemma}
\begin{proof}
Fix some particular $\mu_c$. Recall that $\mu$ is quadratic in the coordinate 
$x$ with no linear terms. Thus $\mu$ is affine in the coordinates
$v_{ij} = x_i \bar{x}_j$, and we may write
$\mu(x) = \phi(v)$, where $\phi(v) = Av - \alpha$, for some linear 
transformation $A: V \otimes V \to \fg$. We have
$|v_{ij}| = |x_i||x_j| \leq \frac{1}{2}\left(|x_i|^2 + |x_j|^2 \right)$,
so that
$|v| \leq \frac{1}{2} \sum_{i,j} |x_i|^2 + |x_j|^2 = N |x|^2$,
where $N = \dim V$. Thus
$|x|^2 |f(x)| = |x|^2 |\phi(v)|^2 \geq N^{-1} |v| |\phi(v)|^2$,
so it suffices to show that
\begin{equation} |v| |\phi(v)|^2 \geq c \left| |\phi(v)|^2 - f_c \right|^\frac{3}{2}, \end{equation}
whenever $|\phi(v) - \mu_c| < \epsilon$. This follows immediately
from Lemma \ref{RhoPhiLemma}, which we state and prove below.
\end{proof}

\begin{lemma}
\label{RhoPhiLemma} Let $V_1$ and $V_2$ be inner product spaces, and consider
an affine map $\phi: V_1 \to V_2$ given by $\phi(v) = Av - \alpha$
for some linear map $A: V_1 \to V_2$ and constant $\alpha \in V_2$. 
Then for any $\phi_c$ in the image of $\phi$, there exist constants $c > 0$ and 
$\epsilon > 0$ such that
\begin{equation} |v| |\phi(v)|^2 \geq c \npvz^\frac{3}{2} \end{equation}
whenever $|\phi(v) - \phi_c| < \epsilon$.
\end{lemma}
\begin{proof} To avoid unnecessary clutter, we use the shorthand 
$v^2 := |v|^2$ below.
First note that if $P$ is a projection such that $AP = A$,
then we have $\phi(v) = \phi(Pv)$, and so
\begin{equation} |v| |\phi(v)|^2 = |v| |\phi(Pv)|^2 \geq |Pv| |\phi(Pv)|^2. \end{equation}
Taking $P$ to be the orthogonal projection from $V_1$ to $(\ker A)^\perp$, 
without loss of generality we can assume that $A$ is injective. Similarly, 
without loss of generality we may assume that $A$ is surjective.
Suppose that $\phi_c \in V_2$ is fixed. Pick $v_c \in V_1$ so that
$\phi(v_c) = \phi_c$. There are four possible cases.

Case 1: $v_c = 0, \phi_c = 0$. In this case, $\alpha = 0$, and
$\phi(v) = Av$, so that $|\phi(v)| \leq |A| |v|$. Thus
$|v| |\phi(v)|^2 \geq |A|^{-1} |\phi(v)|^3$,
as desired. We may take $\epsilon$ to be any positive number, and
$c = |A|^{-1}$.

Case 2: $v_c = 0, \phi_c \neq 0$. Take $\epsilon \leq |\phi_c|/2$.
Then
\begin{align}
|\phi(v)^2 - \phi_c^2|
&= |\ev{\phi(v)-\phi_c,\phi(v)+\phi_c}| \\
&\leq |\phi(v) - \phi_c| |\phi(v) + \phi_c | \\
&\leq |\phi(v) - \phi_c| \left( 2|\phi_c| + \epsilon \right) \\
&\leq \frac{5}{2} |\phi(v) - \phi_c| |\phi_c|,
\end{align}
so we have
$|\phi(v)^2 - \phi_c^2| \leq c_1 |\phi(v) - \phi_c| |\phi_c|$,
where $c_1$ is a numerical constant independent of $\phi_c$.
Since $|\phi(v) - \phi_c| < |\phi_c|/2$, we also have
$|\phi(v)^2 - \phi_c^2| \leq c_2 |\phi_c|^2$.
Combining these two inequalities, we have
$|\phi(v)^2 - \phi_c^2|^\frac{3}{2} \leq c_3 |\phi(v) - \phi_c| |\phi_c|^2$.
Since $\phi(v) - \phi_c = A(v - v_c) = Av$, we have
$|\phi(v) - \phi_c| \leq |A| |v|$. Putting this back into the previous
inequality, we obtain
$|\phi(v)^2 - \phi_c^2|^\frac{3}{2} \leq |A| |v| |\phi_c|^2$.
On the other hand, with our choice of $\epsilon$ we have
$|\phi| \geq \frac{1}{2} |\phi_c|$, so that
$|\phi(v)^2 - \phi_c^2|^\frac{3}{2} \leq 4 |A| |v| |\phi|^2$, 
as desired.

Case 3: $v_c \neq 0, \phi_c = 0$. Take 
$\epsilon = |A^{-1}|^{-1} |A|^{-1} |\alpha| / 2$.
Since in this case $A v_c = \alpha$, we have $\epsilon \leq |A^{-1}|^{-1}|v_c|/2$,
 and
\begin{align}
|v_c| &= |v - (v-v_c)| \\
&\leq |v| + |v-v_c| \\
&= |v| + |A^{-1} \phi(v)| \\
&\leq |v| + |A^{-1}| \epsilon \\
&\leq |v| + \frac{|v_c|}{2}.
\end{align}
Thus $|v| \geq |v_c|/2 \geq |A|^{-1} |\alpha|/2$. Then
\begin{equation} |\phi(v)|^3 \leq \epsilon |\phi(v)|^2 
= \frac{1}{2} |A^{-1}|^{-1}|A|^{-1} |\alpha| |\phi(v)|^2
\leq |A^{-1}|^{-1} |v| |\phi(v)|^2, \end{equation}
which is the desired inequality.

Case 4: $v_c \neq 0, \phi_c \neq 0$. Let $\epsilon'$ be chosen as in case (3),
and let $\epsilon = \min\{\epsilon', |\phi_c|/2\}$. As in the previous cases,
this choice of $\epsilon$ guarantees that $|v| \geq |v_c|/2$, and that
$|\phi(v)| \geq |\phi_c|/2$. As before,
\begin{align}
|\phi(v)^2 - \phi_c^2| &\leq |\phi(v) - \phi_c| |\phi_v + \phi_c| \\
&\leq \epsilon ( 2 |\phi_c| + \epsilon) \\
&\leq \frac{5}{4} |\phi_c|^2. 
\end{align}
Similarly, since $|v-v_c| \leq |v_c|/2$ and 
$|\phi(v) - \phi_c| \leq |A||v - v_c|$, we also have
\begin{equation}|\phi(v)^2 - \phi_c^2| \leq (3/4) |A| |v_c| |\phi_c|.\end{equation}
Putting these together, we have that
\begin{equation} |\phi(v)^2 - \phi_c^2|^\frac{3}{2} 
\leq c|v_c| |\phi_c|^2 \leq c' |v| |\phi(v)|^2, \end{equation}
where $c$ and $c'$ are numerical constants independent of $\phi_c$.
\end{proof}
Lemmas \ref{KeyEstimate1} and
\ref{KeyEstimate2} allow us to prove the following local estimates, which
are essential in the proof of Theorem \ref{GlobalEstimate}.
\begin{proposition} 
\label{FreeCase} 
Let $\mu_c \in \fg$ be fixed and $y$ is some point in $V$ with discrete 
stabilizer. Then there exists an open neighborhood $U$ of $\hat{y}$ and 
constants $c > 0$ and $\epsilon > 0$ such that
\begin{equation}
 |\nabla f(x)|^2 \geq k|f(x) - f_c|^\frac{3}{2} 
\end{equation}
for all $x \in V \setminus \{0\}$ such that $\hat{x} \in U$ and
$|\mu(x) - \mu_c| < \epsilon$.
\end{proposition}
\begin{proof}
Suppose $y \in V$ and $\stab(y) = 0$. Then by
Lemma \ref{KeyEstimate1}, there is a neighborhood $U$ of $\hat{y}$ so that
\begin{equation}
  |v_\xi(x)|^2 \geq c |x|^2 |\xi|^2
\end{equation}
for all $x$ such that $\hat{x} \in U$. Take $\xi = \mu(x)$ and apply
Lemma \ref{KeyEstimate2} to find $c'$ and $\epsilon$ so that
\begin{equation}
 |v_{\mu(x)}(x)|^2 \geq c' \left| |\mu(x)|^2 - f_c \right|^\frac{3}{2}, 
\end{equation}
whenever $|\mu(x) - \mu_c| < \epsilon$.
Since $|\nabla f(x)| = 2|v_{\mu(x)}(x)|$, this gives the desired inequality.
\end{proof}
\begin{proposition} 
\label{InductionStep}
Let $y \in V$ and suppose $\stab(y)$ is a proper nontrivial subspace of $\fg$. 
Then there are proper
subtori $G_1, G_2$ of $G$ such that $G \iso G_1 \times G_2$, a neighborhood 
$U$ of $\hat{y}$, and a constant $c > 0$ such that
\begin{equation}
 |\nabla f(x)| \geq c\left( |\nabla f_{G_1}(x)| + |\nabla f_{G_2}(x)| \right)
\end{equation}
for all $x \in V \setminus \{0\}$ with $\hat{x} \in U$, 
where $f_{G_1} = |\mu_{G_1}|^2$ and $f_{G_2} = |\mu_{G_2}|^2$.
\end{proposition}
\begin{proof} Let $\fg_1 = \stab(y)$ and 
$\fg_2 = \fg_1^\perp$. Since $\fg$ is abelian, both
$\fg_1$ and $\fg_2$ are Lie subalgebras, and $\fg = \fg_1 \oplus \fg_2$.
Then $\mu_G = \mu_{G_1} \oplus \mu_{G_2}$, and 
$|\mu_G|^2 = |\mu_{G_1}|^2 + |\mu_{G_2}|^2$, so the result follows immediately
from Lemma \ref{KeyEstimate1} and inequality (\ref{KeyInequality3}).
\end{proof}
\begin{proof}[Proof of Theorem \ref{GlobalEstimate}]
By Proposition \ref{ReduceToTorus} we will assume that $G$ is a torus.
By Proposition \ref{StrongerCondition}, it suffices to show that for
each $\mu_c \in \fg$, there is some $\epsilon > 0$ so that
inequality (\ref{GlobalInequality}) holds when
$|\mu(x) - \mu_c| < \epsilon$.
Additionally, since $0$ is always a critical point of $f$, it suffices
to prove the estimate only on $V \setminus \{0\}$.
Furthermore, it suffices 
to show that each point $\hat{y}$ of the unit sphere has a neighborhood $U$
such that the estimate holds for all $x$ with $\hat{x} \in U$,
since by compactness we can choose finitely many
such neighborhoods to cover the unit sphere, and this yields the
inequality on $V \setminus \{0\}$.

We will prove the estimate by induction on the dimension of $G$. First
suppose $\dim G = 1$. Then we may assume without loss of generality that 
$G$ acts locally freely on $V \setminus \{0\}$, since otherwise
the fundamental vector field $v_\xi(x)$ vanishes on a nontrivial subspace
and we can restrict our attention to its orthogonal complement.
Then Proposition \ref{FreeCase} yields the desired neighborhoods and estimates.

Now assume that $\dim G > n$ and we have proved the estimate for tori
of dimension $\leq n$. Without loss of generality, we can assume that there is 
no nonzero vector in $V$ which is fixed by all of $G$ (since we can restrict to
its orthogonal complement).
Let $\hat{y}$ be some point in the unit sphere, and let
$\fg_1 = \stab(\hat{y})$. If $\fg_1 = 0$, we may apply
Proposition \ref{FreeCase} to get a neighborhood $U_{\hat{y}}$ and a constant 
$k_{\hat{y}}$ such that the estimate holds on $U_{\hat{y}}$.
Otherwise,  let $\fg_2 = \fg_1^\perp$ so that 
$\fg = \fg_1 \oplus \fg_2$,  corresponding to subtori $G_1$ and $G_2$.
Then we may apply Proposition
\ref{InductionStep} to find a neighborhood $U_{\hat{y}}$ so that
\begin{equation} |\nabla f(x)| 
\geq k \left( |\nabla f_{G_1}(x)| + |\nabla f_{G_2}(x)| \right), \end{equation}
holds for all $x$ with $\hat{x} \in U_{\hat{y}}$.
Let $P_i: \fg \to \fg_i$, $i = 1, 2$ be the orthogonal projections, 
$\mu_{c,i} = P_i \mu_c$,
and $f_{c,i} = |\mu_{c,i}|^2$. Since $\mu_{G_i}$ are the moment maps for the
action of $G_i$, which are tori of dimension $\leq n$, we may apply
the induction hypothesis to find
a neighborhood $U$ of $\hat{y}$ and constants $\epsilon > 0$, $c > 0$
so that
\begin{align}
|\nabla f_{G_1}(x)| &\geq& k' |f_{G_1}(x) - f_{c,1}|^\frac{3}{4}, \\
|\nabla f_{G_2}(x)| &\geq& k' |f_{G_2}(x) - f_{c,2}|^\frac{3}{4},
\end{align}
for all $x$ such that $\hat{x} \in U$ and $|\mu_{G_i}(x) - \mu_{c,i}| < \epsilon$.
For any non-negative numbers $a,b$, we have
$a^\frac{3}{4} + b^\frac{3}{4} \geq (a + b)^\frac{3}{4}$,
so we obtain
\begin{equation} |\nabla f(x)|
\geq k''\left(|f_{G_1}(x) - f_{c,1}| + |f_{G_2}(x) - f_{c,2}|\right)^\frac{3}{4}
\geq k''|f(x) - f_c|^\frac{3}{4}, \end{equation}  
whenever $|\mu(x) - \mu_c| < \epsilon$, as desired.

To obtain the estimate for the functions $|\muc|^2$ and $|\muhk|^2$, 
we simply note that by Lemma \ref{CrossTermVanish}, the norm of the gradient 
is bounded below by a sum of terms of the form $\left| \nabla |\mu_i|^2 \right|$,
and since we can bound each term individually we obtain a bound for the sum.
\end{proof}

\begin{remark} \label{SecondNeemanRemark}
As pointed out in Remark \ref{FirstNeemanRemark}, Theorem \ref{GlobalEstimate}
is a generalization of \cite[Theorem A.1]{NeemanQuotientVarieties}
and much of the argument is similar.
The main new ingredients are Lemma \ref{KeyEstimate2} and
Proposition \ref{FreeCase}, which are absolutely
essential in handling the general case of a nonzero constant term
in the moment map.
\end{remark}
\section{Application to Hypertoric Varieties} \label{sec-toric}

\subsection{Toric Varieties} 
Let $T$ be a subtorus of the standard $N$-torus $(S^1)^N$,
with quotient $K := (S^1)^N / T$. We have a short exact sequence
\begin{equation} \label{ToriSES}
1  \to  T  \xrightarrow{i}  (S^1)^N  \xrightarrow{\pi}  K  \to  1.
\end{equation}
Taking Lie algebras, we have
\begin{equation}\label{SES}
0  \to  \ft \xrightarrow{i}  \BR^N  \xrightarrow{\pi}  \fk  \to 0,
\end{equation}
\begin{equation}
\label{DualSES}
0 \to \fkd \xrightarrow{\pi^\ast} \BR^N \xrightarrow{i^\ast} \ftd \to 0.
\end{equation}
Recall the standard Hamiltonian action of $(S^1)^N$ on $\BC^N$.
This restricts to a Hamiltonian action of $T$ on $\BC^N$, and hence induces 
an action on  $T^\ast \BC^N$. Given $\alpha \in \ft^\ast$, the \emph{toric variety} 
associated to this data is the K\"ahler quotient
\begin{equation}
  \FX_\alpha := \BC^N \reda{\alpha} T,
\end{equation}
and the \emph{hypertoric variety} associated to this data is the hyperk\"ahler quotient
\begin{equation}
  \FM_\alpha := T^\ast \BC^N \rreda{(\alpha,0)} T.
\end{equation}
Note that $\FX_\alpha$ has a residual Hamiltonian action of $K$, and that $\FM_\alpha$
has a residual hyperhamiltonian action of $K$.

It is convenient to organize the data determining $\FX_\alpha$ and $\FM_\alpha$ into
a \emph{hyperplane arrangement} as follows.
Let $\{e_j\}$ be the standard basis of $\BR^N$. Then we obtain a collection
$\{u_j\}$ of \emph{weights} defined by $u_j := i^\ast(e_j) \in \ftd$, as well as 
a collection of \emph{normals} $\{n_i\}$ defined by $n_i = \pi(e_i) \in \fk$.
Note that we allow repetitions, i.e. $u_i$ and $u_j$ are considered to be
distinct elements of $A$ for $i \neq j$ even if $u_i = u_j$ as elements of 
$\ftd$, and similarly for the normals.
Pick some $d \in \BR^N$ such that $i^\ast(d) = \alpha$.
Then we can define affine hyperplanes $H_i$ by
\begin{equation}
H_i = \{ x \in \fkd \suchthat \ev{n_i, x} - d_i = 0 \},
\end{equation}
as well as half-spaces
\begin{equation}
H_i^\pm = \{ x \in \fkd \suchthat \pm(\ev{n_i,x} - d_i) \geq 0 \}.
\end{equation}
This arrangement of hyperplanes will be denoted by $\CA$. An example is pictured 
in Figure \ref{fig-hyperplane-example} below.

\begin{figure}[h]
\centering
\begin{tikzpicture}



  \coordinate (L1a) at ( 0.0, -0.5) {};
  \coordinate (L1b) at ( 0.0,  3.0) {};

  \coordinate (L2a) at (-0.5,  0.0) {};
  \coordinate (L2b) at ( 3.0,  0.0) {};

  \coordinate (L3a) at (-0.5,  1.5) {};
  \coordinate (L3b) at ( 3.0,  1.5) {};

  \coordinate (L4a) at (-0.5,  3.0) {};
  \coordinate (L4b) at ( 3.0, -0.5) {};

  \coordinate (V1) at (intersection of L1a--L1b and L2a--L2b);
  \coordinate (V2) at (intersection of L4a--L4b and L2a--L2b);
  \coordinate (V3) at (intersection of L3a--L3b and L4a--L4b);
  \coordinate (V4) at (intersection of L1a--L1b and L3a--L3b);

  \draw[fill,color=gray] (V1)--(V2)--(V3)--(V4)--cycle;

  \draw[thick] (L1a)--(L1b);
  \draw[thick] (L2a)--(L2b);
  \draw[thick] (L3a)--(L3b);
  \draw[thick] (L4a)--(L4b);



\end{tikzpicture}
\caption{The hyperplane arrangement corresponding to the hyperk\"ahler analogue of 
$\widetilde{\BP^2}$, the blow-up of $\BP^2$ at a point. The lines are the hyperplanes 
$H_i$ and the shaded area is the common intersection $\cap H_i^+$, which is the
moment polytope of $\widetilde{\BP^2}$.}
\label{fig-hyperplane-example}
\end{figure}
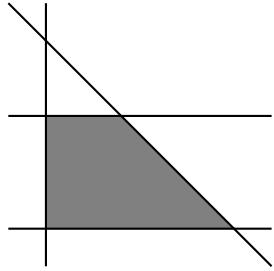


It is well-known that the geometry and topology of toric varieties is
controlled by their moment polytopes, or dually by their normal fans \cite{Fulton}. Hyperplane arrangements
play an analogous role in hypertoric geometry. In fact, most of the geometry and 
topology can be recovered from the \emph{matroid} underlying the hyperplane arrangement.
We will show that Morse theory on hypertoric varieties is more naturally
understood in terms of the \emph{weight matroid} of the arrangement, which is dual
to the matroid of the hyperplane arrangement.
\begin{definition} \label{def-weight-matroid}
Let $\CB = \{u_1, \dots, u_N\}$  be the collection of weights of $T$ acting on $\BC^N$.
For any subset $\CJ \subset \CB$, the \emph{rank} of $\CJ$, denoted $\rk(\CJ)$, is defined to 
be equal to the dimension of the span of $\CJ$. A subset $\CJ$ is called a \emph{flat} if 
for any $u_k \in \CB \setminus \CJ$, $\rk(\CJ \cup \{u_k\}) > \rk(\CJ)$. A flat is called a 
\emph{hyperplane} if  $\rk(\CJ) = \dim(T)-1$. The \emph{weight matroid} is the set of all
flats.
\end{definition}

\begin{remark} The relationship between hypertoric geometry and matroid theory
was developed in \cite{HauselSturmfels}. The weight matroid of $\CB$ and the hyperplane
arrangement $\CA$ are related by \emph{Gale duality}.
\end{remark}

\begin{remark} In our notation, the collection of weights $\CB$ will be the set
$\{u_1, \dots, u_N\}$, where $u_j = i^\ast(e_j)$ as above.
\end{remark}

\subsection{Analysis of the Critical Sets}  
We now consider a quotient of the form $\FM(\alpha_\BR, \alpha_\BC)$ with
$\alpha_\BC$ a regular value of $\muc$. Note that this includes quotients
of the form $\FM(\alpha_\BR, 0)$ as a special case, since we can
always rotate the hyperk\"ahler frame. We identify $T^\ast \BC^N$ with 
$\BC^N \times \BC^N$ and use coordinates $(x, y)$. Shifting $\muc$ by $\alpha_\BC$, 
we can take it to be
\begin{equation} \label{MomentMapFormula} 
  \muc(x,y) = \sum_{i=1}^N x_i y_i u_i - \alpha_\BC,
\end{equation}
and we will consider Morse theory with the function $f = |\muc|^2$.
If $\CJ \subseteq \CB$ is a flat, we define a subspace $V_\CJ \subset \BC^N$ by
\begin{equation} \label{DefnVJ}
V_\CJ := \bigoplus_{u_j \in \CJ} V(u_j).
\end{equation}
and a torus $T_\CJ$ by $T_\CJ = T / N_\CJ$, where $N_\CJ$ is the kernel of the action
map on $V_\CJ$. We define $\FM_\CJ$ to be the hyperk\"ahler quotient 
\begin{equation} \label{DefnMJ}
\FM_\CJ := T^\ast V_\CJ \rreda{\alpha_\CJ} T_\CJ,
\end{equation}
where $\alpha_\CJ$ is an induced moment map level which we define below.\footnote{$\alpha$
can always be chosen so that the moment map levels $\alpha_\CJ$ which appear for flats $\CJ$
are all regular.}

The inclusion $i:T \into (S^1)^N$ induces a surjective map of rings
$i^\ast: H_{(S^1)^N}^\ast \to H_T^\ast$, so that 
$H_T^\ast \iso \BQ[u_1, \ldots, u_N] / \ker i^\ast$ as a graded ring.
By abuse of notation we will write $u_j$ to denote
its image in $H_T^\ast$.\footnote{Note that we also use the symbol $u_j$ to
denote the vectors $i^\ast(e_j)$, but no confusion should arise as it
should be clear from context which of the two meanings is intended.}
If $\mathcal{J} \subseteq \CB$ is a flat, then we define a class $u_\CJ \in H_T^\ast$ by
\begin{equation} \label{EulerClassJ}
u_\CJ := \prod_{\beta \in \CJ^c} \beta,
\end{equation}
and note that the product is taken over the \emph{complement} of $\CJ$.

\begin{proposition} 
\label{CriticalSets} For a generic parameter $\beta$,
the critical set of $f$ is the disjoint union of sets $C_\CJ$,
where the union runs over all flats $\CJ$ of the weight matroid,
where the sets $C_\CJ$ are defined by
\begin{equation}
C_\CJ = \left( \bigcap_{\beta \in \CJ} \{ \beta \cdot \muc = 0\} \right) 
\cap \left( \bigcap_{\beta \not\in \CJ} \{(x_\beta, y_\beta) = 0\} \right).
\end{equation}
The Morse index of $C_\CJ$ is given by $\lambda_\CJ = 2(N - \# \CJ)$.
Up to a nonzero constant, the $T$-equivariant Euler class of the negative
normal bundle to $C_\CJ$ is given by the restriction of the class
$u_\CJ$ to $H_T^\ast(C_\CJ)$, where $u_\CJ$ is defined by (\ref{EulerClassJ}).
The $T$-equivariant Poincar\'e series of $C_\CJ$ is equal to
$(1-t^2)^{-r} P_t(\FM_\CJ)$, where $r$ is the codimension of $T_\CJ$ in $T$ and
$\FM_\CJ$ is the quotient defined by (\ref{DefnMJ}). 
\end{proposition}
\begin{proof} Using equation (\ref{MomentMapFormula}), we see that
\begin{equation} |\nabla f(x,y)|^2 = 
 4 \sum_{i=1}^N \left(|x_i|^2 + |y_i|^2 \right) |u_i \cdot \muc(x,y)|^2. \end{equation}
Since this is a sum of non-negative terms, if $\nabla f(x,y) = 0$, each term in 
the sum must be 0. Thus for each weight $u_i$, we must have either that $(x_i, y_i) = 0$ or
that $u_i \cdot \muc = 0$. Let us fix some particular critical point $(x_c,y_c) \in T^\ast \BC^N$, and 
let $\CJ$ be the set of weights for which $u_i \cdot \muc(x_c,y_c) = 0$.
By construction $\CJ$ is a flat and $(x_c, y_c) \in C_\CJ$. 
Hence every critical point is contained in $C_\CJ$ for some flat $\CJ$.
Conversely, $\nabla f = 0$ on $C_\CJ$ by construction, so we see that
$\Crit f = \cup_\CJ C_\CJ$, where the union runs over the flats $\CJ \subseteq \CB$.

To see that the union is disjoint, write $\muc = \mu_\CJ + \mu_{\CJ^c}$, where
\begin{align}
\label{DefnMuJ}
\mu_\CJ(x,y) &= \sum_{u_i \in \CJ} u_i x_i y_i - (\alpha_\BC)_\CJ, \\
\label{DefnMuJc}
\mu_{\CJ^c}(x,y) &= \sum_{u_i \not\in J} u_i x_\beta y_\beta - (\alpha_\BC)_\CJ^\perp,
\end{align}
$(\alpha_\BC)_\CJ$ is the projection of $\alpha_\BC$ to $\ft_J^\ast$, and 
$(\alpha_\BC)_\CJ^\perp = \alpha_\BC - (\alpha_\BC)_\CJ$. Then at $(x_c, y_c)$, we have
$\mu_\CJ(x_c, y_c) = 0$ and $\mu_{\CJ^c}(x_c, y_c) = -(\alpha_\BC)_\CJ^\perp$.
Thus on $C_\CJ$, $\muc$ takes the value $-(\alpha_\BC)_\CJ^\perp$.
For generic $\alpha_\BC$, we have $(\alpha_\BC)_\CJ^\perp \neq (\alpha_\BC)_{\CJ'}^\perp$
for $\CJ \neq \CJ'$, hence $C_\CJ \cap C_{\CJ'} = \emptyset$ for $\CJ \neq \CJ'$.

To determine the Morse index of $C_\CJ$, we compute
\begin{equation} |\muc(x,y)|^2 = |\mu_\CJ(x,y)|^2 + |\mu_{\CJ^c}(x,y)|^2 
 + 2 \mathrm{Re} \ev{\mu_\CJ(x,y),\mu_{\CJ^c}(x,y)}. \end{equation}
The term $|\mu_\CJ(x,y)|^2$ has an absolute minimum at $(x_c, y_c)$,
and so does not contribute to the Morse index. Since $(\alpha_\BC)_\CJ^\perp$ is
orthogonal to $\mu_\CJ(x,y)$ for all $(x,y)$, the third term can be
rewritten as
\begin{equation} 2 \mathrm{Re} \ev{\mu_\CJ(x,y),\mu_{\CJ^c} + (\alpha_\BC)_\CJ^\perp}. \end{equation}
Looking at the expressions
(\ref{DefnMuJ}) and (\ref{DefnMuJc}) for $\mu_\CJ$ and $\mu_{\CJ^c}$, we see that
at $(x_c, y_c)$, $\mu_\CJ$ vanishes to first order, whereas 
$\mu_{\CJ^c} + (\alpha_\BC)_\CJ^\perp$
vanishes to second order. Hence the inner product of these terms vanishes to 
third order and does not affect the Morse index. Thus the Morse index is
determined solely by the second term, which is
\begin{equation} \label{RelevantTerm}
|\mu_{\CJ^c}(x,y)|^2 
 = |(\alpha_\BC)_\CJ^\perp|^2 - 2 \mathrm{Re} 
 \sum_{i \in \CJ^c} \ev{(\alpha_\BC)_\CJ^\perp,u_i} x_i y_i + \textrm{fourth order}.
\end{equation}
For generic $\alpha_\BC$, we have $\ev{(\alpha_\BC)_\CJ^\perp,u_i} \neq 0$ for all
$u_i \in \CJ^c$, and since each term in the sum is the real part of the 
holomorphic function $\ev{(\alpha_\BC)_\CJ^\perp,u_i} x_i y_i$ it must contribute
$2$ to the Morse index. Hence the Morse index is 
$\lambda_\CJ := 2\# \CJ^c = 2(N - \#\CJ)$. Since the $j$th factor of $(S^1)^N$
acts on $(x_j, y_j)$ with weight $(1, -1)$, this also shows that the
equivariant Euler class is given by a nonzero multiple of $u_\CJ$
(defined by (\ref{EulerClassJ})), as claimed.

Finally, we compute the equivariant Poincar\'e series of $C_\CJ$.
Let $V_\CJ \in \BC^N$ be defined as above, and let $N_\CJ \subset T$ be the
subtorus that acts trivially on $V_\CJ$. Then we have an isomorphism
$T \iso T_\CJ \times N_\CJ$. Let $r$ be the dimension of $N_\CJ$, which is the
codimension of $T_\CJ$ in $T$. The moment map for the action of $T_\CJ$ on
$T^\ast V_\CJ$ is given by the restriction of $\mu_\CJ$ (as defined by
equation (\ref{DefnMuJ})) to $T^\ast V_J$. Hence $C_\CJ = \mu_\CJ^{-1}(0) \cap T^\ast V_\CJ$, and
\begin{equation} P^T_t(C_\CJ) = P^{T_\CJ \times N_\CJ}_t(C_\CJ) 
= (1-t^2)^{-r} P^{T_\CJ}_t(C_\CJ) = (1-t^2)^{-r} P_t(\FM_\CJ). \end{equation}
\end{proof} 

\begin{remark} Note that the critical sets $C_\CJ$ are all nonempty, and
that the Morse indices do not depend on $\alpha_\BC$ (as long as it is generic).
This is due to the fact that $\muc$ is holomorphic. In the real case,
i.e. $|\mur - \alpha_\BR|^2$, the critical sets and Morse indices have a much
more sensitive dependence on the level $\alpha_\BR$.
\end{remark}

By Theorem \ref{SurjectivityCriterion} and Corollary \ref{SurjectivityForTori},
the hyperk\"ahler Kirwan map 
$\kappa: H_T^\ast \to H^\ast(\FM)$
is surjective, and its kernel is the ideal generated by the equivariant
Euler classes of the negative normal bundles to the components of the critical
set. Since we described these explicitly in Proposition \ref{CriticalSets},
we immediately obtain the following description of $H^\ast(\FM)$.
\begin{theorem} \label{HypertoricCohomologyRing}
Assume $\alpha_\BC$ is generic. Then the cohomology ring $H^\ast(\FM)$ is isomorphic to $H^\ast(BT) / \ker \kappa$,
where $\kappa$ is the hyperk\"ahler Kirwan map. Its kernel is the ideal generated by the classes $u_\CJ$,
for every proper flat $\CJ \subset \CB$.
\end{theorem}

\begin{remark} The cohomology ring $H^\ast(\FM)$ was first computed by Konno
\cite[Theorem 3.1]{KonnoToric}. The relations defining $\ker \kappa$
obtained by Konno are not identical to those in Theorem 
\ref{HypertoricCohomologyRing}, 
but it is not difficult to see that they are equivalent.
It was pointed out to us by Proudfoot that this equivalence is a 
special case of Gale duality \cite{OrientedMatroids}.  
\end{remark}
 
\begin{remark} Under the assumption that $\FM$ is smooth (and not just an
orbifold), the same result holds with $\BZ$ coefficients \cite{KonnoToric}.
In principle, it should be possible to extend our Morse-theoretic arguments
to obtain this stronger result using the ideas and results of
\cite[\S 7]{TolmanWeitsman}.
\end{remark}

Quotients of the form $\FM(\alpha_\BR, 0)$ inherit an additional $S^1$-action
induced by the $S^1$ action on $T^\ast \BC^N$ given by $t \cdot(x, y) = (x, ty)$.
This action preserves the K\"ahler structure and rotates the holomorphic 
symplectic form (i.e. $t^\ast \omegac = t \omegac$). 
Let us fix some particular $\alpha_\BR$ and denote $\FM := \FM(\alpha_\BR, 0)$.
We would like to understand the $S^1$-equivariant 
cohomology $H_{S^1}^\ast(\FM)$ (which, unlike the ordinary cohomology of 
$\FM$, \emph{does} depend on the choice of $\alpha_\BR$). 
To compute the $S^1$-equivariant cohomology, it is more convenient 
to work directly with $|\muhk|^2 = |\mur|^2 + |\muc|^2$, where
\begin{align}
\label{MuRFormula} \mur &= \sum_i u_i \left(|x_i|^2 - |y_i|^2 \right) - \alpha_\BR, \\
\label{MuCFormula} \muc &= \sum_i u_i x_i y_i.
\end{align}
By Theorems \ref{SurjectivityCriterion} and \ref{GlobalEstimate} $|\muhk|^2$ 
is minimally degenerate and flow-closed, and since it is also $S^1$-invariant 
we can consider
the $T \times S^1$-equivariant Thom-Gysin sequence. The usual arguments
of the Kirwan method extend to the $S^1$-equivariant setting, so
we obtain surjectivity of map $\kappa_{S^1}: H_{T \times S^1}^\ast \to H_{S^1}^\ast(\FM)$,
and its kernel is generated by the $T \times S^1$-equivariant Euler classes
of the negative normal bundles to the components of the critical set.

To find the components of the critical set of $|\muhk|^2$ and to 
compute the equivariant
Euler classes, we can repeat the arguments of Proposition \ref{CriticalSets} 
almost without modification. The components of the critical set 
are again indexed by flats $\CJ$. The only important difference
is that since we now work with 
$T \times S^1$-equivariant cohomology, we have to be more careful 
in computing the equivariant Euler classes. Let us make the identification
$H_{T \times S^1}^\ast \iso H_T^\ast[u_0]$.
When we expand $|\muhk|^2$ about a critical point as in 
the proof of Proposition \ref{CriticalSets}, the 
relevant term is now (cf. equation (\ref{RelevantTerm}))
\begin{equation}
  - 2 \sum_{u_i \in \CJ^c} \ev{(\alpha_\BR)_\CJ^\perp,u_i} \left(|x_i|^2 - |y_i|^2 \right), 
\end{equation}
where $(\alpha_\BR)_\CJ^\perp$ is defined in a manner analogous to $(\alpha_\BC)_\CJ^\perp$ as 
in the proof of Proposition \ref{CriticalSets}.
We see that the $x_i$ term appears with an overall negative sign
if $\ev{(\alpha_\BR)_\CJ^\perp,u_i} > 0$; otherwise it is the $y_i$ term
that appears with a negative sign. Since $S^1$ acts on $x$ with weight
$0$ and acts on $y$ with weight $1$, and since $T$ acts on $x$ and $y$ with
oppositely signed weights, we find that the equivariant Euler class
is given (up to an overall constant) by
\begin{equation}
\tilde{u}_\CJ := \left(\prod_{u_i \in \CJ^+} u_i \right) \left(\prod_{u_i \in \CJ^-} (u_0 - u_i) \right),
\end{equation}
where 
\begin{equation}
\CJ^\pm := \{ u_i \in \CJ^c \suchthat \pm \ev{(\alpha_\BR)_\CJ^\perp,u_i} > 0 \}.
\end{equation}
We thus obtain the following.
\begin{theorem} The $S^1$-equivariant cohomology $H_{S^1}^\ast(\FM)$ is
isomorphic to 
\begin{equation}
 H_{T \times S^1}^\ast / \ker\kappa_{S^1},
\end{equation}
where $\kappa_{S^1}$ is the $S^1$-equivariant Kirwan map. Its kernel
is the ideal generated by the classes $\tilde{u}_\CJ$, for every proper flat $\CJ$. 
\end{theorem}
 
\begin{remark} The $S^1$-equivariant cohomology rings were first
computed by Harada and Proudfoot 
\cite{HaradaProudfoot}. Note that unlike the ordinary cohomology ring,
the $S^1$-equivariant cohomology ring depends explicitly on the 
parameter $\alpha$.
\end{remark}

\subsection{Hyperk\"ahler Modifications}
Recall from Section \ref{sec-cuts-modification} that there is a natural operation called
\emph{hyperk\"ahler modification}, which is the hyperk\"ahler analogue of symplectic cutting.
If $\FM = T^\ast \BC^N \rred T$ is a hypertoric variety with a residual torus action by $K$, 
then for any choice of  $S^1 \into K$ we can consider the modification
\begin{equation} \label{DefnMTilde}
  \tilde{\FM} := \FM \times T^\ast \BC \rred S^1 \iso T^\ast \BC^{N+1} \rred \tilde{T},
\end{equation}
where we have set $\tilde{T} = T \times S^1$. We also consider the quotient
\begin{equation} \label{DefnMHat}
  \hat{\FM} := \FM \rred S^1 \iso T^\ast \BC^N \rred \tilde{T},
\end{equation}
which are both hypertoric varieties.
As we saw in the previous section, for any hypertoric variety, equivariant Morse theory 
with $|\muc|^2$ is controlled entirely by the underlying matroid. Hence if we can deduce a
relation between the matroids of $\FM, \tilde{\FM}$, and $\hat{\FM}$, we will be able to
determine a relation between their Poincar\'e polynomials and cohomology rings.

Let $\CB$, $\tilde{\CB}$, and $\hat{\CB}$ denote the respective collections of weights defining
the hypertoric varieties $\FM, \tilde{\FM}$, and $\hat{\FM}$, and let
$\mu, \tilde{\mu}$, and $\hat{\mu}$ denote the respective (complex) moment maps.
We can relate collections of weights $\tilde{\CB}$ and $\hat{\CB}$, corresponding
to the modification and quotient of $M$, to the weights $\CB$ as follows.
\begin{lemma} \label{Modification}
Let $\tilde{\CB} = \{\tilde{u}_j\}_{j=1}^{N+1}$. 
Then $\hat{\CB} = \{\tilde{u}_j\}_{j=1}^{N}$ and $\CB = \{u_j\}_{j=1}^N$, where
$u_j$ is the image of $\tilde{u}_j$ after quotienting by $\Span\{\tilde{u}_{N+1}\}$.
\end{lemma}
\begin{proof} The weights are determined by the embeddings
$\ft \into \BR^N$, $\tilde{\ft} \into \BR^{N+1}$, and $\hat{\ft} \into \BR^N$.
Note that $\tilde{\ft} \iso \ft \oplus \BR$.
If we pick a basis of $\ft$ (and use the standard basis of $\BR^N$),
we can represent the embedding $\ft \into \BR^N$ by some matrix $B$.
The $S^1$ action on $M$ is determined by
specifying its weights on $\BC^N$; this is equivalent to adjoining a column
to $B$. This gives the matrix $\hat{B}$ determining $\tilde{\ft} \into \BR^N$.
Finally, to obtain the modification $\tilde{M}$, we let the $S^1$ act
on an additional copy of $\BC$ with weight $-1$. This amounts to
adjoining the row $(0, \ldots, 0, -1)$ to $\hat{B}$, to obtain
$\tilde{B}$ determining $\tilde{\ft} \into \BR^{N+1}$.
Since the weights $u_j$, $\tilde{u}_j$, and $\hat{u}_j$ correspond to
the rows of $B$, $\tilde{B}$, and $\hat{B}$, respectively,
the result follows.
\end{proof}


By the above lemma, we have $\CB, \hat{\CB} \iso \{1, \dots, N\}$
and $\tilde{\CB} \iso \{1, \dots, N+1\}$ as sets.
Hence any subset $\CJ \subset \CB$ may be regarded as a subset of $\hat{\CB}$ as
well as of $\tilde{\CB}$. We will use this identification for the remainder of this section.
\begin{lemma} \label{CriticalRelation}
Let $\CJ \subseteq \CB$. Then
\begin{enumerate}
  \item $\CJ$ is a flat of $\CB$ if and only if $\CJ \cup\{N+1\}$ is a flat of $\tilde{\CB}$.
  \item $\CJ$ is a flat of $\hat{\CB}$ if $\CJ$ is a flat of $\CB$.
  \item $\CJ$ is a flat of $\hat{\CB}$ if and only if at least one of $\CJ$ or $\CJ \cup \{N+1\}$ is a flat of $\tilde{\CB}$.
\end{enumerate}
\end{lemma}
\begin{proof} 
Let $\tilde{\CB} = \{\tilde{u}_j\}_{j=1}^{N+1}$. 
By the previous lemma, 
$\hat{\CB} = \{\tilde{u}_j\}_{j=1}^N$ and $\CB = \{u_j\}_{j=1}^N$,
where $u_j$ is the image of $\tilde{u}_j$ after quotienting by $\Span\{u_{N+1}\}$.

To prove (1), first suppose that $\CJ$ is a flat of $\CB$.
Suppose that there is some $\tilde{u}_i \in \tilde{\ft}_{\CJ\cup\{N+1\}}$.
If $i = N+1$ then $i \in \CJ \cup\{N+1\}$ and there is nothing to check,
so suppose that $i \neq N+1$. Applying the quotient map, we see that
$u_i \in \ft_\CJ$ (since $\tilde{u}_{N+1}$ goes to 0), and since $\CJ$ was a flat
of $\CB$ we see that $i \in \CJ \subset \CJ \cup \{N+1\}$. Hence
$\CJ \cup \{N+1\}$ is a flat of $\tilde{\CB}$. 

Conversely, suppose that $\CJ \cup \{N+1\}$ is a flat of $\tilde{\CB}$.
Suppose that there is some $u_i \in \ft_\CJ$. Then if
$u_i = \sum_{j \in \CJ} a_j u_j$, we see that
$\tilde{u}_i - \sum_{j \in \CJ} a_j \tilde{u}_j$ is in the kernel of the projection,
and so is some multiple of $\tilde{u}_{N+1}$. Hence 
$\tilde{u}_i \in \tilde{\ft}_{\CJ \cup \{N+1\}}$. Since $\CJ \cup \{N+1\}$ is a flat
of $\tilde{\CB}$, we must have $i \in \CJ \cup \{N+1\}$.
But $i \neq N+1$ by assumption, so we have $i \in \CJ$.

To prove (2), suppose that $\CJ$ is a flat of $\CB$.
If $\tilde{u}_i \in \tilde{\ft}_\CJ$, then applying the quotient map we find
$u_i \in \ft_\CJ$. Hence $i \in \CJ$.

To prove (3), first suppose that $\CJ \cup \{N+1\}$ is a flat of
$\tilde{\CB}$. Then by (1), $\CJ$ is a flat of $\CB$, and by (2) $\CJ$ is a flat of
for $\hat{\CB}$. On the other hand, if $\CJ$ is a flat of
$\tilde{\CB}$, then it is certainly a flat of $\hat{\CB}$. This establishes one direction.

Conversely, suppose that $\CJ$ is a flat of $\hat{\CB}$.
If $\tilde{u}_{N+1} \not\in \tilde{\ft}_\CJ$ then $\CJ$ is a flat of
$\tilde{\CB}$; otherwise $\tilde{u}_{N+1} \in \tilde{\ft}_{\CJ}$ and thus
$\CJ \cup \{N+1\}$ is a flat of $\tilde{\CB}$.
\end{proof}  

We rephrase the preceding lemma as the following trichotomy:
\begin{lemma} \label{Trichotomy}
Let $\CJ \subseteq \{1, \cdots, N\}$ be a flat of $\hat{\CB}$.
Then exactly one of the following cases occurs:
\begin{enumerate}
\item $\CJ$ is a flat of $\tilde{\CB}$ and $\CJ$ is not a flat of $\CB$.
\item $\CJ$ is a flat of $\CB$ and both $\CJ$ and $\CJ \cup \{N+1\}$ are flats of $\tilde{\CB}$.
\item $\CJ$ is a flat of $\CB$ and $\CJ \cup\{N+1\}$ is a flat of $\tilde{\CB}$, while $\CJ$ is not.
\end{enumerate}
Moreover, every critical subset for $\CB$ and $\tilde{\CB}$ occurs as exactly one of the above.
\end{lemma} 

\begin{theorem} If $\tilde{\FM}$ is a hyperk\"ahler modification of
$\FM$ and $\hat{\FM}$ is the corresponding quotient, then
\begin{equation} \label{PoincareRelation}
P_t(\tilde{\FM}) = P_t(\FM) + t^2 P_t(\hat{\FM}).
\end{equation}
\end{theorem}
\begin{proof} We will prove this by induction on
 $N = \# \CB$, the number of the number of weights (equivalently, the number of
hyperplanes in $\mathcal{A}$). The base case can be verified
easily, so we assume that the result is true for modifications
$(\tilde{\FM}', \FM', \hat{\FM}')$, where $\FM'$ is a quotient of $T^\ast \BC^{N'}$,
with $N' < N$.

By Theorems \ref{SurjectivityCriterion} and
\ref{GlobalEstimate}, the functions $f = |\muc|^2$, 
$\tilde{f} = |\tilde{\mu}_\BC|^2$, and $\hat{f} = |\hat{\mu}_\BC|^2$ 
are equivariantly perfect.
$\FM$ is a quotient of $T^\ast \BC^N$ by a torus of rank $d$,
while $\tilde{M}$ and $\hat{\FM}$ are quotients of $T^\ast \BC^{N+1}$
and $T^\ast \BC^N$, respectively, by a torus of rank $d+1$.
Hence we have
\begin{align}   
\frac{1}{(1-t^2)^d} 
&= \sum_C t^{\lambda_C} P^T_t(C), \\
\frac{1}{(1-t^2)^{d+1}} 
&= \sum_{\tilde{C}} t^{\tilde{\lambda}_C} P^{\tilde{T}}_t(\tilde{C}), \\
\frac{1}{(1-t^2)^{d+1}} 
&= \sum_{\hat{C}} t^{\hat{\lambda}_C} P^{\hat{T}}_t(\hat{C}).
\end{align}
Since $0$ is the absolute minimum in each case, we obtain
\begin{equation} \frac{1}{(1-t^2)^d} = P_t(\FM) + \sum_{C, f(C) > 0} t^{\lambda_C} P^T_t(C), \end{equation}
and similarly for $\tilde{\FM}$ and $\hat{\FM}$. Since 
\begin{equation} \frac{1}{(1-t^2)^{d+1}} = \frac{1}{(1-t^2)^d} + \frac{t^2}{(1-t^2)^{d+1}} = 0, \end{equation}
we just have to show that
\begin{equation} \sum_{\tilde{C}, \tilde{f}(\tilde{C}) > 0} t^{\tilde{\lambda}_C} P^{\tilde{T}}_t(\tilde{C})
= \sum_{C, f(C) > 0} t^{\lambda_C} P^T_t(C)
 +  \sum_{\hat{C}, \hat{f}(\hat{C}) > 0} t^{\hat{\lambda}_C+2} P^{\hat{T}}_t(\hat{C}). \end{equation}
to obtain the desired recurrence relation among the Poincar\'e polynomials
of $\FM$, $\tilde{\FM}$, and $\hat{\FM}$. By Lemma \ref{Trichotomy}, there
is a trichotomy relating the critical sets of $\tilde{f}$ to the critical
sets of $f$ and $\hat{f}$. We will consider each case separately.

Case (1): We have $\tilde{C} = \tilde{C}_\CJ$ where $\CJ$ is a flat of
$\tilde{\CB}$ and $\hat{\CB}$ but not of $\CB$. From Proposition \ref{CriticalSets}, we have 
$\tilde{\lambda}_\CJ = \hat{\lambda}_\CJ + 2$ and $\tilde{C}_\CJ = \hat{C}_\CJ \times (0, 0)$. Hence 
$t^{\tilde{\lambda}_\CJ} P^{\tilde{T}}_t(\tilde{C}_\CJ) = t^{\hat{\lambda}_\CJ+2} P^{\hat{T}}_t(\hat{C}_\CJ)$.

Case (2): We have $\tilde{C} = \tilde{C}_\CJ$ where $\CJ$ is a flat of $\hat{\CB}$ and 
$\CJ \cup \{N+1\}$ is a flat of $\tilde{\CB}$. Then $\tilde{\lambda}_\CJ = \hat{\lambda}_\CJ + 2$
and $\tilde{\lambda}_{\CJ\cup\{N+1\}} = \lambda_\CJ$. The terms involving $\tilde{C}_\CJ$ and $\hat{C}_\CJ$ are equal as in case (1).
Since both $\CJ$ and $\CJ \cup \{N+1\}$ are flats of $\tilde{\CB}$, it must
be that $\tilde{u}_{N+1} \not\in \tilde{\ft}_\CJ$. Hence
$\tilde{T}_{\CJ \cup \{N+1\}} \iso \tilde{T}_\CJ \times S^1$, where the 
last factor is generated by $\tilde{u}_{N+1}$, and we find that
$\tilde{\FM}_{\CJ \cup \{N+1\}} \iso \FM_\CJ$. Hence 
$P^{\tilde{T}}_t(\tilde{C}_{\CJ \cup \{N+1\}}) = P^T_t(C_\CJ)$.

Case (3): $\tilde{C} = \tilde{C}_{\CJ \cup \{N+1\}}$, where
$\CJ$ is critical for $\CB$ and $\hat{\CB}$ but not for $\tilde{\CB}$.
By Proposition \ref{CriticalSets}, we have
$P^{\tilde{T}}_t(\tilde{C}_{\CJ \cup\{N+1\}}) = (1-t^2)^{-r} P_t(\tilde{\FM}_{\CJ \cup \{N+1\}})$,
$P^T_t(C_J) = (1-t^2)^{-r} P(\FM_J)$, and 
$P^{\hat{T}}_t(\hat{C}_\CJ) = (1-t^2)^{-r} P_t(\hat{\FM}_\CJ)$, where $r$ is
the codimension of $T_\CJ$ in $T$. Thus we have to show that
\begin{equation} \label{InductionRelation}
P_t(\tilde{\FM}_{\CJ \cup \{N+1\}}) = P_t(\FM_\CJ) + t^2 P_t(\hat{\FM}_\CJ).
\end{equation}
But $\tilde{\FM}_{\CJ \cup \{N+1\}}$ is a modification of $\FM_\CJ$, and $\hat{\FM}_\CJ$
is the corresponding quotient of $\FM_\CJ$. Since 
$\CJ$ is a proper subset of $\{1, \ldots, N\}$, the relation
(\ref{InductionRelation}) is true by induction.
\end{proof}

The relation (\ref{PoincareRelation}) is equivalent to
the following recurrence relation among the Betti numbers
of $\FM$, $\tilde{\FM}$, and $\hat{\FM}$:
\begin{equation}
  \tilde{b}_{2k} = b_{2k} + \hat{b}_{2k} + \hat{b}_{2k-2}.
\end{equation}
If we let $d_k$ denote the number of bounded $k$-dimensional facets of
the polyhedral complex generated by the half-spaces $H_i^\pm$ 
(with $\tilde{d}_k$ and $\hat{d}_k$ defined similarly), 
it is easy to see that these satisfy the same relation:
\begin{equation}
  \tilde{d}_k = d_k + \hat{d}_k + \hat{d}_{k-1}.
\end{equation}
Since any toric hyperk\"ahler orbifold can be constructed out of 
a finite sequence of modifications starting with $T^\ast \BC^n$, an
easy induction argument then yields the following explicit description
of $P_t(\FM)$.
\begin{corollary}[{Bielawski-Dancer \cite[Theorem 7.6]{BielawskiDancer}}]
\label{HypertoricPoincare}
The Poincar\'e polynomial of $\FM$ is given by
\begin{equation} 
  P_t(\FM) = \sum_k d_k (t^2 - 1)^k.
\end{equation}
\end{corollary}



\chapter{Quivers with Relations} \label{ch-quiver}

\section{Moduli of Quiver Representations}

\subsection{Quiver Representations}
In this chapter, we will study varieties associated to quivers. The basic results concerning theory of
quiver representations can be found in \cite{CrawleyBovey}, and the theory of moduli spaces
of representations from the algebro-geometric point of view was developed in \cite{King}.
We will focus on hyperk\"ahler varieties constructed from quivers, studying a mild generalization
of the varieties introduced by Nakajima \cite{Nakajima94}. For the most part, we will not rely
on any deep results from these works, and will instead work directly with the tools developed
in Chapter \ref{ch-morse}.

\begin{definition}
A \emph{quiver} $\CQ = (\CI, \CE)$ is a finite directed graph with vertex
set $\CI$ and edge set $\CE$. We will assume that the vertices of $\CQ$ are labelled by by $\{\bigcirc, \Box\}$
so that $\CI = \CV \sqcup \CW$, where $\CV$ consists of the vertices labelled
``$\bigcirc$'' and $\CW$ consists of the vertices labelled ``$\Box$''. 
A \emph{dimension vector}
$\bd = (d_i)_{i\in\CI}$ is an assignment of a non-negative integer $d_i$ to each $i \in \CI$.
We will often write $\bd = (\bv, \bw)$ so that $d_i = v_i$ for $i \in \CV$ and $d_i = w_i$ for $i \in \CW$.
\end{definition}

Given a quiver with dimension vector $(\CQ, \bd=(\bv,\bw))$ we define vector spaces 
\begin{align}
  V &= \bigoplus_{i \in \CV} V_i, \\
  W &= \bigoplus_{i \in \CW} W_i, \\
  \Rep(\CQ, \bd) &= \bigoplus_{i \to j \in \CE} \Hom(D_i, D_j),
\end{align}
where $V_i = \BC^{\bv_i}$, $W_i = \BC^{\bw_i}$, and $D_i = V_i$ for $i \in \CV$ and $D_i = W_i$ for $i \in \CW$.
The vector space $\Rep(\CQ, \bd)$ is the space of representations of $\CQ$ of dimension $\bd$.
We will often write this as $\Rep(\CQ, \bv, \bw)$ to emphasize the decomposition $\bd = (\bv, \bw)$.
We define groups
\begin{align}
  G_\bv &= \prod_{i \in \CV} U(v_i), \\
  G_\bw &= \prod_{i \in \CW} U(w_i), \\
  G_\bd &= G_\bv \times G_\bw,
\end{align}
which act naturally on $\Rep(\CQ, \bv, \bw)$. We denote by $\mu_\bv, \mu_\bw, \mu_\bd$
the corresponding moment maps.
We also consider the \emph{double} 
$\double{\CQ}$ of a quiver, defined by $\double{\CQ} = (\CV, \CE \sqcup \overline{\CE})$
where $\overline{\CE}$ denotes reversing the orientation of all the edges in $\CE$. 


For a fixed dimension vector $\bd = (\bv, \bw)$ we now define several moduli spaces of 
representations as follows:
\begin{align}
    \FX_\alpha(\bv, \bw) &:= \Rep(\CQ, \bv, \bw) \reda{\alpha} G_\bv, \\
    \double{\FX}_\alpha(\bv, \bw) &= \Rep(\double{\CQ}, \bv, \bw) \reda{\alpha} G_\bv, \\
    \FM_\alpha(\bv, \bw) &:= T^\ast \Rep(\CQ, \bv, \bw) \rreda{\alpha} G_\bv.
\end{align}
We call $\FX_\alpha(\bv,\bw)$ the \emph{K\"ahler quiver variety}, $\double{\FX}_\alpha(\bv,\bw)$
the \emph{doubled quiver variety}, and $\FM_\alpha(\bv,\bw)$ the \emph{hyperk\"ahler quiver variety}.
Note that there are natural inclusions\footnote{It is easy to check that $\double{\FX}$ is
holomorphic Poisson, $\FM \into \double{\FX}$ is a symplectic leaf, and $\FX \into \FM$ is
a Lagrangian subvariety.}
\begin{equation}
  T^\ast \FX_\alpha(\bv,\bw) \subset \FM_\alpha(\bv,\bw) \subset \double{\FX}_\alpha(\bv,\bw)
\end{equation}
 and that all three of these spaces have a residual Hamiltonian action 
of $G_\bw$. Hence we may consider their quotients $\FX_\alpha(\bv,\bw) \red G_\bw$ and
$\FM_\alpha(\bv,\bw) \rred G_\bw$. Of course, these quotients do not depend on the particular
choice of decomposition $\CI = \CV \sqcup \CW$---this is just reduction in stages (cf. Remark \ref{rmk-reduction-in-stages}).
However, we will see in the following sections that we will be able to understand these
quotients by studying the $G_\bw$-equivariant stratification on varieties of the form
$\FX_\alpha(\bv, \bw)$.

\begin{example} The quiver pictured in Figure \ref{fig-adhm-quiver} is called the
ADHM quiver. By the ADHM construction \cite{ADHM}, the hyperk\"ahler quiver variety 
$\FM_\alpha(k,n)$ may be identified with the moduli space of framed $SU(n)$ instantons 
on $\BR^4$ of charge $k$. This construction was later generalized by
Kronheimer and Nakajima \cite{KronheimerNakajima} to construct Yang-Mills instanton
moduli spaces from any affine $ADE$ quiver.
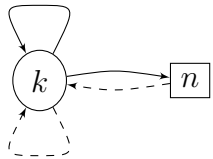
\begin{figure}[h]
\centering
\begin{tikzpicture}
  \node[ellipse,draw] (k) at (0,0) {$k$};
  \node[rectangle,draw] (n) at (2,0) {$n$};
  \coordinate (offsetx) at (0.5,0);
  \coordinate (offsety) at (0,1);  

  \doublearrow{k}{n};

  \coordinate (c1) at ($(k) + (offsety) + (offsetx)$);
  \coordinate (c2) at ($(k) + (offsety)$);
  \coordinate (c3) at ($(k) + (offsety) - (offsetx)$);

  \coordinate (c4) at ($(k) - (offsety) + (offsetx)$);
  \coordinate (c5) at ($(k) - (offsety)$);
  \coordinate (c6) at ($(k) - (offsety) - (offsetx)$);

  \path[->,>=latex',draw,smooth] (k) .. controls (c1) .. (c2) .. controls (c3) .. (k);
  \path[->,>=latex',draw,smooth,dashed] (k) .. controls (c4) .. (c5) .. controls (c6) .. (k);

\end{tikzpicture}
\caption{The ADHM quiver, with doubled edges indicated by dashed arrows.}
\label{fig-adhm-quiver}
\end{figure}
\end{example}

\begin{example} \label{ex-polygon-4}
See the quiver pictured in Figure \ref{fig-simple-quiver}. We have
$\bv = (2)$ and $\bw = (1,1,1,1)$. Then $\FX_\alpha(\bv,\bw) \iso \Gr(2,4)$,
and $\FX_\alpha(\bv,\bw) \red G_\bw \iso \Gr(2,4) \red (S^1)^4$. If we interchange
the labels $\Box \leftrightarrow \bigcirc$, then we have $\FX_\alpha(\bv,\bw) \iso (\BP^1)^4$,
and $\FX_\alpha(\bv,\bw) \red G_\bw \iso (\BP^1)^4 \rred SU(2)$.
This is the most basic example of a polygon space \cite{Klyachko}, and its hyperk\"ahler
analogue is a hyperpolygon space \cite{KonnoPolygon, HaradaProudfoot}.
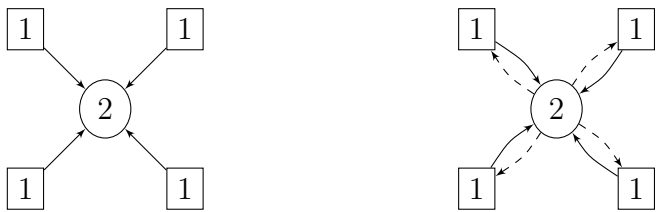
\begin{figure}[h]
\centering
\begin{tikzpicture}
  \node[ellipse,draw] (sink1) at (-3,0) {$2$};
  \node[ellipse,draw] (sink2) at ( 3,0) {$2$};

  \foreach \s in {1,2,3,4} {
    \coordinate (o\s) at (45+90*\s:1.5);
    \node[rectangle,draw] (m\s) at ( $(sink1) + (o\s)$) {$1$};
    \node[rectangle,draw] (n\s) at ( $(sink2) + (o\s)$) {$1$};
    \path[->,>=latex',draw] (m\s) -- (sink1);
    \doublearrow{n\s}{sink2}
  }
\end{tikzpicture}
\caption{Left: a star quiver with dimension vector; right: its double.}
\label{fig-simple-quiver}
\end{figure}
\end{example}

\begin{remark} The definition of quiver variety that we use here is a slightly different
than the usual the notion of a Nakajima quiver variety \cite{Nakajima94}. 
To obtain a Nakajima variety from the construction above, first start with a quiver
$\CQ$ all of whose vertices are labelled by $\bigcirc$. Then construct a new quiver
$\CQ'$ from $\CQ$ by adding a ``shadow vertex'' $i'$, labelled by $\Box$, for
each vertex $i$ of $\CQ$, as well as a single edge from $i'$ to $i$.
Then the hyperk\"ahler quiver variety $\FM_\alpha(\bv,\bw)$ associated to $\CQ'$ by the above construction
coincides with the Nakajima variety associated to $\CQ$. We illustrate this in Figure
\ref{fig-nakajima-shadow}.
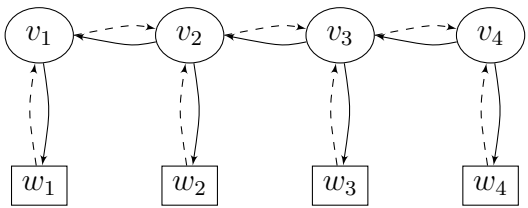
\begin{figure}[h]
\centering
\begin{tikzpicture}
  \coordinate (dx) at (2,0);
  \coordinate (dy) at (0,2);
  \coordinate (shift) at (5,0);

  \foreach \s in {1,2,3,4} {

    \node[ellipse,draw] (v\s) at ( $(0,0) + \s*(dx)$ ) {$v_\s$};
    \node[rectangle,draw] (w\s) at ( $(0,0) + \s*(dx)-(dy)$ ) {$w_\s$};
    \doublearrow{v\s}{w\s};

    \ifnum \s=1
    \else
     \pgfmathsetmacro{\t}{\s-1};
     \doublearrow{v\s}{v\t};
    \fi
  }
\end{tikzpicture}
\caption{A Nakajima quiver constructed by adding ``shadow'' vertices.}
\label{fig-nakajima-shadow}
\end{figure}
\end{remark}

\subsection{Relations in the Path Algebra}

We wish to consider certain natural subvarieties of quiver varieties, constructed in
the following way. Suppose that $E$ is a linear representation of $G_\bd$, and 
$f: \Rep(\CQ, \bd) \to E$ is an equivariant algebraic map. Then 
$f^{-1}(0) \subset \Rep(\CQ, \bd)$ is $G_\bd$-invariant, hence it defines a subvariety 
\begin{equation}
V(f) \subset \FX_\alpha(\bv, \bw).
\end{equation} By equivariance, this subvariety is preserved by the
$G_\bw$-action, so we also obtain a subvariety
\begin{equation}
V(f) \red G_\bw \subset \FX_\alpha(\bv,\bw) \red G_\bw.
\end{equation}
Understanding compact subvarieties of this form will be the primary focus of this chapter.
To understand the relevance of these subvarieties, let us recall the path algebra of a quiver.
\begin{definition} Let $\CQ$ be a quiver. The \emph{path algebra} $\BC[\CQ]$ of $\CQ$ is the
unital associative algebra generated by directed paths in $\CQ$. The product of paths
$p_1$ and $p_2$ is defined to be the composition of $p_1$ and $p_2$ if they are composable,
and zero otherwise.
\end{definition}

\begin{construction} Let $f \in \BC[\CQ]$. Then for any dimension vector $\bd$, $f$
induces a $G_\bd$-equivariant map $\Rep(\CQ, \bd) \to E_\bd$, where $E_\bd$ is a linear
representation of $G_\bd$. 
We can write $f$ (uniquely) as
\begin{equation}
  f = \sum_{(i,j) \in \CI \times \CI} f_{ij}
\end{equation}
where each summand $f_{ij}$ is a linear combination of paths from vertex $i$ to vertex $j$.
Hence $f_{ij}$ induces a map $\Rep(\CQ, \bd) \to \Hom(D_i, D_j)$. Hence we get an induced map
(which, by abuse of notation, we denote by $f$)
\begin{equation}
  f: \Rep(\CQ, \bd) \to \bigoplus_{(i,j): f_{ij} \neq 0} \Hom(D_i, D_j).
\end{equation}
It is easy to see that this map is $G_\bd$-equivariant.
\end{construction}

Hence given any polynomial $f$ in the edges, we obtain an equivariant map $f: \Rep(\CQ, \bd) \to E$,
and the subvariety $V(f) \subset \FX_\alpha(\CQ, \bd)$ can be interpreted as the subvariety
of representations which satisfy the relation $f=0$, i.e. we can interpret $V(f)$ as a moduli
space of representations of the quotient path algebra $\BC[\CQ] / \ev{f}$.

\begin{remark} \label{rem-hyperkahler-double}
The complex moment map on $Rep(\double{\CQ}, \bd)$ is a polynomial in the edges, and the hyperk\"ahler
quiver variety is the subvariety of $\double{\FX}_\alpha(\CQ,\bd)$ defined by the corresponding
relation.
\end{remark}

\begin{proposition} Given an equivariant map $f: \Rep(\CQ, \bd) \to E$ as above, there
is a natural $G_\bw$-equivariant bundle $\CE \to \FX_\alpha(\bv,\bw)$ with section $s$
such that $V(f) = s^{-1}(0)$. In particular, $V(f)$ is 
smooth if $s$ is transverse to $0$. Similarly, $V(f) \red G_\bw$ is the vanishing
set of a section of a vector bundle over $\FX_\alpha(\bv,\bw) \red G_\bw$.
\end{proposition}
\begin{proof}
Let $P = \mu_\bv^{-1}(\alpha)$, where $\mu_\bv$ is the moment map for the $G_\bv$ action.
This is the total space of a principal $G_\bv$-bundle over $\FX_\alpha(\bv, \bw)$. 
Hence we may take the associated bundle $\CE = P \times_{G_\bv} E$.
Since $f$ is $G_\bv$-equivariant, it induces a section $s$ of $\CE$. 
\end{proof}

\begin{remark} Given $V(f) \subset \FX$ we can always compactify via symplectic cuts to
obtain $\overline{V(f)} \subset \overline{\FX}$, where $\overline{\FX}$ is projective.
Moreover, $\overline{V(f)}$ may be identified with the vanishing set of a section of
a bundle $E \to \overline{\FX}$. If this bundle happens to be positive, then Sommese's
Theorem \cite[Chapter 7]{Lazarsfeld2004II} implies that the restriction
$H^\ast(\overline{\FX}) \to H^\ast(\overline{V(f)})$ is surjective in degrees less than
$\dim \FX - \rk E$. It may be possible to deduce hyperk\"ahler Kirwan surjectivity by
arguments along these lines. We plan to investigate this in future work.
\end{remark}

\section{Morse Theory on Quiver Varieties}

\subsection{Circle Actions} \label{sec-quiver-morse-circle}
We define a \emph{circle compact variety} to be a K\"ahler variety together with a 
Hamiltonian circle action such that the fixed-point set is compact, and the moment map
is proper and bounded below. Now let $\CQ = (\CI, \CE)$ be a quiver. For each edge
$e \in \CE$, pick a non-negative integer $w_e$. Then we can define an $S^1$-action on
$\Rep(\CQ, \bd)$ by
\begin{equation}
  s \cdot x_e = s^{w_e} x_e.
\end{equation}
This commutes with the $G_\bd$-action, and hence descends to an $S^1$-action on $\FX_\alpha(\bv,\bw)$
(as well as $\FX_\alpha(\bv,\bw) \red G_\bw$). The moment map for this action is given by
\begin{equation}
  \mu_{S^1}(x) = \frac{1}{2} \sum_{e \in \CE} w_e |x_e|^2
\end{equation}
which is non-negative. We have the following lemma.
\begin{lemma} Every quiver variety $\FX_\alpha(\bv,\bw)$ admits an $S^1$-action of the
form $s \cdot [x_e] = [s^{w_e} x_e]$ such that the moment map is proper and bounded below.
\end{lemma}
\begin{proof} By the preceding discussion, we need only show that the weights $w_e$
may always be chosen so that $\mu_{S^1}$ is proper on $\FX_\alpha(\bv,\bw)$.
We can always arrange this by taking $w_e=1$ for every $e \in \CE$, since in this
case $\mu_{S^1}$ is proper on $\Rep(\CQ,\bv,\bw)$, before quotienting, and hence
descends to a proper map on $\FX_\alpha(\bv,\bw)$.
\end{proof}

\begin{remark} \label{rem-different-circle-action}
Note that the action given by taking $w_e=1$ for all $e \in \CE$ is, in some sense, the
worst case scenario. 
It is often the case that we can pick a different $S^1$-action, which
is more convenient for certain purposes. The most important
special case is when $\CQ$ is an acyclic quiver. Then for its double we have
$\Rep(\double{\CQ}, \bv,\bw) \iso T^\ast \Rep(\CQ, \bv, \bw)$ and it suffices to take
the weight one action on the fibers, i.e. $t \cdot (x,y) = (x, ty)$. In worked examples,
we will always make it clear what $S^1$-action we have chosen.
\end{remark}

The only remaining condition we must check for circle compactness of $\FX_\alpha(\bv,\bw)$
is that the fixed-point set is compact. Since $\mu_{S^1}$ is proper, the fixed-point set is a union of compact subvarieties.
What remains to be shown is that this is a finite union, i.e. that there are only finitely many
connected components of the fixed-point set. We will do this by enumerating the connected components
combinatorially. 

Suppose that we have fixed an $S^1$-action on $\Rep(\CQ, \bd)$ as above, and denote this
action by $x_e \mapsto s \cdot x_e$. Now suppose that there is some $S^1$ action on $D = V \oplus W$.
This induces a new $S^1$-action on $\Rep(\CQ, \bd)$ which we denote by $x_e \mapsto s \ast x_e$.

\begin{definition} Given two $S^1$-actions on $\Rep(\CQ, \bd)$, denoted $\cdot$ and $\ast$,
a \emph{compatible representation} is a representation $x \in \Rep(\CQ, \bd)$ such that
for all $s \in S^1$, $s \cdot x = s \ast x$.
\end{definition}

The set of compatible representations may be naturally identified with
a vector subspace isomorphic to $\Rep(\tilde{\CQ}, \tilde{\bd})$, for some
auxiliary quiver $(\tilde{\CQ}, \tilde{\bd})$. We can take the vertex set
$\tilde{I}$ to consist of the direct summands $D_i(\lambda)$ in the weight space
decomposition of $D_i$, for $i \in \CI$, and the dimension of each vertex to be
$\dim D_i(\lambda)$. For each edge $e \in \CE$ mapping $i$ to $j$, we add edges
in $\tilde{\CQ}$ from $D_i(\lambda)$ to $D_i(\lambda')$, for all $\lambda,\lambda'$
such that $\lambda' - \lambda$ is equal to the weight of the edge $e$ under the $\cdot$ action.

\begin{example} Consider the quiver with one vertex and a single edge taking this vertex
to itself. Let the $S^1$-action be given by weight 1 on this edge. Consider the case of
dimension vector $(3)$. In order to produce a compatible representation with at least
one edge, the $S^1$-action on $\BC^3$ must have at least two distinct weights differing by $1$. The inequivalent choices
of weights are then $(0, 0,1)$, $(0,1,1)$, $(0,1,2)$, and $(0,1,a)$ with $a \geq 3$.
These lead to the quivers in the Figure \ref{fig-example-compatible-subrep} below.
\begin{figure}[h]
\centering
\begin{tikzpicture}
  \node[ellipse,draw] (A1) at (0,0) {$2$};
  \node[ellipse,draw] (A2) at (1.5,0) {$1$};
  \path[->,>=latex',draw] (A1) -- (A2);

  \node[ellipse,draw] (B1) at (3,0) {$1$};
  \node[ellipse,draw] (B2) at (4.5,0) {$2$};
  \path[->,>=latex',draw] (B1) -- (B2);

  \node[ellipse,draw] (C1) at (6,0) {$1$};
  \node[ellipse,draw] (C2) at (7.5,0) {$1$};
  \node[ellipse,draw] (C3) at (9,0) {$1$};
  \path[->,>=latex',draw] (C1) -- (C2) -- (C3);

  \node[ellipse,draw] (D1) at (10.5,0) {$1$};
  \node[ellipse,draw] (D2) at (12,0) {$1$};
  \path[->,>=latex',draw] (D1) -- (D2);

\end{tikzpicture}
\caption{Compatible representations of the quiver with one node and a single edge loop, with dimension vector $\bd = (3)$.}
\label{fig-example-compatible-subrep}
\end{figure}
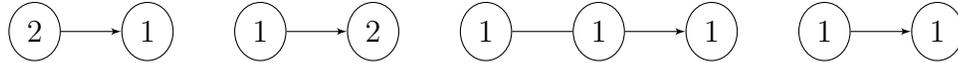
\end{example}

\begin{remark} \label{rem-finite-subrep}
It is clear that there are only finitely many quivers produced by this construction for
any finite quiver $\CQ$ with dimension vector $\bd$.
\end{remark}

\begin{proposition} \label{prop-connected-components}
Let $\CQ$ be a quiver and suppose that we have picked an $S^1$-action on $\Rep(\CQ, \bv, \bw)$
as above, inducing an $S^1$-action on $\FX_\alpha(\bv, \bw)$. Then every connected component
of the fixed-point set consists of equivalence classes of compatible representations,
for some auxiliary $S^1$-action on $V \oplus W$ which acts trivially on $W$.
Consequently, the $S^1$-fixed point set of $\FX_\alpha(\bv,\bw)$ consists of finitely
many connected components.
\end{proposition}
\begin{proof} 
By Proposition \ref{prop-fixed-lift}, on any connected component of 
$\FX_\alpha(\bv,\bw)^{S_1}$ we can choose a homomorphism $\phi: S^1 \to G_\bv$ such that
for each edge $x$, we have
\begin{equation} \label{eqn-conn-comp-fixed}
  s \cdot x = s^{w_e} x = \phi(g) \cdot x = \phi(g)_j x \phi(g)_i^{-1}.
\end{equation}
Moreover, since $S^1$ is abelian, we can assume (by conjugation by an element of $G_\bv$ if 
necessary) that $\phi$ maps to the maximal torus of $G_\bv$. Hence the vector space
$V_i$ associated to each vertex $i \in \CV$ decomposes into weight spaces, and
equation (\ref{eqn-conn-comp-fixed}) implies that for each edge $x$ mapping $i$ to $j$,
the component of $x$ mapping weight space $V_i(\lambda)$ to $V_j(\nu)$ is identically zero
unless $\nu-\lambda=w_e$. Hence every representation on this connected component is a compatible
representation as described above.
\end{proof}

\begin{example} We consider the double of the polygon space of Example \ref{ex-polygon-4},
with the $S^1$-action given by the fiberwise action as in Remark \ref{rem-different-circle-action}.
We assume that the moment map level is chosen so that the polygon space is non-empty.
Figure \ref{fig-rank-2-types} depicts all possible components of the fixed-point set.
(Not all of these will be non-empty, as this depends on the particular choice of moment
map level.)
\begin{figure}[h]
\centering
\begin{tikzpicture}
  \node[ellipse,draw] (sink1) at (-2.5,1.5) {$2$};
  \node[ellipse,draw] (sink2) at ( 0,1.5) {$1$};
  \node[ellipse,draw] (sink3) at (-4.5,-2.0) {$1$};
  \node[ellipse,draw] (sink4) at  ( 0, -2.0) {$1$};

  \foreach \s in {1,2,3,4} {
    \coordinate (o\s) at (45+90*\s:1.5);
    \node[ellipse,draw] (m\s) at ( $(sink1) + (o\s)$) {$1$};
    \path[->,>=latex',draw] (m\s) -- (sink1);
  }

  \node[ellipse,draw] (source1) at ( $(sink2) + (2,0)$ ) {$1$};
  \foreach \s in {-1,0,1} {
    \coordinate (o\s) at ( $(1,0) + \s*(0,-1)$ );
    \node[ellipse,draw] (m\s) at ( $(sink2) + (o\s)$) {$1$};
    \path[->,>=latex',draw] (m\s) -- (sink2);
    \path[->,>=latex',draw,dashed] (source1) -- (m\s);
  }
  \node[ellipse,draw] (tail1) at ( $(source1) + (1,0)$ ) {$1$};
  \path[->,>=latex',draw] (tail1) -- (source1);

  \node[ellipse,draw] (source2) at ( $(sink3) + (2,0)$ ) {$1$};
  \foreach \s in {-1,1} {
    \coordinate (o\s) at ($(1,0) + \s*(0,1)$);
    \node[ellipse,draw] (m\s) at ( $(sink3) + (o\s)$) {$1$};
    \node[ellipse,draw] (n\s) at ( $(source2) + (o\s)$) {$1$};
    \path[->,>=latex',draw] (m\s) -- (sink3);
    \path[->,>=latex',draw] (n\s) -- (source2);
    \path[->,>=latex',draw,dashed] (source2) -- (m\s);

  }

  \node[ellipse,draw] (source3) at ( $(sink4) + (2,0)$ ) {$1$};
  \foreach \s in {-1,0,1} {
    \coordinate (o\s) at ( $(1,0) + \s*(0,-1)$ );
    \node[ellipse,draw] (m\s) at ( $(source3) + (o\s)$) {$1$};
    \path[->,>=latex',draw] (m\s) -- (source3);
  }
  \node[ellipse,draw] (middle) at ( $(sink4) + (1,0)$ ) {$1$};
  \path[->,>=latex',draw] (middle) -- (sink4);
  \path[->,>=latex',draw,dashed] (source3) -- (middle);

\end{tikzpicture}
\caption{The four possible types of fixed-point set for the $S^1$-action on the doubled
polygon space.}
\label{fig-rank-2-types}
\end{figure}
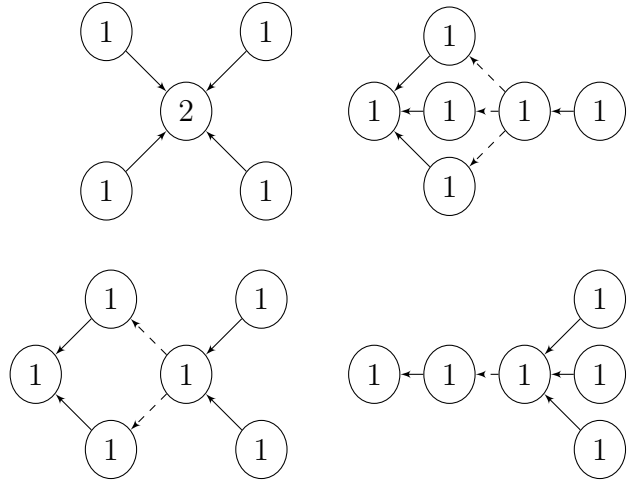
\end{example}

\subsection{Morse Theory with $\mu_\bw$ and $|\mu_\bw|^2$} \label{sec-quiver-morse-gw}

In this section we would like to understand the $G_\bw$-equivariant Morse stratification on
$\FX_\alpha(\bv,\bw)$ with respect to the $G_\bw$-moment map. We will assume throughout that
$\mu_\bw$ is proper on $\FX_\alpha(\bv, \bw)$.

Before we proceed, let us note that there is a natural direct sum operation
\begin{equation}
  \Rep(\CQ, \bv, \bw) \oplus \Rep(\CQ, \bv', \bw') \into \Rep(\CQ, \bv+\bv', \bw+\bw'),
\end{equation}
given by sending a pair of representations $(x,x')$ to the representation
$x\oplus x'$. We define a \emph{partition} of $(\bv,\bw)$ to be 
a finite sequence of dimension vectors $(\bv_1,\bw_1), \dots, (\bv_l, \bw_l)$ such that
$(\bv_1,\bw_1) + \dots + (\bv_l,\bw_l) = (\bv,\bw)$ . Denote by $\CP(\bv,\bw)$ the set of all partitions of 
$(\bv,\bw)$. For each $\lambda \in \CP(\bv,\bw)$, there is a subspace of $\Rep(\CQ, \bv,\bw)$ consisting
of those representations which decompose as a direct sum according to the partition
$\lambda$. Such representations can be enumerated diagrammatically, as shown by
the following lemma, which follows directly from the definition of direct sum.
\begin{lemma} Let $\lambda \in \CP(\bv,\bw)$. Then the representations of $\CQ$ that decompose
according to $\lambda$ are naturally identified with representations of  $|\lambda|$
disjoint copies of $\CQ$, where $|\lambda|$ denotes the length of the partition.
\end{lemma}

\begin{example} Consider the star quiver as in Figure \ref{fig-simple-quiver} with
dimension vector $\bd = (2, 1, 1, 1, 1)$, and consider the partition
$\bd = (1, 1, 0, 0,0) + (1, 0, 1, 0, 0) + (0, 0, 0, 1,1)$. Then we can represent the space of
representations which decompose according to this partition by Figure \ref{fig-simple-partition}.
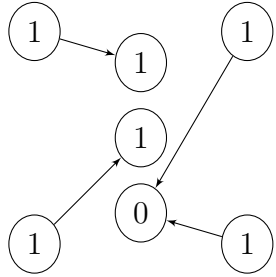
\begin{figure}[h]
\centering
\begin{tikzpicture}
  \node[ellipse,draw] (sink1) at (0, 1) {$1$};
  \node[ellipse,draw] (sink2) at (0, 0) {$1$};
  \node[ellipse,draw] (sink3) at (0,-1) {$0$};

  \foreach \s in {1} {
    \coordinate (o\s) at (45+90*\s:2);
    \node[ellipse,draw] (m\s) at ( $(sink2) + (o\s)$) {$1$};
    \path[->,>=latex',draw] (m\s) -- (sink1);
  }

  \foreach \s in {2} {
    \coordinate (o\s) at (45+90*\s:2);
    \node[ellipse,draw] (m\s) at ( $(sink2) + (o\s)$) {$1$};
    \path[->,>=latex',draw] (m\s) -- (sink2);
  }

  \foreach \s in {3,4} {
    \coordinate (o\s) at (45+90*\s:2);
    \node[ellipse,draw] (m\s) at ( $(sink2) + (o\s)$) {$1$};
    \path[->,>=latex',draw] (m\s) -- (sink3);
  }
\end{tikzpicture}
\caption{Diagrammatic depiction of the representations decomposing according to the partition
described above.}
\label{fig-simple-partition}
\end{figure}
\end{example}

\begin{theorem} \label{quiver-thm-direct-sum}
Let $\beta \in t_\bw$ be a topological generator of a torus $T_\beta \subseteq T_\bw$.
Then $[x] \in \FX_\alpha(\bv,\bw)$ is fixed by $T_\beta$ only if there is a representative $x$ 
which decomposes as a direct sum for some partition $\lambda \in \CP(\bv,\bw)$.
Conversely, if $x$ decomposes as a non-trivial direct sum, then $[x]$ is fixed by some non-trivial sub-torus
of $T_\bw$.
\end{theorem}
\begin{proof}  Since $[x]$ is fixed by $T_\beta$,
for each $s \in T_\beta$, there is some $\phi(s) \in G_\bv$ such that
\begin{equation}
  s \cdot x = \phi(s) \cdot x,
\end{equation}
and in particular we get a homomorphism $\phi: T_\beta \to G_\bv$. 
Define subtorus $T \subset G_\bd$ as the image of $s \mapsto s \phi(s)^{-1}$.
Then we can decompose $V \oplus W$ into weight spaces under the action of $T$,
and the above equation implies that for each edge $e$, the component $x_e$
decomposes as a direct sum.

Conversely, if $x$ decomposes according to a nontrivial partition $\lambda$, we can always
choose a subtorus $T_\beta \subset T_\bw$ and a map $\phi: T_\beta \to T_\bv$ such that the
weight space decomposition of $V \oplus W$ corresponds to $\lambda$. Hence $[x]$ will be
fixed by $T_\beta$.
\end{proof}

\begin{remark} The above theorem gives an explicit enumeration of the connected components
of the critical set of $|\mu_\bw|^2$ in terms of simpler quivers, since every
critical point of $|\mu_\bw|^2$ is conjugate to a point fixed by some $T_\beta \subseteq T_\bw$. 
By the results of Section \ref{morse-sec-eq-morse}, we may explicitly compute the Morse
indices, and inductively compute the equivariant cohomology. 
\end{remark}

\begin{remark} In the special case $\CW = \CI, \CV = \emptyset$, all of the edges
are even and so representations fixed by $T_\beta$ decompose as direct sums in the ordinary sense.
Additionally, in this special case we may drop the hypothesis that $\mu_\bw$ is proper,
since by Theorem \ref{GlobalEstimate} the gradient flow of $|\mu_\bw|^2$ always converges.
This special case was studied in detail in \cite{HaradaWilkin}, where it was proved
that the $G_\bd$-equivariant Morse stratification agrees with the algebraic stratification
induced by the Harder-Narasimhan filtration.
\end{remark}

\subsection{Cohomology of Subvarieties}

The results of Section \ref{sec-quiver-morse-circle} show that any quiver variety
$\FX_\alpha(\bv,\bw)$ is circle compact, and moreover that the circle action can
be chosen so that the fixed-point set is a finite union of simple quiver varieties.
The results of Chapter \ref{ch-morse}  and Section \ref{sec-quiver-morse-gw} then allow us to compute the Betti numbers of
the connected components of the fixed-point set, and in addition we may compute
the cohomology \emph{rings} of these compact quiver varieties using the more sophisticated techniques of Chapter \ref{ch-residue}.

However, hyperk\"ahler quiver varieties are defined as \emph{subvarieties} cut out by the complex
moment map equation, so there is still some work to be done. In this section we will describe
a technique that allows us to piggyback the results of Chapter \ref{ch-morse} and the preceding
sections, so as to settle this question for a large class of subvarieties of quiver varieties. 
This technique is heavily computational, and we will give a detailed case study in Section \ref{sec-rank-3-star} below.

Let us fix a quiver $\CQ$ with dimension vector $\bd = (\bv,\bw)$,
and assume that we have fixed a circle action so that $\FX_\alpha(\bv,\bw)$ is circle compact. Let 
$f: \Rep(\CQ, \bv, \bw) \to E$ be some $G_\bd \times S^1$-equivariant polynomial in the edges,
and assume that $V(f) \subset \FX_\alpha(\bv,\bw)$ is smooth. Note that
\begin{equation}
  V(f)^{S^1} = V(f) \cap \FX_\alpha(\bv,\bw)^{S^1}
\end{equation}
and hence
\begin{equation}
  \Crit\left( \left.|\mu_\bw|^2\right|_{V(f)} \right) = V(f) \cap \Crit |\mu_\bw|^2,
\end{equation}
so that the Morse stratifications with respect to the $S^1$-action and to $|\mu_\bw|^2$ on $V(f)$
can understood by the combinatorial arguments of Sections \ref{sec-quiver-morse-circle} and
\ref{sec-quiver-morse-gw}. Putting everything together, we have the following theorem.
\begin{theorem} \label{thm-inductive-procedure}
Let $(\CQ, \bv, \bw)$ be a quiver with dimension vector, and suppose that we have chosen
a circle action on $\Rep(\CQ,\bv,\bw)$ as in the preceding sections so that $\FX_\alpha(\bv,\bw)$
is circle compact and such that $f: \Rep(\CQ, \bv, \bw) \to E$ is $S^1 \times G_\bd$-equivariant,
and such that $V(f)$ is smooth. Then we may compute $P_t( V(f) \red G_\bw)$ by the following procedure.
\begin{enumerate}
  \item By Morse theory with the $S^1$-action, express the Poincar\'e polynomial of $V(f) \red G_\bw$
    in terms of the Poincar\'e polynomials of the connected components of $V(f)^{S_1} \red G_\bw$.
    Working one connected component at a time, we may now assume that $\FX_\alpha(\bv,\bw)$ is compact.
  \item By equivariant formality, $P_t^{G_\bw}( V(f) ) = P_t(BG_\bw) P_t (V(f))$.
  \item Compute $P_t(V(f))$ by Morse theory with a generic component of $\mu_\bw$.
  \item Compute $P_t( V(f) \red G_\bw )$ by Morse theory with $|\mu_\bw|^2$.
    The equivariant cohomology of the higher critical sets of $|\mu_\bw|^2$ may be computed
    from the ordinary cohomology of subvarieties of simple quiver varieties,
    which we compute by induction.
\end{enumerate}
\end{theorem}

\begin{remark} The crucial hypothesis is that we have chosen a decomposition
$\CI = \CV \sqcup \CW$ so that $V(f) \subset \FX_\alpha(\bv,\bw)$ is \emph{smooth},
so that our reduction in stages procedure involves only smooth varieties. 
We do not claim that it is always possible to do this, but we are not aware of any
counterexamples. We suspect that it is always possible for acyclic quivers.
We will see in the following section that it is always possible for star-shaped quivers.
\end{remark}
\section{Rank Three Star Quivers} \label{sec-rank-3-star}

\subsection{Star Quivers}

In this section we will illustrate the techniques developed in this chapter by giving a
detailed calculation of the cohomology of the fixed-point sets of a particular family
of hyperk\"ahler quiver varieties. We will revisit this family in Chapters \ref{ch-residue}
and \ref{ch-integrable}.

The family that we consider are the hyperk\"ahler varieties associated to the 
quiver pictured in Figure \ref{fig-rank-three-star}. In the rank two case
(where the sink has dimension two rather than three), the K\"ahler quiver varieties
are polygon spaces \cite{Klyachko, HausmannKnutson}. Their hyperk\"ahler analogues
are called hyperpolygon spaces, and have been well-studied \cite{KonnoPolygon, HaradaProudfoot}.
In particular, hyperk\"ahler Kirwan surjectivity is known to hold and their cohomology
rings have been explicitly computed.

In the rank three case, the calculations become considerably more involved. It will take
the remainder of this chapter to compute their Betti numbers, and we will not show that
hyperk\"ahler Kirwan surjectivity holds until Chapter \ref{ch-residue}. Our Morse
theory calculation closely parallels that of rank three parabolic Higgs
bundles \cite{GGM05}. This is no coincidence, as we shall see in Chapter \ref{ch-integrable}.
\begin{figure}[h]
\centering
\begin{tikzpicture}[auto, node distance=1.7cm]
\tikzset{
  unframed/.style={circle,fill=white!20,draw},
  framed/.style={rectangle,fill=white!20,draw},
  dummy/.style={draw=none,fill=none},
  myarrow/.style={->,>=latex',draw},
  myline/.style={>=latex',draw}
}
  \node[unframed] (0) at (0, 0) {$3$};

  \node[unframed] (1) at (4, 1.5) {$1$};
  \node[unframed] (2) at (4, 0.5) {$1$};
  \node[unframed] (3) at (4, -0.5) {$1$};
  \node[unframed] (4) at (4, -1.5) {$1$};

  \path[myarrow] (2,1.5) -- (0);
  \path[myarrow] (2,0.5) -- (0);
  \path[myarrow] (2,-0.5) -- (0);
  \path[myarrow] (2,-1.5) -- (0);

  \path[myline] (2,1.5) -- node {$x_1$} (1);
  \path[myline] (2,0.5) -- node {$x_2$} (2);
  \path[myline] (2,-0.5) -- node {$...$} (3);
  \path[myline] (2,-1.5) -- node {$x_n$} (4);
\end{tikzpicture}
\caption{Rank three star quiver.}
\label{fig-rank-three-star}
\end{figure}
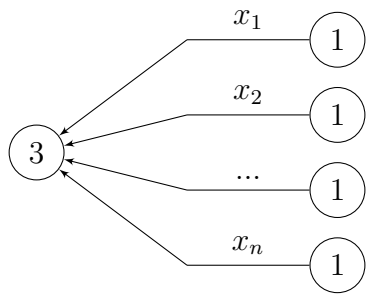

Let $x_1, \dots, x_n$ denote the incoming edges,
and let $y_1, \dots, y_n$ denote the outgoing edges on the doubled quiver. We will
explicitly write these as column and row vectors, respectively. Note that the overall
diagonal $S^1$ subgroup of $G_\bd$ acts trivially on $\Rep(\CQ, \bd)$, so strictly
speaking we should quotient only by $PG_\bd := G_\bd / S^1$. We have $\mathfrak{pg}_\bw \iso \mathfrak{su}_3 \oplus \BR^n$
and hence we take (central) moment map levels $\alpha \in \BR^n$.
We always assume that the components $\alpha_i$ of $\alpha$ are taken to be positive.
We denote the K\"ahler and hyperk\"ahler quotients by $\FX_\alpha(n)$ and $\FM_\alpha(n)$, respectively.
In what follows, we adopt the following notation. We denote by $[n]$ the set $\{1, \dots, n\}$,
and for a subset $S \subseteq [n]$, we denote by $\alpha_S$ the sum
\begin{equation}
  \alpha_S = \sum_{i \in S} \alpha_i.
\end{equation}
When it is clear from context, we will denote $\alpha_{[n]} = \sum_{i=1}^n \alpha_i$ simply
by $\alpha$.

\begin{proposition} The real and complex moment maps on $T^\ast \Rep(\CQ, \bd)$ are given by
\begin{align}
  \mur &= \frac{1}{2} \left( (xx^\ast - y^\ast y)_0, |y_1|^2-|x_1|^2, \dots, |y_n|^2-|x_n|^2 \right) \in \mathfrak{su}_3 \oplus \BR^n, \\
  \muc &= \left( (xy)_0, -y_1 x_1, \dots, -y_n x_n \right) \in \mathfrak{sl}_3 \oplus \BC^n,
\end{align}
where $(\cdot)_0$ denotes the traceless part of a matrix. For a regular value $(\alpha,0)$
of $\muhk$, both $\FX_\alpha(n)$ and $\FM_\alpha(n)$ are smooth.
\end{proposition}
\begin{proof}
The formulas for the moment maps follow by direct calculation. Smoothness follows from
\cite[Theorem 2.8]{Nakajima94} (this is a general fact about quiver varieties).
\end{proof}

\subsection{Critical Sets for the Circle Action}
The first order of business is to reduce the problem to computing the Betti numbers of
compact subvarieties, by choosing an appropriate circle action. We define the action on
$T^\ast \Rep(\CQ, \bd)$ to be
\begin{equation}
  s \cdot (x, y) = (x, sy).
\end{equation}
The moment map for this action is
\begin{equation}
  \mu_{S^1}(x,y) = \frac{1}{2} \sum_i |y_i|^2.
\end{equation}

\begin{lemma} This action makes $\FM_\alpha(n)$ circle compact.
\end{lemma}
\begin{proof} From the moment map equations,\footnote{To avoid factors of $1/2$, we are taking the
reduction at moment map level $\alpha/2$.} we have
\begin{equation}
  |x_i|^2 = \alpha_i + |y_i|^2,
\end{equation}
for each $i \in [n]$. Since the moment map for the $S^1$-action is given by $\frac{1}{2} |y|^2$,
we have that 
\begin{equation}
  |x|^2 + |y|^2 = \sum_i |x_i|^2 + |y|^2 = \sum_i \alpha_i + 2|y|^2 = \sum_i \alpha_i + 2\mu_{S^1}(x,y),
\end{equation} 
and hence $\mu_{S^1}$ is proper on $\FM_\alpha(n)$.
\end{proof}

Now that we have a circle action making $\FM_\alpha(n)$ circle compact, we must enumerate
the connected components of the fixed-point set. A point $[x,y] \in \FM_\alpha(n)$ will
be fixed only if there is a homomorphism $\phi: S^1 \to U(3) \times (S^1)^n$ such that
\begin{align}
  x_i &= g(s) x_i h_i^{-1}(s), \\
  s y_i &= h_i(s) y_i g(s)^{-1}, 
\end{align}
for each $i \in [n]$, where we have denoted the components of $\phi$ by $g \in U(2)$ and
$h_i \in S^1$. Since the moment map equations force $x_i \neq 0$ for each $i \in [n]$,
we have that $x_i$ is an eigenvector of $g(s)$ with eigenvalue $h_i(s)$. The second
equation shows that $y_i^\ast$ is an eigenvector of $g(s)$ with eigenvalue $s^{-1} h_i(s)$.

There are four possible cases depending on the multiplicities of the eigenvalues of 
$g(s)$. In the first case, suppose that $g(s)$ has a single eigenvalue of multiplicity three.
Then this forces all $y_i$ corresponding to $\FX_\alpha(n) \subset \FM_\alpha(n)$.
We call this type $(3)$.

Next suppose that $g(s)$ has two distinct eigenvalues. In order to avoid the previous case,
at least one of the $y_i$ must be non-zero. Hence the eigenvalues of $g(s)$ are
of the form $(\lambda, \lambda s^{-1})$. There are two possible cases for their
multiplicities: $(\lambda, \lambda, s^{-1}\lambda)$ and $(\lambda, s^{-1} \lambda, s^{-1} \lambda)$.
We call these type $(2,1)$ and $(1,2)$, respectively. Each of these types corresponds to
many connected components of the critical set, enumerated by the possible choices of
$h_i(s)$. Since every $x_i$ is non-zero, $h_i(s)$ can only take the value $\lambda$ or
$s^{-1} \lambda$. Hence each particular connected component is indexed by 
the set $S = \{ i \in [n] \suchthat h_i(s) = \lambda \}$. These are pictured in
Figure \ref{fig-critical-types-21-12} below.

\begin{figure}[h]
\centering
\begin{tikzpicture}[auto, node distance=1.7cm]
\tikzset{
  unframed/.style={ellipse,fill=white!20,draw},
  framed/.style={rectangle,fill=white!20,draw},
  plain/.style={draw=none,fill=none},
  myarrow/.style={->,>=latex',draw},
  myline/.style={>=latex',draw}
}
  \node[unframed] (sink1) at (4,0) {$1$};
  \node[unframed] (source1) at ($(sink1) + (4,0)$) {$2$};

  \node[unframed] (sink2) at ($(sink1)+(8,0)$) {$2$};
  \node[unframed] (source2) at ($(sink2) + (4,0)$) {$1$};


  \foreach \s in {1,2} {

    \node[plain]    (s2) at ($(sink\s)+(2,-2.5)$) {$S$};
    \node[plain]    (s3) at ($(sink\s)+(6, -2.5)$) {$S^c$};

    \node[unframed] (21) at ($(sink\s)+(2,1.5)$) {$1$};
    \node[unframed] (22) at ($(sink\s)+(2, 0.5)$) {$1$};
    \node[plain]    (23) at ($(sink\s)+(2, -0.5)$) {$...$};
    \node[unframed] (24) at ($(sink\s)+(2, -1.5)$) {$1$};

    \node[unframed] (31) at ($(source\s)+(2, 1.5)$) {$1$};
    \node[unframed] (32) at ($(source\s)+(2, 0.5)$) {$1$};
    \node[plain]    (33) at ($(source\s)+(2, -0.5)$) {$...$};
    \node[unframed] (34) at ($(source\s)+(2, -1.5)$) {$1$};

    \path[myarrow] (31) -- node[above] {$x_j$} (source\s);
    \path[myarrow] (32) -- (source\s);
    \path[myarrow] (34) -- (source\s);
    
    \path[myarrow] (21) -- node[above] {$x_i$} (sink\s);
    \path[myarrow] (22) -- (sink\s);
    \path[myarrow] (24) -- (sink\s);

    \path[myarrow,dashed] (source\s) -- node[above] {$y_i$} (21);
    \path[myarrow,dashed] (source\s) -- (22);
    \path[myarrow,dashed] (source\s) -- (24);
  }

\end{tikzpicture}
\caption{Left: type $(1,2)$; right: type $(2,1)$.}
\label{fig-critical-types-21-12}
\end{figure}
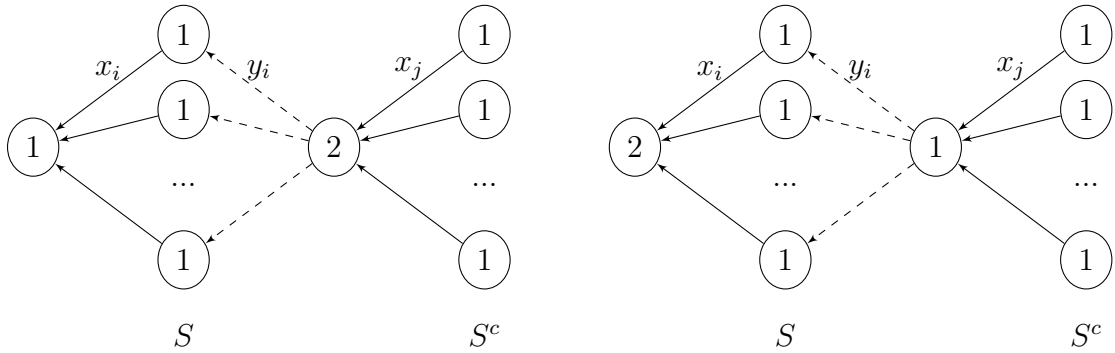

In the last case, we assume that $g(s)$ has three distinct eigenvalues. 
We can assume that they are of the form $(\lambda, s^{-1} \lambda, s^{-a}\lambda)$ and with $a \neq 1, 0, -1$.
Now, if $\alpha$ is assumed to be generic so that $\FM_\alpha(n)$ is smooth, by Theorem
\ref{quiver-thm-direct-sum} no representation $[x,y] \in \FM_\alpha(n)$ admits
a non-trivial direct sum decomposition. Hence the quiver representing this connected
component must be connected, and the only way to obtain this is to take $a=2$.
We call this a type $(1,1,1)$ critical set. The type $(1,1,1)$ critical sets are
indexed by the choice of 
\begin{align}
  S_1 &= \{ i \in [n] \suchthat h_i(s) = \lambda \}, \\
  S_2 &= \{i \in [n] \suchthat h_i(s) = s^{-1} \lambda\}, \\
  S_3 &= \{i \in [n] \suchthat h_i(s) = s^{-2} \lambda\}.
\end{align}
The diagrammatic representation of the type $(1,1,1)$ critical set is pictured in
Figure \ref{fig-critical-type-1-1-1}.
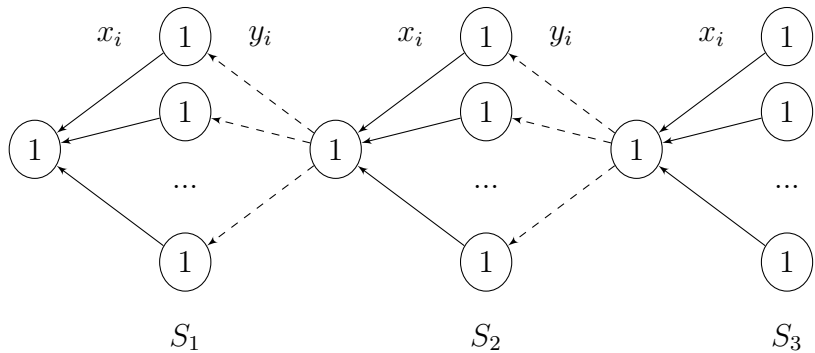
\begin{figure}[h]
\centering
\begin{tikzpicture}[auto, node distance=1.7cm]
\tikzset{
  unframed/.style={ellipse,fill=white!20,draw},
  framed/.style={rectangle,fill=white!20,draw},
  plain/.style={draw=none,fill=none},
  myarrow/.style={->,>=latex',draw},
  myline/.style={>=latex',draw}
}
  \node[unframed] (1) at (0,0) {$1$};
  \node[unframed] (2) at (4,0) {$1$};
  \node[unframed] (3) at (8,0) {$1$};

  \node[plain]    (s1) at (2, -2.5) {$S_1$};
  \node[plain]    (s2) at (6, -2.5) {$S_2$};
  \node[plain]    (s3) at (10, -2.5) {$S_3$};

  \node[plain]    (x1) at (1, 1.5) {$x_i$};
  \node[plain]    (x2) at (5, 1.5) {$x_i$};
  \node[plain]    (x3) at (9, 1.5) {$x_i$};

  \node[plain]    (y1) at (3, 1.5) {$y_i$};
  \node[plain]    (y2) at (7, 1.5) {$y_i$};

  \node[unframed] (11) at (2, 1.5) {$1$};
  \node[unframed] (12) at (2, 0.5) {$1$};
  \node[plain]    (13) at (2, -0.5) {$...$};
  \node[unframed] (14) at (2, -1.5) {$1$};

  \node[unframed] (21) at (6, 1.5) {$1$};
  \node[unframed] (22) at (6, 0.5) {$1$};
  \node[plain]    (23) at (6, -0.5) {$...$};
  \node[unframed] (24) at (6, -1.5) {$1$};

  \node[unframed] (31) at (10, 1.5) {$1$};
  \node[unframed] (32) at (10, 0.5) {$1$};
  \node[plain]    (33) at (10, -0.5) {$...$};
  \node[unframed] (34) at (10, -1.5) {$1$};

  \path[myarrow] (31) -- (3);
  \path[myarrow] (32) -- (3);
  \path[myarrow] (34) -- (3);

  \path[myarrow] (21) -- (2);
  \path[myarrow] (22) -- (2);
  \path[myarrow] (24) -- (2);
  \path[myarrow,dashed] (3) -- (21);
  \path[myarrow,dashed] (3) -- (22);
  \path[myarrow,dashed] (3) -- (24);

  \path[myarrow] (11) -- (1);
  \path[myarrow] (12) -- (1);
  \path[myarrow] (14) -- (1);
  \path[myarrow,dashed] (2) -- (11);
  \path[myarrow,dashed] (2) -- (12);
  \path[myarrow,dashed] (2) -- (14);

\end{tikzpicture}
\caption{Type $(1,1,1)$ critical set.}
\label{fig-critical-type-1-1-1}
\end{figure}

We summarize the above arguments in the following theorem.
\begin{proposition} Let $\FM_\alpha$ be a rank 3 star quiver variety, and assume
that $\alpha$ is generic. Then the connected components of $\FM_\alpha^{S^1}$ can be 
classified into three types: $(3)$, $(2,1)$, $(1,2)$, and $(1,1,1)$. Type $(3)$ consists
of the single component $\FX_\alpha(n)$. The type $(2,1)$ and $(1,2)$ critical sets are
indexed by a subset $S \subset[n]$, and the type $(1,1,1)$ critical sets are indexed by
a pair of disjoint subsets $S_1, S_2 \subset [n]$.
\end{proposition}

In the remaining sections, we will show how to use Theorem \ref{thm-inductive-procedure} to
compute the Betti numbers of these critical sets. The homotopy type of $\FM_n(\alpha)$ does
not depend on the particular choice of $\alpha$, as long as it is generic,\footnote{This is
a well-known fact about hyperk\"ahler quotients. Various special cases appear in the literature,
e.g. \cite[Corollary 4.2]{Nakajima94},
and in our case it is easily deduced from Corollary \ref{HomotopyCorollary}.}
so we can simplify the calculation by choosing $\alpha$ to be \emph{extreme}---we can ensure 
that the type $(3)$ and $(1,2)$ critical sets are always empty. This reduces
the problem to calculating the Betti numbers of the type $(2,1)$ and $(1,1,1)$ critical sets,
and of these only type $(2,1)$ turns out to be non-trivial. 
However, it should be emphasized
that the method of extreme stability is in no way whatsoever essential---it is just a trick to reduce
the amount of calculation necessary.

\subsection{The Type $(3)$ Critical Set}

As noted above, the type (3) critical set is just the K\"ahler quiver variety $\FX_\alpha \subset \FM_\alpha$,
and thus has index $0$, since it is the absolute minimum of $\mu_{S^1}$. Its Poincar\'e polynomial can be computed straightforwardly
using the equivariant Morse theory of Chapter \ref{ch-morse}. However, we can avoid this calculation
entirely simply by choosing $\alpha$ to be extreme.

\begin{proposition} Suppose that $\alpha$ satisfies
$\alpha_n > \sum_{i=1}^{n-1} \alpha_i$.
Then the K\"ahler quiver variety $\FX_\alpha(n)$ is empty.
\end{proposition}
\begin{proof} The moment map equations are
\begin{align}
  \sum_i (x_i x_i^\ast)_0 &= 0, \\
  |x_i|^2 &= \alpha_i.
\end{align}
Now consider the map sending $x_i$ to $v_i := (x_i x_i^\ast)_0 \in \mathfrak{su}_3$. 
The first moment map equation demands that $\sum v_i = 0$, and hence
\begin{equation}
  |v_n| = \left| \sum_{i=1}^{n-1} v_i \right| \leq \sum_{i=1}^{n-1} |v_i|.
\end{equation}
On the other hand, we have
\begin{align}
  |v_i|^2 &= \Tr\left(v_i v_i^\ast\right) \\ 
   &= \Tr \left( \left(x_i x_i^\ast - \frac{1}{3} |x_i|^2 \right) 
       \left(x_i x_i^\ast - \frac{1}{3} |x_i|^2 \right) \right) \\
   &= |x_i|^4 - \frac{2}{3} |x_i|^4 + \frac{1}{3} |x_i|^4 \\
   &= \frac{2}{3} |x_i|^4 = \frac{2}{3} \alpha_i^2.
\end{align}
Hence if we choose $\alpha_n > \sum_{i=1}^{n-1} \alpha_i$, then $|v_n|$ violates the above upper bound,
and there can be is no solution to the moment map equations.
\end{proof}

\subsection{The Type $(1,2)$ Critical Sets}

Next we consider the type $(1,2)$ critical sets. Their Poincar\'e polynomials can be computed
in a manner analogous to the technique described below for the type $(2,1)$. However, we can 
always eliminate this case by choosing $\alpha$ to be extreme.

\begin{proposition} \label{quiver-prop-type-1-2}
The type $(1,2)$ critical set indexed by $S \subset [n]$ is empty unless the following
conditions hold: 
\begin{inparaenum}[(i)]
  \item $\alpha_{S^c} > 2 \alpha_S$, 
  \item $\#S \geq \max(1, 6-n)$, and
  \item $\#S^c \geq 2$.
\end{inparaenum}
The isotropy representation at any point in this connected component of the critical set
is given by
\begin{equation}
  n+\#S-6 + (n+\#S-6)t + (\#S^c-2)t^2 + (\#S^c-2) t^{-1}
\end{equation}
In particular, the complex dimension of this connected component is $n+\#S-6$ and
its Morse index is $2(\#S^c - 2)$.
\end{proposition}
\begin{proof}
A type $(1,2)$ critical set labelled by $S \subset [n]$ corresponds to a map $\phi: S^1 \to U(3) \times (S^1)^n$
given by 
\begin{equation}
 g(s) = \diag(1,  s^{-1},  s^{-1}), \ h_i(s) = 1\ \textrm{for}\ i \in S, \ h_i(s) = s^{-1}\ \textrm{for}\ i \in S^c.
\end{equation}
For $i \in S$, we have $x = (\ast, 0, 0)^T$, $y_i = (0, \ast, \ast)$, and for $i \in S^c$
we have $x_i = (0, \ast, \ast)^T$, $y_i = 0$. Then the real moment map equation 
$\sum_i x_i x_i^\ast - y_i^\ast y_i = (\alpha/3) 1_{3 \times 3}$ gives
\begin{equation} \label{Eqn-12-1}
\sum_{i\in S} |x_i|^2 = \alpha / 3,
\end{equation}
\begin{equation} \label{Eqn-12-2}
\sum_{i \in S^c} |x_i|^2 - \sum_{i \in S} |y_i|^2 = 2 \alpha / 3. 
\end{equation}
Using the real moment map equation $|x_i|^2 = |y_i|^2 + \alpha_i$, we find
\begin{equation} 
  \sum_i |y_i|^2 = \sum_{i \in S} |y_i|^2 = \sum_{i \in S} ( |x_i|^2 - \alpha_i) =\alpha/3 - \alpha_S = \frac{\alpha_{S^c} - 2\alpha_S}{3}, 
\end{equation}
from which we deduce $\alpha_{S^c} > 2 \alpha_S$. From equation \ref{Eqn-12-1} we see
that $S$ is nonempty, so $\# S \geq 1$. From the real moment map equation, we also see 
that $\sum_{i \in S^c} x_i x_i^\ast$ has rank 2, so $\# S^c \geq 2$.
Finally, the map $\phi$ determines representations
\begin{align}
  [T_x \Rep(\CQ)] &= 2n + \#S + (2n + \#S)t + (\#S^c)t^{-1} + (\#S^c)t^2 \\
  [\fg] &= n+4 + 2t + 2t^{-1} \\
  [E] &= 2 + (n+4)t + 2t^2.
\end{align}
Using Theorem \ref{thm-rep-ring} we obtain the isotropy representation as desired.
\end{proof}

Now using condition (1) above, we can ensure that neither the type $(3)$ nor the type
$(1,2)$ critical sets contribute to $P_t(\FM_\alpha(n))$.

\begin{proposition} Suppose that $\alpha$ satisfies $\alpha_n > 2 \sum_{i=1}^{n-1} \alpha_i$. 
Then there are no type $(1,2)$ critical sets.
\end{proposition}
\begin{proof} We repeat essentially the same argument as in the type $(3)$ case.
From the proof of Proposition \ref{quiver-prop-type-1-2} above, we have that 
$|x_i|^2 = \alpha_i$ for $i \in S^c$, and that 
\begin{equation}
  \sum_{i \in S^c} x_i x_i^\ast = \frac{\alpha_{S^c}}{3} 1_{3 \times 3}.
\end{equation}
Now, the condition $\alpha_{S^c} > 2\alpha_S$ implies that necessarily $n \in S^c$.
As in the type $(3)$ case, consider the vectors $v_i := (x_i x_i^\ast)_0$. The
moment map equation implies $\sum v_i = 0$, so by the triangle inequality we find
\begin{equation}
  |v_n| \leq \sum_{i=1}^{n-1} |v_i|,
\end{equation}
but as in the type $(3)$ case this contradicts $\alpha_n > 2 \sum_{i=1}^{n-1} \alpha_i > \sum_{i=1}^{n-1} \alpha_i$.
\end{proof}

\subsection{The Type $(2,1)$ Critical Sets}

Finally, we encounter the first case which cannot be eliminated by extreme stability.

\begin{proposition} \label{quiver-prop-type-2-1}
The type $(2,1)$ critical set is empty unless the following conditions hold:
\begin{inparaenum}[(i)]
  \item $2\alpha_{S^c} > \alpha_S$, 
  \item $\#S \geq 3$, and 
  \item $\#S^c \geq 1$.
\end{inparaenum}
The isotropy representation of any point is given by
\begin{equation}
  2\#S-6 + (2\#S-6)t + (2\#S^c-2)t^2 + (2\#S^c-2)t^{-1}. 
\end{equation}
In particular, if it is non-empty then its complex dimension is $2\#S-6$ and its
Morse index is $2(2\#S^c-2)$.
\end{proposition}
\begin{proof}
A type $(2,1)$ critical set labelled by $S \subseteq [n]$ corresponds to a homomorphism
$\phi: S^1 \to U(3) \times (S^1)^n$ given by
\begin{equation}
 g(s) = \diag(1,  1,  s^{-1}), \ h_i(s) = 1\ \textrm{for}\ i \in S, \ h_i(s) = s^{-1}\ \textrm{for}\ i \in S^c.
\end{equation}
For $i \in S$, we have $x = (\ast, \ast, 0)^T$, $y_i = (0, 0, \ast)$, and for $i \in S^c$
we have $x_i = (0, 0, \ast)^T$, $y_i = 0$. Then the real moment map equation 
$\sum_i x_i x_i^\ast - y_i^\ast y_i = (\alpha/3) 1_{3 \times 3}$ gives
\begin{equation} \label{Eqn-21-1} \sum_{i \in S} |x_i|^2 = 2 \alpha / 3, \end{equation}
\begin{equation} \label{Eqn-21-2} \sum_{i \in S^c} |x_i|^2 - \sum_{i \in S} |y_i|^2 = \alpha / 3. \end{equation}
Using the real moment map equation $|x_i|^2 = |y_i|^2 + \alpha_i$ and rearranging, we find
\begin{equation} \sum_i |y_i|^2 = \frac{2\alpha_{S^c} - \alpha_S}{3}, \end{equation}
from which we deduce $2\alpha_{S^c} > \alpha_S$. From equation \ref{Eqn-21-2} we see
that $S^c$ is nonempty, so $\# S \geq 1$. From the real moment map equation, we also see 
that $\sum_{i \in S} x_i x_i^\ast$ has rank 2, so $\# S \geq 2$.
Finally, the subgroup determines representations
\begin{align}
 [T_x \Rep(\CQ)] &= n + 2\#S + (n+2\#S)t + (2\#S^c) t^2 + (2\#S^c)t^{-1} \\
 [\fg] &= n+4 + 2t + 2t^{-1} \\
 [E] &= 2 + (n+4)t + 2t^2 
\end{align}
where $E \iso \mathfrak{psl}_2 \oplus \BC^n$ is the target space of the complex moment map.
Hence by Theorem \ref{thm-rep-ring} we obtain the desired expression for the 
isotropy representation.
\end{proof}

We will compute the Poincar\'e polynomials of the type $(2,1)$ critical sets by Theorem \ref{thm-inductive-procedure}.
To do this, we have to choose a decomposition $\CI = \CV \sqcup \CW$ in order to apply
reduction in stages. To ensure that the subvarieties $V(f)$ (defined by the residual
part of the complex moment map) are smooth, we make the following choice: $\CW = S$.
The decorated quiver corresponding to this decomposition is pictured in Figure \ref{fig-type-2-1-stages}.

\begin{figure}[h]
\centering
\begin{tikzpicture}[auto, node distance=1.7cm]
\tikzset{
  unframed/.style={ellipse,fill=white!20,draw},
  framed/.style={rectangle,fill=white!20,draw},
  plain/.style={draw=none,fill=none},
  myarrow/.style={->,>=latex',draw},
  myline/.style={>=latex',draw}
}
  \node[unframed] (2) at (4,0) {$2$};
  \node[unframed] (3) at (8,0) {$1$};

  \node[plain]    (s2) at (6, -2.5) {$S$};
  \node[plain]    (s3) at (10, -2.5) {$S^c$};

  \node[plain]    (x2) at (5, 1.5) {$x_i$};
  \node[plain]    (x3) at (9, 1.5) {$x_i$};

  \node[plain]    (y2) at (7, 1.5) {$y_i$};

  \node[framed] (21) at (6, 1.5) {$1$};
  \node[framed] (22) at (6, 0.5) {$1$};
  \node[plain]    (23) at (6, -0.5) {$...$};
  \node[framed] (24) at (6, -1.5) {$1$};

  \node[unframed] (31) at (10, 1.5) {$1$};
  \node[unframed] (32) at (10, 0.5) {$1$};
  \node[plain]    (33) at (10, -0.5) {$...$};
  \node[unframed] (34) at (10, -1.5) {$1$};

  \path[myarrow] (31) -- (3);
  \path[myarrow] (32) -- (3);
  \path[myarrow] (34) -- (3);

  \path[myarrow] (21) -- (2);
  \path[myarrow] (22) -- (2);
  \path[myarrow] (24) -- (2);
  \path[myarrow] (3) -- (21);
  \path[myarrow] (3) -- (22);
  \path[myarrow] (3) -- (24);

\end{tikzpicture}
\caption{The reduction in stages used for the type $(2,1)$ critical set.}
\label{fig-type-2-1-stages}
\end{figure}
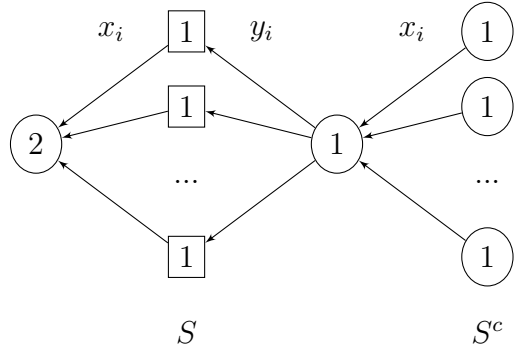

To simplify things, let us note that up to isomorphism we can ignore the
edges $x_i$ for $i \in S^c$. This is pictured in Figure \ref{fig-type-2-1-stages-simple}.

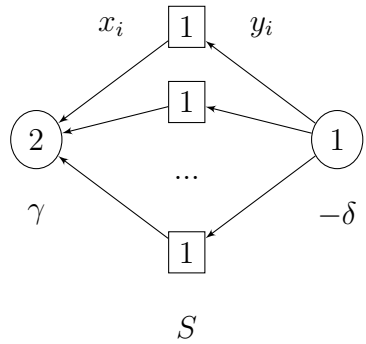
\begin{figure}[h]
\centering
\begin{tikzpicture}[auto, node distance=1.7cm]
\tikzset{
  unframed/.style={ellipse,fill=white!20,draw},
  framed/.style={rectangle,fill=white!20,draw},
  plain/.style={draw=none,fill=none},
  myarrow/.style={->,>=latex',draw},
  myline/.style={>=latex',draw}
}
  \node[unframed] (2) at (4,0) {$2$};
  \node[unframed] (3) at (8,0) {$1$};
  
  \coordinate (dy) at (0,1);

  \node[plain] (a2) at ( $(2)-(dy)$ ) {$\gamma$};
  \node[plain] (a3) at ( $(3)-(dy)$ ) {$-\delta$};

  \node[plain]    (s2) at (6, -2.5) {$S$};

  \node[plain]    (x2) at (5, 1.5) {$x_i$};

  \node[plain]    (y2) at (7, 1.5) {$y_i$};

  \node[framed] (21) at (6, 1.5) {$1$};
  \node[framed] (22) at (6, 0.5) {$1$};
  \node[plain]  (23) at (6, -0.5) {$...$};
  \node[framed] (24) at (6, -1.5) {$1$};

  \path[myarrow] (21) -- (2);
  \path[myarrow] (22) -- (2);
  \path[myarrow] (24) -- (2);
  \path[myarrow] (3) -- (21);
  \path[myarrow] (3) -- (22);
  \path[myarrow] (3) -- (24);

\end{tikzpicture}
\caption{Simplification of the type $(2,1)$ critical set, with the moment map level
$(\gamma,\delta)$ as indicated.}
\label{fig-type-2-1-stages-simple}
\end{figure}

\begin{lemma} The type $(2,1)$ critical set indexed by $S \subset [n]$ is isomorphic
to the moduli space of representations of the quiver with the edges
$x_i$ deleted for $i \in S^c$, as in Figure \ref{fig-type-2-1-stages-simple}. 
Under this isomorphism, the moment map level at the leftmost vertex is given by
$(\gamma/2)1_{2 \times 2}$ where $\gamma = 2\alpha/3$,
 and at the rightmost vertex by $-\delta$, where $\delta = (2\alpha_{S^c} - \alpha_S) / 3$.
\end{lemma}
\begin{proof} For each $i \in S^c$, the moment map equation is given by $|x_i|^2 = \alpha_i$.
From Proposition \ref{quiver-prop-type-2-1}, the moment map equation at the vertex is
\begin{equation}
  \sum_{i \in S^c} |x_i|^2 - \sum_{i \in S} |y_i|^2 = \frac{\alpha}{3}
\end{equation}
Now, since for $i \in S^c$ $x_i$ is a $1 \times 1$ matrix of norm $\alpha_i$, up to the 
$S^1$ action we simply have $x_i = \sqrt{\alpha_i}$ and there are no additional moduli.
Hence we can simply delete these edges. Rearranging the above equation, we have
\begin{equation}
  -\sum_{i \in S} |y_i|^2 = \frac{\alpha}{3} - \sum_{i \in S^c} |x_i|^2 = \frac{2\alpha_{S^c}-\alpha_S}{3} = -\delta,
\end{equation}
as claimed. The moment map equation
\begin{equation}
  \sum_i x_i x_i^\ast = \frac{\alpha}{3} 1_{3 \times 3} + \sum_i y_i^\ast y_i,
\end{equation}
also implies $\sum_{i \in S} x_i x_i^\ast = (\alpha/3) 1_{2 \times 2}$, so $\gamma = 2\alpha/3$
as claimed.
\end{proof}

For the remainder of this section, we let $\CQ = (\CV \sqcup \CW, \CE)$ denote the type
$(2,1)$ quiver pictured in Figure \ref{fig-type-2-1-stages-simple}. The first step
of Theorem \ref{thm-inductive-procedure} is to compute the Poincar\'e polynomial
of $V(f) \subset \FX_\alpha(\CQ, \bv, \bw)$ where $f: \Rep(\CQ, \bv, \bw) \to E \iso \BC^2$ is the residual part of the complex
moment map, given by
\begin{equation}
  f(x,y) = \sum_{i \in S} x_i y_i.
\end{equation}

\begin{proposition} \label{quiver-pro-vf-poincare} The
Poincar\'e polynomial of $V(f) \subset \FX_\alpha(\CQ,\bv,\bw)$ is given by 
\begin{equation}
  P_t( V(f) ) = P_t( Gr(2,n) ) P_t( \BP^{n-3} )
\end{equation}
\end{proposition}
\begin{proof} Pick a generic component $\beta$ of $\ft_\bw$ to that $T_\beta = T_\bw$.
(We will make a particular choice of $\beta$ once we need it to compute Morse indices.)
Let $h = \ev{\mu_\bw, \beta}$ so that $h$ is a perfect Morse-Bott function and 
$\Crit(h) = \FX_\alpha(\CQ, \bv,\bw)^{T_\bw}$. To proceed, we have to compute the $T_\bw$-fixed
point set of $V(f)$. As in the proof of Theorem \ref{quiver-thm-direct-sum},
a point $[x,y] \in \FX_\alpha(\CQ, \bv,\bw)$ is fixed if there is a map $\phi: T_\bw \to G_\bv$
such that $t \cdot (x,y) = \phi(t) \cdot (x,y)$.
If we denote $t = (t_1, \dots, t_n)$ then, writing this out in components, we have the
equations 
\begin{align}
  g_1(t) x_i  &= t_i x_i, \\
  y_i g_2(t)^{-1} &= t_i^{-1} y_i.  
\end{align}
Now, by the moment map equations, at least two of the $x_i$ must be non-zero and
linearly independent, say $x_a$ and $x_b$ with $a<b$. Then the first of the above
equations force $g_1(t) = \diag(t_a, t_b)$. On the other hand, the moment map 
equations also force at least one of the $y_i$ to be non-zero, say $y_c$.
The second equation then forces $g_2(t) = t_c$, and this determines $\phi$ completely.
Furthermore, the residual complex moment map equation is $\sum_{i \in S} x_i y_i = 0$, 
and this can only be satisfied if $c \neq a,b$. Hence the fixed-points are isolated
and are indexed by tuples $(a,b,c)$ with $a<b$ and $c\neq a,b$.

To compute Morse indices, we first compute the isotropy representations using 
Theorem \ref{quiver-thm-direct-sum}. We find
\begin{align}
  [T_x \Rep(\CQ, \bv,\bw)] &= \sum_l \left( t_a t_l^{-1} + t_b t_l^{-1} + t_l t_c^{-1} \right) \\
  [\fg_\bv] &= t_a t_b^{-1} + t_b t_a^{-1} + 3 \\
  [E] &= t_a t_c^{-1} + t_b t_c^{-1}.
\end{align}
If we take the particular generic component $(1, \epsilon, \epsilon^2, \dots, \epsilon^{n-1})$
for $\epsilon$ small, then a weight of the form $t_i t_j^{-1}$ is negative if and only if
$j < i$. Hence the Morse index is given by
\begin{equation}
  \lambda(a,b,c) = 2n -6 + 2a + 2b + 2c - 2d,
\end{equation}
where $d = 0$ if $c < a$, $1$ if $a < c < b$, and $0$ if $c > b$. Hence 
\begin{equation}
  P_t(V) = \sum_{a=1}^{n-1} \sum_{b=a+1}^n \sum_{c \neq a,b} t^{\lambda(a,b,c)}
\end{equation}
and with a bit of algebraic manipulation this yields the claimed formula for $P_t(V(f))$.
\end{proof}

\begin{remark} A quicker way to calculate $P_t(V)$ is to simply observe that the
projection $[x,y] \mapsto [x]$ makes $V(f)$ the total space of a fiber bundle over
$Gr(2,n)$ with fiber $\BP^{n-3}$. However, Morse theory with the $T_\bw$-action provides
a systematic way to compute $V(f)$ for a large class of subvarieties of quiver varieties
without relying on such coincidences.
\end{remark}

Now that we have calculated the Poincar\'e polynomial of $P_t(V(f))$, we must next
enumerate the connected components of the critical set of $|\mu_\bw|^2$. By Theorem \ref{quiver-thm-direct-sum},
these are enumerated by the possible direct sum decompositions. At the left-most vertex,
which has dimension $2$, there are two possibilities: either it does not decompose,
or it decomposes as $1+1$. We call these two possibilities type $(2)$ and type $(1,1)$
respectively.

First let us consider type $(2)$. Under a direct sum decomposition of this form, the
only possibility is that some of the edges are identically zero. Define
\begin{align}
  S = \{i \in [n] \suchthat x_i \neq 0\ \textrm{generically} \}, \\
  T = \{i \in [n] \suchthat y_i \neq 0\ \textrm{generically} \}.
\end{align}
The type $(2)$ critical set corresponding to the choice $S,T \subset[n]$ is pictured in 
Figure \ref{fig-type-2-1-stratum-2}.

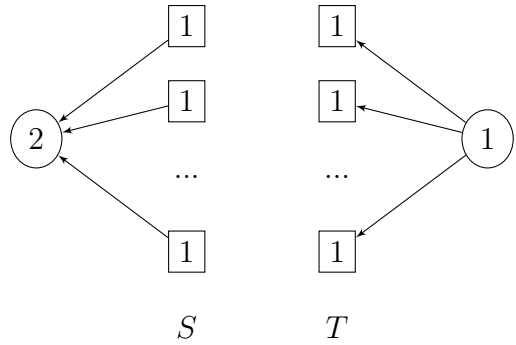
\begin{figure}[h]
\centering
\begin{tikzpicture}[auto, node distance=1.7cm]
\tikzset{
  unframed/.style={ellipse,fill=white!20,draw},
  framed/.style={rectangle,fill=white!20,draw},
  plain/.style={draw=none,fill=none},
  myarrow/.style={->,>=latex',draw},
  myline/.style={>=latex',draw}
}  
  \node[unframed] (1) at (-2, 0) {$2$};
  \node[framed] (L1) at (0, 1.5) {$1$};
  \node[framed] (L2) at (0, 0.5) {$1$};
  \node[plain]  (L3) at (0, -0.5) {$...$};
  \node[framed] (L4) at (0, -1.5) {$1$};

  \node[plain]  (S) at (0, -2.5) {$S$};

  \path[myarrow] (L1) -- (1);
  \path[myarrow] (L2) -- (1);
  \path[myarrow] (L4) -- (1);

  \node[unframed] (2) at (4, 0) {$1$};

  \node[framed]   (R1) at (2, 1.5) {$1$};
  \node[framed]   (R2) at (2, 0.5) {$1$};
  \node[plain]    (R3) at (2, -0.5) {$...$};
  \node[framed]   (R4) at (2, -1.5) {$1$};
  \node[plain]    (T)  at (2, -2.5) {$T$};
  
  \path[myarrow] (2) -- (R1);
  \path[myarrow] (2) -- (R2);
  \path[myarrow] (2) -- (R4);

\end{tikzpicture}
\caption{The type $(2)$ stratum of the type $(2,1)$ critical set.}
\label{fig-type-2-1-stratum-2}
\end{figure}

Now, suppose that $S \cap T \neq \emptyset$, and let $i \in S \cap T$.
Since $x_i \neq 0$ generically, this forces $g_1 = \diag(t_i, t_i)$.
Similarly, since $y_i \neq 0$ generically, we have $g_2 = t_i$.
Hence the generic stabilizer is given by the subtorus
$(s_1, \dots, s_n)$ of $T_\bw$, where $s_l = t_i$ for $l \in S \cup T$,
and $s_l = t_l$ for $l \neq S \cup T$. However, the fixed-point set
of this stabilizer subgroup consists of 
\begin{equation}
  \{x_i = 0 \suchthat i \not\in (S\cup T) \} \cup \{y_i = 0 \suchthat i \not\in (S\cup T)\},
\end{equation}
so in fact $S = T$. This divides the type $(2)$ critical sets into two subcases,
$S = T$ and $S \cap T = \emptyset$, which we call type $A$ and type $B$
respectively.

Next we consider the remaining case, type $(1,1)$. Now we define $S$ and $T$ as
in the type $(2)$ case, and the decomposition $2=1+1$ further decomposes
$S$ as $S = S_1 \sqcup S_2$, as pictured in Figure \ref{fig-type-2-1-stratum-1-1}.
Now, as in the type $(2)$ case there are additional subcases. By arguments
analogous to those above, either $T$ is
equal to one of the $S_i$ (which we may assume is $S_1$, simply by relabelling),
or $T$ is disjoint from both $S_1$ and $S_2$. We denote these subcases by
type $C$ and type $D$, respectively.

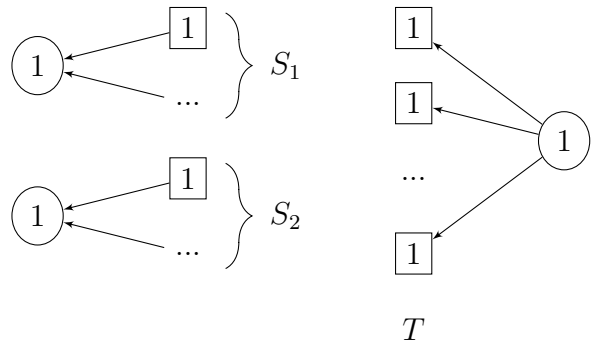
\begin{figure}[h]
\centering
\begin{tikzpicture}[auto, node distance=1.7cm]
\tikzset{
  unframed/.style={ellipse,fill=white!20,draw},
  framed/.style={rectangle,fill=white!20,draw},
  plain/.style={draw=none,fill=none},
  myarrow/.style={->,>=latex',draw},
  myline/.style={>=latex',draw}
}
  \node[unframed] (1a) at (-4, 1) {$1$};
  \node[unframed] (1b) at (-4, -1) {$1$};
  \node[framed] (L1) at (-2, 1.5) {$1$};
  \node[plain]  (L2) at (-2, 0.5) {$...$};
  \node[framed] (L3) at (-2, -0.5) {$1$};
  \node[plain]  (L4) at (-2, -1.5) {$...$};

  \path[myarrow] (L1) -- (1a);
  \path[myarrow] (L2) -- (1a);
  \path[myarrow] (L3) -- (1b);
  \path[myarrow] (L4) -- (1b);

  \node[plain] at (-0.7, 1) {$S_1$};
  \node[plain] at (-0.7, -1) {$S_2$};

  \draw[decorate,decoration={brace, amplitude=10pt}] (-1.5, 1.7) -- (-1.5, 0.3);
  \draw[decorate,decoration={brace, amplitude=10pt}] (-1.5, -0.3) -- (-1.5, -1.7);

  \node[unframed] (2) at (3, 0) {$1$};

  \node[framed]   (R1) at (1, 1.5) {$1$};
  \node[framed]   (R2) at (1, 0.5) {$1$};
  \node[plain]    (R3) at (1, -0.5) {$...$};
  \node[framed]   (R4) at (1, -1.5) {$1$};
  \node[plain]    (T)  at (1, -2.5) {$T$};
  
  \path[myarrow] (2) -- (R1);
  \path[myarrow] (2) -- (R2);
  \path[myarrow] (2) -- (R4);

\end{tikzpicture}
\caption{The type $(1,1)$ stratum on the type $(2,1)$ quiver.}
\label{fig-type-2-1-stratum-1-1}
\end{figure}

Summarizing the above arguments, we have the following theorem.

\begin{theorem}
The critical sets for $|\mu_\bw|^2$ on the type $(2,1)$ are classified into the types 
$A, B, C$, and $D$ described above. Hence The Poincar\'e polynomial of the type $(2,1)$ critical set on 
$\FM_\alpha(n)$ may be computed as 
\begin{equation}
  P_t( V(f) \red G_\bw) = P_t(BG_\bw) P_t(V(f)) - R_A - R_B - R_c - R_D,
\end{equation}
where $R_A, R_B, R_C, R_D$ are corrections due to the type $A,B,C,D$ critical sets
of $|\mu_\bw|^2$, respectively. These may be computed by Propositions
\ref{quiver-prop-type-a}, \ref{quiver-prop-type-b}, \ref{quiver-prop-type-c}, and 
\ref{quiver-prop-type-d} below.
\end{theorem}

\subsubsection{Type $A$ Critical Sets}
We have a subset $S \subset [n]$ satisfying $S^c \neq \emptyset$.
This critical set is given by $x_i = 0, y_i = 0$ for $i \in S^c$.

\begin{lemma} The stabilizer of a generic point of the type $A$ critical set
is the subgroup $T_S \subseteq T_\bw$ consisting of elements of the form $t_i = s$ for $i \in S$ with
no constraint on the $t_i$ for $i \in S^c$. The map $\phi: T_S \to T_\bv$ is given by 
$g_1 = \diag(s^{-1},s^{-1}), g_2 = s^{-1}$. The isotropy representation is given by
\begin{equation}
  3(\#S)-7 + \sum_{i \in S^c} (2s t_i^{-1} + s^{-1} t_i).
\end{equation}
\end{lemma}
\begin{proof} The generic stabilizer follows from the definition of the type $A$ critical set.
To compute the isotropy representation, we simply compute
\begin{align}
  [\Rep(\CQ, \bv,\bw)] &= 3(\#S) + \sum_{i \in S^c} (2 s t_i^{-1} + s^{-1} t_i) \\
  [\fg_\bv] &= 5 \\
  [E] &= 2
\end{align}
and appeal to Theorem \ref{thm-rep-ring}.
\end{proof}

\begin{lemma} The Morse index $\lambda$ is given by
\begin{equation}
  \lambda = 2\#S^c + 2\#\{j \in S^c \ | \ \mu_S + \alpha_j > 0 \}.
\end{equation}
where $\mu_S = (\gamma - \delta - \alpha_S) / (\# S)$.
The equivariant topology of the type $A$ critical set
is that of the type $(2,1)$ critical set with stability parameters
$(\gamma, \alpha', \delta)$, where $\alpha' = \{\alpha_i'\}_{i \in S}$ 
and $\alpha_i' = \alpha_i + \mu_S$.
Hence its Poincar\'e polynomial may be computed inductively.
\end{lemma}
\begin{proof} 
The condition of being critical forces $\mu(x)$ to lie in $\stab(x)$, and hence
$\mu_i = \mu_{i'}$ for $i,i' \in S$. Denote this common value by $\mu_S$.
Using the moment map equations, we find
\begin{equation}
 (\#S) \mu_S = \sum_{i \in S} (|x_i|^2 - |y_i|^2 - \alpha_i) = \gamma - \delta - \alpha_S.
\end{equation}
Since $x_j, y_j=0$ for $j\in S^c$, we have $\mu_j = -\alpha_j$ for $j \in S^c$. 
Appealing to the isotropy representation, we obtain the desired formula for the Morse index.
\end{proof}

\begin{proposition} \label{quiver-prop-type-a} 
The correction $R_A$ of a type $A$ critical set labelled by 
$S=T \subset [n]$ is given by
\begin{equation}
  R_A(S) = \frac{t^\lambda}{(1-t^2)^{\# S^c}} P_t( C^{(2,1)}(\gamma, \alpha', \delta))
\end{equation}
where $\lambda$ is the Morse index as computed above and $C^{(2,1)}(\gamma,\alpha',\delta)$ denotes
a type $(2,1)$ critical set with moment map level $(\gamma, \alpha', \delta)$, as computed
above.
\end{proposition}

\subsubsection{Type $B$ Critical Sets}
We have disjoint subsets $S_1, S_2 \subset [n]$ satisfying $\#S_1 \geq 3$ and $\#S_2 \geq 1$.
The coordinates $(x_i, y_i)$ satisfy the conditions $x_i = 0$ for $i \in S_1^c$ and 
$y_i = 0$ for $i \in S_2^c$.

\begin{lemma} The stabilizer of a generic point of the type $B$ critical set
is the subgroup $T_S \subseteq T_\bw$ consisting of elements of the form $t_i = s_1$ for $i \in S_1$
$t_i = s_2$ for $i \in S_2$ with no constraint on the $t_i$ for $i \in (S_1 \cup S_2)^c$. 
The map $\phi: T_S \to T_\bv$ is given by $g_1 = \diag(s_1^{-1},s_1^{-1}), g_2 = s_2^{-1}$.
The isotropy representation is given by 
\begin{equation}
   (2\#S_1+\#S_2-5) + (\#S_1+2\#S_2-2) s_1 s_2^{-1} + \sum_{i \not\in S_1 \cup S_2} ( 2 s_1 t_i^{-1} + s_2^{-1} t_i)
\end{equation}
\end{lemma}
\begin{proof} 
The form of the stabilizer follows from the definition of the type $B$ critical set,
and to obtain the isotropy representation, we simply compute
\begin{align}
  [\Rep(\CQ, \bv,\bw)] &= \sum_{i \in S^1} (2 + s_1 s_2^{-1}) 
      + \sum_{i \in S_2} (2 s_1 s_2^{-1} + 1) 
      + \sum_{i \not\in S_1 \cup S_2} (2 s_1 t_i^{-1} + s_2^{-1} t_i ) \\
  [\fg_\bv] &= 5 \\
  [E] &= 2s_1 s_2^{-1}
\end{align}
and appeal to Theorem \ref{thm-rep-ring}.
\end{proof}

\begin{lemma} The Morse index is given by
$2\lambda_0 + 4\lambda_1 + 2\lambda_2$, where $\lambda_0 = \#S_1+2\#S_2-2$ if $\mu_{S_1} > \mu_{S_2}$
and $\lambda_0 = 0$ otherwise, and
\begin{align}
  \lambda_1 &= \#\{i \not\in S_1 \cup S_2 \suchthat \mu_{S_1} + \alpha_i > 0 \},  \\
  \lambda_2 &= \#\{i \not\in S_1 \cup S_2 \suchthat \mu_{S_2} + \alpha_i < 0 \},
\end{align}
and where $\mu_{S_1} = (\gamma - \alpha_{S_1}) / (\#S_1)$ and 
$\mu_{S_2} = -(\delta+\alpha_{S_2}) / (\#S_2)$.
The equivariant topology of the type $B$ critical set
is that of a rank 2 polygon space with moment map level $\alpha' = \{\alpha_i'\}_{i \in S}$ 
and $\alpha_i' = \alpha_i + \mu_{S_1}$.
\end{lemma}
\begin{proof} 
The condition of being critical forces $\mu(x)$ to lie in $\stab(x)$, and hence
$\mu_i = \mu_{i'}$ for $i,i' \in S_1$ and similarly for $i,i' \in S_2$. 
Denote these common values by $\mu_{S_1}, \mu_{S_2}$, respectively.
Using the moment map equations, we find
\begin{align}
 (\#S_1) \mu_{S_1} &= \sum_{i \in S_1} (|x_i|^2 - \alpha_i) = \gamma - \alpha_{S_1}, \\
 (\#S_2) \mu_{S_2} &= \sum_{i \in S_1} (-|y_i|^2 - \alpha_i) = -\delta - \alpha_{S_2},
\end{align}
Since $x_j, y_j=0$ for $j\in S^c$, we have $\mu_j = -\alpha_j$ for $j \in S^c$. 
Appealing to the isotropy representation, we obtain the desired formula for the Morse index.
\end{proof}

\begin{lemma} The Poincar\'e polynomial of the polygon space  $\CP(\alpha)$  may be computed
inductively by equivariant Morse theory.
\end{lemma}
\begin{proof}
The calculation is very similar to those carried out in the remainder of this chapter,
so we will just give a sketch.
Let $(\CQ, \bd)$ be the star quiver with dimension vector $(2, 1, \dots, 1)$.
We choose the decomposition $\bd = \bv + \bw$, where $\bv = (2, 0, \dots, 0)$ 
and $\bw = (0, 1, \dots, 1)$. Then $G_\bv = U(2)$ and $G_\bw = T_\bw = (S^1)^n$.
If a subtorus $T_\beta \subset T_\bw$ stabilizes a point $[x]$, then we can choose
some homomorphism $\phi: T_\beta \to G_\bv$ so that for all $t \in T_\beta$ we have
\begin{equation}
  t \cdot x = \phi(t) \cdot x,
\end{equation}
for some representative $x$ of $[x]$. We can assume that $\phi(t) = \diag(t^{\nu_1}, t^{\nu_2})$
for characters $\nu_1, \nu_2$ of $T_\beta$. Then there are two possible cases:
$\nu_1 = \nu_2$, and $\nu_1 \neq \nu_2$.
In the first case, the stabilizer $T_\beta$ must take the form $h_i(t) = t^\nu$,
for all $i \in S$ for some subset $S \subset [n]$.
The equivariant topology of the critical set is that of the polygon space
$\CP(\alpha|_S)$, where $\alpha|_S$ denotes the vector
$ \{\alpha_i + \mu_S\}_{i \in S}$, where $\mu_S = \alpha_{S^c} / (\# S)$.
In the second case, the critical set corresponds to a choice of disjoint subsets $S_1, S_2 \subset [n]$
where $h_i(t) = t^{\nu_j}$ for $i \in S_j$. The Morse index is straightforward to compute, and
one finds that the equivariant topology of the critical set is that of a point.
By Corollary \ref{thm-equivariant-poincare}, we have
\begin{equation} \label{eqn-rnk2-polygon-cor}
  P_t(\CP(\alpha)) = \frac{P_t(\Gr(2,n))}{(1-t^2)^{n-1}} - \sum_{C} t^{\lambda_C} P_t^G(C),
\end{equation}
where the sum is over the higher critical sets $C$ of $|\mu|^2$, which we have just enumerated.
Since the terms on the right hand side can be expressed in terms of simpler polygon spaces,
we obtain a recursion relation for their Poincar\'e polynomials.
\end{proof}

\begin{remark} Klyachko \cite{Klyachko} obtained a similar recursion relation for the special moment
map level $(1,\dots,1)$. Later, Konno \cite{KonnoPolygon} obtained an explicit formula by studying
the circle action on hyperpolygons. In this case, the fixed-point set consists of
$\CP(\alpha)$ together with a finite union of projective spaces. Hence the Poincar\'e
polynomial of $\CP(\alpha)$ can be expressed as a difference of the Poincar\'e polynomial
of its hyperk\"ahler analogue and a weight sum of Poincar\'e polynomials of projective spaces.
By taking an extreme moment map level so that $\CP(\alpha)$ is empty,
Konno obtains an explicit formula for the Poincar\'e polynomial of the hyperpolygon space.
Since the latter is independent of the moment map level, substituting it into the former
yields an explicit expression for the Poincar\'e polynomial of $\CP(\alpha)$.
\end{remark}

\begin{proposition} \label{quiver-prop-type-b} 
The correction $R_B$ of a type $B$ critical set labelled by 
$S_1,S_2 \subset [n]$ is given by
\begin{equation}
  R = \frac{t^\lambda}{(1-t^2)^{1+\#S^c}} P_t(\CP(\alpha'))
\end{equation}
where $\lambda$ is the Morse index as computed above and $\CP(\alpha')$ is the polygon
space with moment map level $\alpha'$ as computed above.
\end{proposition}

\subsubsection{Type $C$ Critical Sets}
For the type $C$ critical set, we have disjoint subsets $S_1, S_2 \subset [n]$ satisfying
$\#S_1 \geq 2$, $\#S_2 \geq 1$.

\begin{lemma} The stabilizer of a generic point of the type $C$ critical set
is the subgroup $T_S \subseteq T_\bw$ consisting of elements of the form $t_i = s_j$ for $i \in S_j$,
with no constraint on $t_i$ for $i \not\in S_1 \cup S_2$.
The map $\phi: T_S \to T_\bv$ is given by $g_1 = \diag(s_1^{-1},s_2^{-1}), g_2 = s_1^{-1}$.
The isotropy representation is given by 
\begin{equation}
  \begin{split}
  2(\#S_1) + \#S_2-4 + (\#S_1+\#S_2-1) s_1^{-1} s_2 + (\#S_2-1) s_1 s_2^{-1} \\
    + \sum_{i \not\in S_1 \cup S_2} s_1 t_i^{-1} + s_2 t_i^{-1} + s_1^{-1} t_i.
  \end{split}
\end{equation}
\end{lemma}
\begin{proof} 
The form of the stabilizer follows from the definition of the type $C$ critical set,
and to obtain the isotropy representation, we simply compute
\begin{align}
  \begin{split}
  [\Rep(\CQ, \bv,\bw)] &= \sum_{i \in S^1} (2 + s_1^{-1} s_2) 
      + \sum_{i \in S_2} (1 + s_1^{-1} s_2 + s_1 s_2^{-1} ) \\
   &\ \ \   + \sum_{i \not\in S_1 \cup S_2} (s_1 t_i^{-1} + s_2 t_i^{-1} + s_1^{-1} t_i ) 
  \end{split} \\
  [\fg_\bv] &= 3 + s_1 s_2^{-1} + s_1^{-1} s_2 \\
  [E] &= 1 + s_1^{-1} s_2
\end{align}
and appeal to Theorem \ref{thm-rep-ring}.
\end{proof}

\begin{lemma} 
Let $\mu_{S_1} = (\gamma/2-\delta-\alpha_{S_1})/(\#S_1)$ and 
$\mu_{S_2} = (\gamma/2-\alpha_{S_2})/(\#S_2)$. Then the Morse index is given by
$2\lambda_0 + 2\lambda_1 + 2\#(S_1\cup S_2)^c$, where 
\begin{align}
  \lambda_0 &= \left\{ \begin{array}{ll}
    \#S_2-1,       & \mu_{S_1} > \mu_{S_2} \\ 
    \#S_1+\#S_2-1, & \mu_{S_1} < \mu_{S_2}
  \end{array} \right. \\
  \lambda_1 &= \#\{i \not\in S_1 \cup S_2 \suchthat \mu_{S_2} + \alpha_i > 0 \}, 
\end{align}
The equivariant topology of the type $C$ critical set is that of the projective space $\BP^{\#S_1-2}$.
\end{lemma}
\begin{proof} 
The condition of being critical forces $\mu(x)$ to lie in $\stab(x)$, and hence
$\mu_i = \mu_{i'}$ for $i,i' \in S_1$ and similarly for $i,i' \in S_2$. 
Denote these common values by $\mu_{S_1}, \mu_{S_2}$, respectively.
Using the moment map equations, we find
\begin{align}
 (\#S_1) \mu_{S_1} &= \sum_{i \in S_1} (|x_i|^2 - |y_i|^2 - \alpha_i) = \frac{\gamma}{2} - \delta - \alpha_{S_1}, \\
 (\#S_2) \mu_{S_2} &= \sum_{i \in S_1} (|x_i|^2 - \alpha_i) = \frac{\gamma}{2} - \alpha_{S_2}, \\
\end{align}
Since $x_j, y_j=0$ for $j\in (S_1 \cup S_2)^c$, we have $\mu_j = -\alpha_j$ for $j \in S^c$. 
Appealing to the isotropy representation, we obtain the desired formula for the Morse index.
\end{proof}

\begin{proposition} \label{quiver-prop-type-c} 
The correction $R_C$ of a type $C$ critical set labelled by 
$S,T \subset [n]$ is given by
\begin{equation}
  R_C = \frac{t^\lambda}{(1-t^2)^{1+\#(S_1\cup S_2)^c}} P_t(\BP^{\#S_1-2})
\end{equation}
where $\lambda$ is the Morse index as computed above.
\end{proposition}

\subsubsection{Type $D$ Critical Sets}
We have disjoint subsets $S_1, S_2, S_3 \subset [n]$ satisfying
$\#S_1 \geq 1$, $\#S_2 \geq 1, \#S_3 \geq 1$.

\begin{lemma} The stabilizer of a generic point of the type $D$ critical set
is the subgroup $T_S \subseteq T_\bw$ consisting of elements of the form $t_i = s_j$ for $i \in S_j$,
with no constraint on $t_i$ for $i \not\in S_1 \cup S_2 \cup S_3$.
The map $\phi: T_S \to T_\bv$ is given by $g_1 = \diag(s_1^{-1},s_2^{-1}), g_2 = s_3^{-1}$.
The isotropy representation is given by 
\begin{align}
  \begin{split}
    & (\#S_1 + \#S_2 + \#S_3-3) + (\#S_1-1) s_1^{-1} s_2 + (\#S_2-1) s_1 s_2^{-1} \\ 
    & + (\#S_1+\#S_3-1) s_1 s_3^{-1} + (\#S_2 + \#S_3-1) s_2 s_3^{-1}  \\
    & + \sum_{i \not\in S_1 \cup S_2 \cup S_3} (s_1 t_i^{-1} + s_2 t_i^{-1} + s_3^{-1} t_i )
  \end{split}
\end{align}
\end{lemma}
\begin{proof} 
The form of the stabilizer follows from the definition of the type $C$ critical set,
and to obtain the isotropy representation, we simply compute
\begin{align}
\begin{split}
  [\Rep(\CQ, \bv,\bw)] &= \sum_{i \in S^1} (1 + s_2 s_1^{-1} + s_1 s_3^{-1}) 
      + \sum_{i \in S_2} (1 + s_1 s_2^{-1} + s_2 s_3^{-1} ) \\ 
   &  + \sum_{i \in S_3} (1 + s_1 s_3^{-1} + s_2 s_3^{-1}) 
      + \sum_{i \not\in S_1 \cup S_2} (s_1 t_i^{-1} + s_2 t_i^{-1} + s_3^{-1} t_i )
\end{split} \\
  [\fg_\bv] &= 3 +s_1 s_2^{-1} + s_1^{-1} s_2 \\
  [E] &= s_1 s_3^{-1} + s_2 s_3^{-1}
\end{align}
and appeal to Theorem \ref{thm-rep-ring}.
\end{proof}

\begin{lemma} 
Let $\mu_{S_1} = (\gamma/2-\alpha_{S_1})/(\#S_1)$,
$\mu_{S_2} = (\gamma/2-\alpha_{S_2})/(\#S_2)$, and $\mu_{S_3} = -(\delta+\alpha_{S_3})/(\#S_3)$,
and assume that $\mu_{S_1} > \mu_{S_2}$.
Then the Morse index is given by $2(\#S_2-1) + 2\lambda_0 + 2\lambda_1 + 2\lambda_2+2\lambda_3$, where
\begin{align}
  \lambda_0 &= (\#S_1 + \#S_3-1) \theta(\mu_{S_1}-\mu_{S_3}) + (\#S_2 + \#S_3-1) \theta(\mu_{S_2}-\mu_{S_3}), \\
  \lambda_1 &= \#\{i \not\in S_1 \cup S_2 \cup S_3 \suchthat \mu_{S_1} + \alpha_i > 0 \},  \\
  \lambda_2 &= \#\{i \not\in S_1 \cup S_2 \cup S_3 \suchthat \mu_{S_2} + \alpha_i > 0 \}, \\
  \lambda_3 &= \#\{i \not\in S_1 \cup S_2 \cup S_3 \suchthat \mu_{S_3} + \alpha_i < 0 \},
\end{align}
where $\theta$ denotes the Heaviside step function.
The equivariant topology of the type $D$ critical set is that of a point.
\end{lemma}
\begin{proof} 
The condition of being critical forces $\mu(x)$ to lie in $\stab(x)$, and hence
$\mu_i = \mu_{i'}$ for $i,i' \in S_1$ and similarly for $i,i' \in S_2$. 
Denote these common values by $\mu_{S_1}, \mu_{S_2}$, respectively.
Using the moment map equations, we find
\begin{align}
 (\#S_1) \mu_{S_1} &= \sum_{i \in S_1} (|x_i|^2 - \alpha_i) = \frac{\gamma}{2} - \alpha_{S_1}, \\
 (\#S_2) \mu_{S_2} &= \sum_{i \in S_1} (|x_i|^2 - \alpha_i) = \frac{\gamma}{2} - \alpha_{S_2}, \\
 (\#S_3) \mu_{S_3} &= \sum_{i \in S_1} (-|y_i|^2 - \alpha_i) = -\delta - \alpha_{S_3},
\end{align}
Appealing to the isotropy representation, we obtain the desired formula for the Morse index.
\end{proof}

\begin{proposition} \label{quiver-prop-type-d} 
The correction $R_D$ of a type $D$ critical set labelled by subsets $S_1, S_2, S_3 \subset [n]$
is equal to
\begin{equation}
  R_D = \frac{t^\lambda}{(1-t^2)^{2+\#(S_1\cup S_2 \cup S_3)^c}}
\end{equation}
where $\lambda$ is the Morse index as computed above.
\end{proposition}

\subsection{The Type $(1,1,1)$ Critical Sets}

\begin{proposition} Let $S_1, S_2, S_3$ be a decomposition of $[n]$ indexing a critical
set of type $(1,1,1)$. Then these subsets must satisfy
\begin{inparaenum}[(i)]
  \item $\alpha_{S_2 \cup S_3} > 2 \alpha_{S_1}$,
  \item $2 \alpha_{S_3} > \alpha_{S_1 \cup S_2}$,
  \item $\#S_1 \geq 2$,
  \item $\#S_2 \geq 2$, and (v) $\#S_3 \geq 1$.
\end{inparaenum}
The isotropy representation is given by
\begin{equation} 
  (\#S_1+\#S_2-4)(1+t) + (n-3)(t^{-1} + t^2) + (\#S_3-1)(t^{-2} + t^3). 
\end{equation}
The dimension of the critical set is $\#S_1+\#S_2-4$ and its Morse index
is $2(n+\#S_3-4)$. It is isomorphic to $\BP^{\#S_1-2} \times \BP^{\#S_2-2}$
\end{proposition}
\begin{proof}
A fixed-point of this type corresponds to the map $\phi: S^1 \to G_\bv \iso U(3) \times (S^1)^n$
given by $g(t) = \diag(1, t^{-1}, t^{-2})$ and
\begin{equation}
  h_i = \left\{ \begin{array}{ll}
    1, & i \in S_1 \\
    t^{-1}, & i \in S_2 \\
    t^{-2}, & i \in S_3
    \end{array} \right.
\end{equation}
To compute the isotropy representation, we compute
\begin{align}
  [T_x \Rep(Q)] &= n + (\#S_1+\#S_2)t + (\#S_1)t^2 + (\#S_2+\#S_3)t^{-1} + (\#S_3)t^{-2} \\
 [\mathfrak{pg}_\bv] &= n+2 + 2t + 2t^{-1} + t^2 + t^{-2} \\
 [E] &= 2 + (n+2)t + 2t^2 + t^3 + t^{-1}
\end{align}
and appeal to Theorem \ref{thm-rep-ring}. The isomorphism with the product of projective
spaces is given as follows. First, we consider the map to $\BP^{\#S_1-1} \times \BP^{\#S_2-1}$
given by
\begin{equation}
  [x_i, y_i] \mapsto [ x_i y_i: i \in S_1] \times [ x_j y_j: j \in S_2].
\end{equation}
This is easily seen to be a well-defined embedding (at least for generic $\alpha$).
However, the complex moment map equation gives the relations $\sum_{i \in S_1} x_i y_i =0$
and $\sum_{i \in S_2} x_i y_i = 0$, so the image is the subvariety defined by these equations,
which is isomorphic to $\BP^{\#S_1-2} \times \BP^{\#S_2-2}$.
\end{proof}

\subsection{The Poincare Polynomial}

The preceding sections described the computation of the Morse indices of the $S^1$-fixed
point sets, as well as inductive methods to compute the Poincar\'e polynomials
of types $(2,1)$ and $(1,1,1)$. Since we can ensure that only these types appear by choosing
the moment map level $\alpha$ to be extreme, this gives an explicit inductive procedure
to compute $P_t(\FM_\alpha(n))$ by Morse theory.\footnote{Recall that the Betti numbers
of $\FM_\alpha(n)$ do not depend on $\alpha$, as long as $\alpha$ is generic.}
We have implemented the this procedure in the computer software Sage.
For small values of $n$, we summarize the results in the following theorem.
\begin{theorem} \label{thm-rank-3-betti-numbers}
For small values of $n$, the Poincar\'e polynomials of the rank 3 star quiver varieties $\FM_\alpha(n)$
are the following:
\begin{align*}
P_t(\FM_\alpha(4)) &= 1 \\ 
P_t(\FM_\alpha(5)) &= 1 + 5t^2 + 11t^4 \\ 
P_t(\FM_\alpha(6)) &= 1 + 6t^2 + 22 t^4 + 51 t^6 + 66 t^8 \\ 
P_t(\FM_\alpha(7)) &= 1 + 7t^2 + 29 t^4 + 85 t^6 + 190 t^8 + 308 t^{10} + 302 t^{12} \\ 
P_t(\FM_\alpha(8)) &= 1 + 8t^2 + 37 t^4 + 121 t^6 + 311 t^8 + 653 t^{10} + 1115 t^{12} + 1450 t^{14} + 1191 t^{16}
\end{align*}
\end{theorem}

\begin{remark} Using the Weil conjectures, Hausel produced an
explicit generating function for the Poincar\'e polynomials of Nakajima quiver
varieties \cite{HauselBetti}. As a consistency check, we have verified that our calculation
is consistent with Hausel's generating function for small values of $n$.
Note that unlike the method of this chapter, Hausel's method does not give any direct information about the $S^1$-fixed point set.
\end{remark}

\chapter{The Residue Formula Revisited} \label{ch-residue}

\section{Generating Functions for Intersection Pairings}

\subsection{Cogenerators of Rings}
Let $k$ be a field of characteristic $0$,\footnote{For our purposes we will only need to consider $\BQ, \BR$, and $\BC$.}
$V$ a graded $k$-vector space, and $V^\ast$ its dual. 
Let $R = k[V] \iso \Sym\ V^\ast$ and $R^\vee = k[V^\ast] \iso \Sym\ V$. Note that $R$ has
a natural grading $R = \bigoplus_{d \geq 0} R^d$, where $R^d := \Sym^d~ V^\ast$ consists of
homogeneous polynomials of degree $d$.
The natural pairing
between $V$ and $V^\ast$ gives $R^\vee$ the structure of an $R$-module, with $\Sym^d~V^\ast$
acting on $R^\vee$ as homogeneous linear differential operators of degree $d$.

\begin{lemma} \label{lemma-fourier}
Suppose that $I \subseteq R$ is a homogeneous ideal such that $S = R / I$ is 
finite-dimensional over $k$. Then there exists a finitely-generated $R$-submodule 
$S^\vee \subset R^\vee$ such that $I = \ann_R(S^\vee)$. Consequently, the action of $R$ on $S^\vee$ descends
to an action of $S$ on $S^\vee$.
\end{lemma}
\begin{proof} We define $S^\vee := \Hom_R(R/I, R^\vee)$. Since an element
$\phi \in \Hom_R(R/I, R^\vee)$ is uniquely determined by its image $\phi(1)$, $S^\vee$ is
naturally a sub-module of $R^\vee$. An element $r \in R^\vee$ is the image of $1$ under
$\phi \in \Hom_R(S, R^\vee)$ if and only if $I \cdot r = 0$. Hence
  \begin{equation}
    S^\vee \iso \{ r \in R^\vee \ | \ I \cdot r = 0 \}.
  \end{equation}
By construction, $I \subseteq \ann_R(S^\vee)$. To see that this is in fact an equality,
simply note that for any degree $d$ the vector spaces $R^d$ and $(R^\vee)^d$ are naturally dual.
Since $I$ is homogeneous, we have that $(S^\vee)^d$ is the annihilator of $I^d$ under
the dual pairing, and conversely.
\end{proof}

\begin{remark} The conclusion of this lemma is definitely not true if we drop the 
hypothesis that $I$ is homogeneous. For example, if $R = k[u]$ and $I = \ev{u-1}$ then
$S = R/I \iso k$, but $S^\vee = 0$, since the differential equation
\begin{equation} \left( \frac{\partial}{\partial u} - 1 \right)f(u) = 0 \end{equation}
has no non-trivial polynomial solutions. Hence $\ann(S^\vee) = R \neq I$.
\end{remark}

\begin{definition} Let $S = R/I$ be as above. We call $S^\vee = \Hom_R(S, R^\vee)$ the
\emph{Fourier dual} of $S$.
 A \emph{cogenerator of $S$} is an 
element $f \in S^\vee$. A set $\{f_1, \dots, f_n\}$ of cogenerators is said to be 
\emph{complete} if this set generates $S^\vee$ as an $R$-module.
\end{definition}

\begin{definition} A finite-dimensional graded $k$-algebra is called \emph{level}
if every non-zero element divides a non-zero element of top degree. It is said to
satisfy \emph{Poincar\'e duality} if it is both level and one-dimensional in top degree.
\end{definition}

\begin{remark} A connected compact oriented manifold $M$ satisfies Poincar\'e duality, which
implies that the ring $H^\ast(M)$ satisfies Poincar\'e duality as defined above.
In the non-compact setting, levelness can be a suitable replacement for Poincar\'e duality
in many arguments. This property is known to hold for hypertoric varieties, 
and is conjectured to hold for a large class of holomorphic symplectic quotients \cite{ProudfootGIT}.
\end{remark}

\begin{proposition} Let $I \subseteq R$ be as above. Then the quotient ring $S = R/I$ is 
level if and only if the dual $R$-module $S^\vee$ is generated in top degree, and $S$
satisfies Poincar\'e duality if and only if $S^\vee$ is generated by a single polynomial
of top degree.
\end{proposition}
\begin{proof} 
Let $S_0^\vee \subseteq S^\vee$ be the $R$-submodule of $S^\vee$ generated by $(S^\vee)^n$,
where $n$ is the top non-vanishing degree of $S$.
Then $S^\vee$ is generated in top degree if and only if $S^\vee = S_0^\vee$. By Lemma \ref{lemma-fourier}
 this is true if and only if $\ann(S^\vee) = \ann(S_0^\vee)$. Note that
we always have the containment $\ann(S^\vee) \subseteq \ann(S_0^\vee)$, so $S^\vee$ is 
not generated in top degree if and only if there exists $f \in R, f \not\in I$ such
that $f \cdot S_0^\vee = 0$. If there is such an $f$, then for all $g$ of complementary
degree we have
\begin{equation} (fg) \cdot (S^\vee)^n = (fg) \cdot (S_0^\vee) = g \cdot \left( f \cdot (S_0^\vee) \right) = 0. \end{equation}
By the lemma, the image of $f$ in $S$ does not divide any non-zero element of
top-degree, hence $S$ is not level.
\end{proof}

In what follows, $R$ will be the torus-equivariant cohomology of a point and $I$ is the kernel
of the Kirwan map. The surprising result is that while the cohomology ring
$S$ might have a very complicated presentation in terms of generators and relations,
in many cases we can compute generators for the Fourier dual module rather easily.
Hence, Fourier duality gives a convenient way to encode the entire cohomology ring
as a single polynomial.

\begin{example} \label{ex-pn}
  Let $R = \BQ[u]$ and $I = \ev{u^{n+1}}$ with $\deg u = 2$. Then 
  $R^\vee = \BQ[\lambda]$ where $\lambda$ is dual to $u$ and $u$ acts on $R^\vee$ as
  $d/d\lambda$. The dual module $S^\vee$ consists of
  all polynomials in $\lambda$ of total degree less than or equal to $2n$, where
  $\deg \lambda = 2$. Hence $S^\vee$ is generated as an $R$-module by $\lambda^n$.
  It is well-known that $S$ is isomorphic to the cohomology ring of $\BP^n$, and
  we will see that its Jeffrey-Kirwan residue is (up to an overall constant) given by $\lambda^n$.
\end{example}



We will give many more examples in Section \S\ref{sec-residue-examples}.

Let us record two easy lemmas, which we will need later in the chapter.

\begin{lemma} \label{lemma-residue-quotient-u}
Let $I$ be a homogeneous ideal of $R[u]$ such that $S = R[u]/I$ is finite-dimensional.
Then the Fourier dual of $S / \ev{u}$ is given by
\begin{equation}
  \{ f \in S^\vee \suchthat \frac{\partial f}{\partial u} = 0 \}.
\end{equation}
\end{lemma}
\begin{proof} We have that $S / \ev{u} = R[u] / \ev{I, u}$, so its Fourier dual consists
of those elements of $R[u]$ which are simultaneously annihilated by $I$ and $u$, i.e.
those elements of $S^\vee$ which do not depend on $u$.
\end{proof}

\begin{lemma} \label{lemma-residue-varpi}
Let $\varpi \in R$ be homogeneous, let $I$ be a homogeneous ideal of $R$ such that $S=R/I$
is finite dimensional, and suppose that $\{f_i\}_{i\in I}$ is a complete set of generators of
$S^\vee$. Then a complete set of generators of $(S/\ann_S(\varpi))^\vee$ is given by
$\{\varpi \cdot f_i\}_{i \in I}$.
\end{lemma}
\begin{proof} If $a \in \ann_S(\varpi)$, then $a \cdot \varpi \in I$, so that
\begin{equation}
  a \cdot (\varpi \cdot f_i) = (a \varpi) \cdot f_i = 0,
\end{equation}
and hence $\ann_R( D_\varpi f_i ) \subset \ev{I, \ann_S(\varpi)}$.
Conversely, suppose that $r \in R$ such that $[r] \in S / \ann_S(\varpi)$ is non-zero.
Then $r \varpi \not\in I$, and hence $(r\varpi) \cdot f_i \neq 0$ for some $i \in I$.
Hence $\bigcap_{i \in I} \ann_R( \varpi \cdot f_i) = \ev{I, \ann_S(\varpi)}$.
\end{proof}

\subsection{Generating Functions}
Let $G$ be a compact Lie group and $T \subset G$ a maximal torus. 
Denote the Weyl group $N(T)/T$ by $W$. Let $R = H^\ast(BT) \iso \BQ[u_1, \ldots, u_r]$,
with $\deg X_i = 2$ and $r$ the rank of $G$. Then $H^\ast(BG) \iso R^W$ and we regard
elements of $H^\ast(BG)$ as Weyl-invariant polynomials in the $u_i$. Then 
$R^\vee = \BQ[\lambda_1, \ldots, \lambda_r]$ for dual variables $\lambda_i$, and for
any polynomial $f = \sum a_I x^I$ we have the Fourier dual differential operator
\begin{equation}
  D_f := \sum a_I \left( -i\frac{\partial}{\partial \lambda} \right)^I.
\end{equation}

Suppose that $G$ acts linearly on $X = \BC^N$ and assume that $X$ has a $G$-invariant
Hermitian inner product. Then for any regular central $\alpha \in \fg^\ast$ the 
symplectic reduction $\FX_\alpha := X \reda{\alpha} G$ is a symplectic orbifold.
We assume throughout this section that the moment map $\mu: X \to \fg^\ast$ is proper
so that $\FX_\alpha$ is compact.
\begin{definition}
We define the \emph{generating function $Z: \ft^\ast \to \BQ$} by
  \begin{equation}
    Z(\lambda) = \frac{1}{|W|} \int_{\FX_\alpha}  
      \kappa\left(\sum_{s \in W} e^{i(\lambda-\alpha)(s \cdot u)}\right) e^\omega,
  \end{equation}
where $\kappa: H^\ast(BG) \to H^\ast(\FX_\alpha)$ is the Kirwan map.
\end{definition}
Note that although $Z$ is defined as a Weyl-invariant formal power series in $\lambda$, 
due to the finite-dimensionality of $\FX_\alpha$ it is actually a polynomial. 
It is a generating function in the following sense.
\begin{lemma} \label{lemma-generating}
For any Weyl-invariant polynomial $f \in R^W$ we have
  \begin{equation}
    D_f Z(\alpha) = \int_{\FX_\alpha} \kappa(f) e^\omega.
  \end{equation}
Hence the intersection pairings in $\FX_\alpha$ completely determine $Z$, and conversely.
\end{lemma}
\begin{proof} 
Simply note that for any polynomial $f(u)$ we have
  \begin{equation}
    D_f e^{i(\lambda-\alpha)(u)} = f(u) e^{i(\lambda-\alpha)(u)}.
  \end{equation}
Now assume that $f$ is Weyl-invariant. Average over the Weyl group, apply the Kirwan map, 
multiply both sides by $e^\omega$, and integrate over $\FX_\alpha$ to obtain
  \begin{equation}
    D_f Z(\lambda) = \frac{1}{|W|} \int_{\FX_\alpha} \kappa 
      \left( f(u) \sum_{s \in W} e^{i(\lambda-\alpha)(s \cdot u)} \right) e^\omega.
  \end{equation}
Now evaluate at $\lambda=\alpha$ to obtain the desired result. This shows that the 
intersection pairings are the Taylor coefficients of $Z$ about $\lambda=\alpha$, and
since $Z$ is a polynomial, its Taylor coefficients determine it completely.
\end{proof}

\begin{theorem} $f \in \ker \kappa$ if and only if $D_f Z = 0$. Hence the $R^W$-module
generated by $Z$ has annihilator equal to $\ker \kappa$, and therefore it is the
Fourier dual of $H^\ast(\FX_\alpha)$.
\end{theorem}
\begin{proof} Without loss of generality assume that $f$ is homogeneous. Then by
Kirwan surjectivity and Poincar\'e duality, $f \in \ker \kappa$ if and only if
for all $g \in R^W$,
  \begin{equation}
    \int_{\FX_\alpha} \kappa(fg) e^\omega = 0.
  \end{equation}
By the above lemma, $\int_{\FX_\alpha} \kappa(fg) e^\omega = 0$ if any only if 
$(D_f D_g Z)(\alpha) = 0$ for all $g$. Then the Taylor expansion of $D_f Z$ about 
$\lambda=\alpha$ is identically $0$, hence $D_f Z = 0$ identically.
\end{proof}

\section{Jeffrey-Kirwan Residues} \label{sec-residues}

\subsection{The Residue Operation}
We desire an effective method to compute the generating function $Z$ associated to
any compact symplectic quotient $\FX_\alpha$ as in the preceding section. 
We begin by recalling the residue operation:
\begin{proposition}\cite[Proposition 3.2]{JeffreyKirwan96}
Let $\Lambda \subset \ft$ be a non-empty open cone and suppose that
$\beta_1, \ldots, \beta_N \in \ft^\ast$ all lie in the dual cone $\Lambda^\ast$.
Suppose that $\lambda \in \ft^\ast$ does not lie in any cone of dimension at most
$n-1$ spanned by a subset of $\{\beta_1, \ldots, \beta_N\}$. Let $\{u_1, \ldots, u_n\}$
be any system of coordinates on $\ft$ and let $du = du_1 \wedge \cdots \wedge du_n$ be 
the associated volume form. Then there exists a residue operation $\Res^\Lambda$
defined on meromorphic differential forms of the form
\begin{equation}
  h(u) = \frac{q(u) e^{i\lambda(u)}}{\prod_{j=1}^N \beta_j(u)} du
\end{equation}
where $q(u)$ is a polynomial. The operation $\Res^\Lambda$ is linear in its
argument and is characterized uniquely by the following properties:
\begin{enumerate}[(i)]
  \item If $\{\beta_1, \ldots, \beta_N\}$ does not span $\ft$ as a vector space then
    \begin{equation}\Res^\Lambda\left(\frac{u^J e^{i\lambda(u)}}{\prod_{j=1}^N \beta_j(u)} du\right) = 0\end{equation}
  \item For any multi-index $J$, we have
    \begin{equation} \Res^\Lambda\left(\frac{u^J e^{i\lambda(u)}}{\prod_{j=1}^N \beta_j(u)} du\right)
       =  \sum_{m\geq 0} \lim_{s \to 0^+} 
       \Res^\Lambda\left(\frac{u^J (i(\lambda(u))^m e^{is\lambda(u)}}
         {m! \prod_{j=1}^N \beta_j(u)} du\right). \end{equation}
  \item The limit 
    \begin{equation}\lim_{s \to 0^+} 
       \Res^\Lambda\left(\frac{u^J e^{is\lambda(u)}}{\prod_{j=1}^N \beta_j(u)} du\right) \end{equation}
    is zero unless $N - |J| = n$.
  \item If $N = n$ and $\{\beta_1, \ldots, \beta_n\}$ spans $\ft^\ast$ as a vector space then 
      \begin{equation}\Res^\Lambda\left(\frac{u^J e^{i\lambda(u)}}{\prod_{j=1}^n \beta_j(u)} du\right) = 0\end{equation}
      unless $\lambda$ is in the cone spanned by $\{\beta_1, \ldots, \beta_n\}$. If $\lambda$
      is in this cone, then the residue is equal to $\bar{\beta}^{-1}$, where $\bar{\beta}$
      is the determinant of an $n \times n$ matrix whose columns are the coordinates of
      $\beta_1, \ldots, \beta_n$ with respect to any orthonormal basis defining the same
      orientation as $\beta_1, \ldots, \beta_n$.
\end{enumerate}
\end{proposition}
Although the above properties determine the residue operation completely, they can be
difficult to work with. However, in the one-dimensional case it is essentially the usual
residue operation from complex analysis. Consider meromorphic functions of $z \in \BC$ which 
take the form
\begin{equation} f(z) = \sum_{j=1}^m g_j(z) e^{i\lambda_j z} \end{equation}
where $g_j(z)$ are rational functions and $\lambda_j \in \BR\setminus\{0\}$. We introduce
an operation $\Res_z^+$ defined by
\begin{equation} 
 \Res_z^+ f(z) dz = \sum_{\lambda_j \geq 0} \sum_{b \in \BC} 
   \Res\left(g_j(z) e^{i \lambda_j z} dz; z = b \right).
\end{equation}
The following proposition reduces the residue operation to a series of iterated one-dimensional
residues.
\begin{proposition} \label{prop-iterated-residue} \cite[Proposition 3.4]{JeffreyKirwan96}
Let 
\begin{equation} h(u) = \frac{q(u) e^{i\lambda(u)}}{\prod_{j=1}^N \beta_j(u)} \end{equation}
where $q(u)$ is a polynomial, and $\lambda$ is not in any proper subspace spanned by
a subset of $\{\beta_1, \ldots, \beta_N\}$. For a generic choice of coordinate system
$\{u_1, \ldots, u_n\}$ on $\ft$ for which $(0,\ldots,0,1) \in \Lambda$ we have
\begin{equation} \Res^\Lambda\left( h(u) du \right) 
    = \Delta \Res^+_{u_1} \cdots \Res^+_{u_n}
       h(u) du_1 \wedge \cdots \wedge du_n \end{equation}
where the variables $u_1, \ldots, u_{n-1}$ are held constant while calculating 
$\Res^+_{u_n}$, and $\Delta$ is the determinant of any $n \times n$ matrix
whose columns are the coordinates of an orthonormal basis of $\ft$ defining the
same orientation as the chosen coordinate system.
\end{proposition}

\begin{example} \label{eg-res-pn} We compute the residue $e^{i\lambda u}/u^{n+1}$
with respect to the dual cone $\Lambda^\ast = \{ \lambda | \lambda > 0 \}$
using Proposition \ref{prop-iterated-residue}.
In this case, it is zero if $\lambda < 0$ and otherwise it is equal to the usual
residue from elementary complex analysis. Hence we have
\begin{equation}
 \Res^\Lambda\left( \frac{e^{i\lambda u}}{u^{n+1}} \right) =
 \theta(\lambda) \Res \left(\frac{e^{i\lambda u}}{u^{n+1}}; u=0\right) = \theta(\lambda) \frac{(-i\lambda)^n}{n!}, \end{equation}
where $\theta(\lambda)$ is the Heaviside step function, defined to be
equal to $1$ for $\lambda > 0$ and $0$ for $\lambda < 0$.
\end{example}

\begin{example} \label{eg-res-blowup} We compute the residue
\begin{equation} \Res^\Lambda \left( \frac{e^{i\lambda_1 u_1 +i\lambda_2 u_2}}{u_1^2 u_2 (u_2-u_1)} \right) \end{equation}
where $\Lambda$ is the dual cone 
$\Lambda^\ast = \{(\lambda_1, \lambda_2) \ | \ 0 < \lambda_2, \lambda_1 < \lambda_2 \}$.
Since in this coordinate system $(0,1) \in \Lambda$,  we may evaluate it via 
Proposition \ref{prop-iterated-residue}. We first take the residue with respect to $u_2$,
keeping $u_1$ fixed. There are two poles: $u_2 = 0$ and $u_2 = u_1$. In the first case
we obtain $-e^{i\lambda_1 u_1} / u_1^3$, and in the second $e^{i(\lambda_1+\lambda_2)u_1} / u_1^3$.
Taking the residue at $u_1 = 0$ in each case, we obtain
\begin{equation} \Res^\Lambda \left( \frac{e^{i\lambda_1 u_1 +i\lambda_2 u_2}}{u_1^2 u_2 (u_2-u_1)} \right) 
   = \theta(\lambda_1) \theta(\lambda_2) \frac{\lambda_1^2}{2}  
    - \theta(\lambda_1+\lambda_2) \theta(\lambda_2) \frac{(\lambda_1+\lambda_2)^2}{2}, \end{equation}
where the Heaviside factors are due to the fact that in each iterated one-dimensional
residue, only those poles for which $\lambda_j > 0$ contribute.
\end{example}

\begin{remark} In the examples above, the chosen coordinate systems do not 
satisfy the genericity assumption of Proposition \ref{prop-iterated-residue}. However,
this assumption is needed only to avoid ambiguous Heaviside factors of the form
$\theta(0)$. As long as we assume that $\lambda$ is sufficiently generic, we are free
to work in non-generic coordinate systems. However, the assumption that $(0,\ldots,1) \in \Lambda$
is crucial.
\end{remark}

We shall need the following lemma in the next section.
\begin{lemma} \label{lemma-dual-cone}
If $\lambda$ is not in the dual cone $\Lambda^\ast$ then
\begin{equation} \Res^\Lambda \left( \frac{u^J e^{i\lambda(u)}}{\prod_{j=1}^N \beta_j(u)} \right) = 0 \end{equation}
\end{lemma}
\begin{proof} If $\{\beta_1, \ldots, \beta_N\}$ do not span $\ft^\ast$ as a vector
space, then by property (i) the residue is zero and there is nothing to show. 
Otherwise, by properties (ii) and (iii) it suffices to consider residues of the form
\begin{equation} \Res^\Lambda \left( \frac{u^J e^{i\lambda(u)}}{\prod_{j=1}^N \beta_j(u)} \right) \end{equation}
with $|J| = N - n$. Since $\{\beta_1, \ldots, \beta_N\}$ span $\ft^\ast$ as a vector space,
we can write
\begin{equation} u_k = \sum_{j=1}^N c_{kj} \beta_j(u) \end{equation}
for some constants $c_{kj}$, for each $k = 1, \ldots, n$. Hence
\begin{equation} \frac{u_k}{\prod_{j=1}^N \beta_j(x)} = \sum_{j} \frac{c_{kj}}{\prod_{j\neq k} \beta_j(u)}. \end{equation}
By induction on $|J|$, $u^J / \prod_{j=1}^N \beta_j(u)$ can be written as 
linear combination of terms of the form $\prod_{j \in I} \beta_j(u)^{-1}$ where
$|I| = n$. By property (iv), the corresponding residues are $0$ if $\lambda$ is not
in the dual cone $\Lambda^\ast$.
\end{proof}

\subsection{The Linear Residue Formula}
We now recall the Jeffrey-Kirwan residue formula. Let $G$ be a compact group acting on a 
compact symplectic manifold $X$ with moment map $\mu$. Let $\alpha \in \fg^\ast$ be regular
central and let $\FX_\alpha = X \reda{\alpha} G$ be the symplectic reduction. Let $T \subset G$
be a maximal torus, and let $X^T$ be the $T$-fixed point set. Let $\mathcal{F}$ denote the
collection of connected components of $X^T$. For each component $F \in \mathcal{F}$,
we have the $T$-equivariant Euler class $e_F \in H^\ast_T(F)$ of the normal bundle. By the splitting
principle, we can assume that each $e_F$ factors as a product
\begin{equation}
  e_F = \prod_{i=1}^{n_F} (c^F_i + \beta^F_i)
\end{equation}
where the $c^F_i$ are the Chern classes and the $\beta^F_i$ the weights of the decomposition
of the pullback of the normal bundle into line bundles. Let $\mathcal{B}$ be the collection of all weights
appearing in this decomposition, for all components $F \in \mathcal{F}$. Let $\Lambda \subset \ft$
be any cone complementary to the hyperplanes $\{ \beta(u) = 0 \ | \ \beta \in \mathcal{B} \}$.

\begin{theorem} \cite{JeffreyKirwan95,JeffreyKirwan96} \label{thm-residue-formula}
Let $G$ act in a Hamiltonian fashion on a compact symplectic manifold
$(X, \omega)$, and let $f \in H_G^\ast$. Then
  \begin{equation}
    \int_{\FX_\alpha} \kappa(f) e^\omega = n_0 c_G 
      \Res^\Lambda \left( \varpi(u)^2 \sum_{F \in \mathcal{F}}  r^f_F du \right),
  \end{equation}
where $\varpi$ is the product of the positive roots of $G$,
$n_0$ is the order of the stabilizer in $G$ of a generic point of $\mu^{-1}(\alpha)$,
and the constant $c_G$ is defined by
\begin{equation}
 c_G = \frac{i^n}{|W| \mathrm{vol}(T)};
\end{equation}
$u$ is a coordinate on $\ft \otimes \BC$ and $\mathcal{F}$ is the set of components of the fixed
point set $X^T$ of the action of $T$ on $M$. If $F$ is one of these components the 
meromorphic function $r^f_F$ on $\ft \otimes \BC$ is defined by
\begin{equation}
  r^f_F(u) = e^{i\left(\mu_T(F)-\alpha\right)(u)} \int_F 
   \frac{i^\ast_F \left(\kappa(f)(u) e^\omega \right)}{e_F(u)}
\end{equation}
where $i_F: F \to M$ is the inclusion and $e_F$ the $T$-equivariant Euler class of the
normal bundle to $F$ in $M$. The individual terms of the sum on the right hand side
depend on the choice of $\Lambda$, but the overall sum does not.
\end{theorem}

We would like to apply the residue formula to quotients of the form $\FX_\alpha = X \reda{\alpha} G$,
where $X = \BC^N$. Unfortunately, this does not satisfy the hypotheses of Theorem \ref{thm-residue-formula}
as $X$ is not compact. We will need the following.
\begin{theorem} \label{thm-linear-residue-formula}
Let $G$ be a compact group with maximal torus $T$, and suppose $G$ acts unitarily on 
$X = \BC^N$. Let $\alpha \in \fg^\ast$ be regular central and let $\FX_\alpha = X \reda{\alpha} G$. 
Let $\{\beta_1, \ldots, \beta_N\}$ be the multiset of weights of the action
of $T$ on $X$ (counted with multiplicity). Suppose further that there is some $\xi \in \ft$ 
such that the $\xi$-component of the moment map is proper and bounded below. 
Then for all $f \in H_G^\ast$, we have
  \begin{equation} \label{eqn-linear-residue-formula}
    \int_{\FX_\alpha}\kappa(f) e^\omega = n_0 c_G 
      \Res^\Lambda \left( \frac{\varpi^2(u) f(u) e^{-i\alpha(u)}}{\prod_{j=1}^N \beta_j(u)} \right), 
  \end{equation}
where $n_0$ and $c_G$ are as in Theorem \ref{thm-residue-formula}, and the cone
$\Lambda$ is defined by
\begin{equation}
  \Lambda = \{ u \in \ft \ | \ \beta_j(u) < 0, j = 1, \ldots, N\}.
\end{equation}
\end{theorem}
\begin{proof} The moment map for the $T$ action is given by
\begin{equation} \mu_T(x) = \frac{1}{2} \sum_{j=1}^N \beta_j |x_j|^2, \end{equation}
hence the $\xi$-component is
\begin{equation} \mu_T^\xi(x) = \frac{1}{2} \sum_{j=1}^N \beta(\xi) |x_\beta|^2. \end{equation}
Thus $\mu_T^\xi$ is proper and bounded below if and only if $\beta_j(\xi) > 0$ for
all $j=1,\ldots,N$. This shows that $-\xi \in \Lambda$ hence $\Lambda$ is non-empty.
Note that the dual cone $\Lambda^\ast$ contains $\{-\beta_1, \ldots, -\beta_N\}$.
To apply the residue formula, we compactify by taking a symplectic cut (see \S\ref{sec-cuts-modification}).
Let $\sigma > 0$ to be determined. We set
\begin{equation}
  X_\sigma := (X \oplus \BC) \reda{\sigma} S^1
\end{equation}
where the $S^1$ action on $X \oplus \BC$ is given by
\begin{equation} e^{i\theta} \cdot (x \oplus z) = (e^{i\theta} x) \oplus (e^{i\theta} z) \end{equation}
As a variety we have $X_\sigma \iso \BP^{m}$, but the K\"ahler metric and
hence symplectic form depend on the parameter $\sigma$. Any point $[x,z] \in X_\sigma$
can be represented by $x\oplus z \in X \oplus \BC$ satisfying
\begin{equation} \label{eqn-cut}
 \frac{1}{2} |x|^2 + \frac{1}{2} |z|^2 = \sigma.
\end{equation}
Note that this implies that for any $[x,z] \in X_\sigma$ a representative $x\oplus z$
satisfies
\begin{equation}
  \frac{1}{2} |x|^2 \leq \sigma.
\end{equation}
Moreover, $G$ acts naturally on $X_\sigma$ and the moment map for the action is given
by 
\begin{equation}
  \mu_G([x,z]) = \mu_G(x).
\end{equation}
By Lemma \ref{lemma-cut}, for $\sigma \gg \alpha$, we may
identify $X \reda{\alpha} G \iso X_\sigma \reda{\alpha} G$ and hence it
suffices to compute the intersection pairings of $X_\sigma \reda{\alpha} G$
for large $\sigma$. Since $X_\sigma$ is compact we may apply the usual
residue formula, Theorem \ref{thm-residue-formula}. We have to identify the
components of $X_\sigma^T$. In addition to the isolated fixed point 
$[0,\sqrt{2\sigma}] \in X_\sigma$ corresponding to the origin in $X$, 
the cutting procedure introduces additional fixed-points at infinity. These 
are all of the form
\begin{equation}
  F_\beta = \{ [x, 0] \ | \ x \in X(\beta) \}
\end{equation}
where $X(\beta) = \{ x \in X \ | \ t \cdot x = \beta(t) x \ \forall \ t \}$.
On a component $F_\beta$, we have
\begin{equation}
  \mu_T(F_\beta) = \frac{1}{2} \beta |x_\beta|^2 = \frac{1}{2} \sigma
\end{equation}
where we used equation (\ref{eqn-cut}) and the fact that the points of $F_\beta$
have the form $[x_\beta, 0]$. Now consider
\begin{equation}
  \ev{\mu_T(F_\beta) - \alpha, \xi} = \frac{1}{2} \sigma \beta(\xi) - \alpha(\xi).
\end{equation}
If $\sigma > 2 \alpha(\xi) / \beta(\xi)$ then the above quantity is strictly 
positive, and hence $\mu_T(F_\beta) - \alpha$ does not lie in the cone spanned
by $\{-\beta_1, \ldots, -\beta_N\}$. By Lemma \ref{lemma-dual-cone},
the residue of the term corresponding to $F_\beta$ is necessarily zero. Hence
none of the fixed-points at infinity contribute to the residue formula, and
the only remaining term is the contribution from the origin, which is given exactly
by equation (\ref{eqn-linear-residue-formula}).
\end{proof}

As stated, Theorem \ref{thm-linear-residue-formula} provides a method to calculate the
intersection pairings directly, but does not immediately yield the generating function for the intersection pairings.
To obtain this, we define a residue operation $\Res_\alpha$ as follows. Given a meromorphic
differential form $h(u) e^{i\lambda(u)} du$, the residue
$\Res^\Lambda\left( h(u) e^{i\lambda(u)} du \right)$ is a piecewise polynomial function on $\ft^\ast$.
We define $\Res_\alpha\left( h(u) e^{i\lambda(u)} du \right)$ to be the unique polynomial in $\lambda$
which is equal to $\Res^\Lambda \left(h(u) e^{i\lambda(u)} \right)$ when $\lambda$ is in the same
chamber as $\alpha$.

\begin{corollary} Suppose $\alpha \in \fg^\ast$ is regular central.
Then the generating function $Z(\lambda)$ of $\FX_\alpha$ is given by
  \begin{equation}
    Z(\lambda) = n_0 c_G 
      \Res_\alpha \left( \frac{\varpi^2(u) e^{-i\lambda(u)}}{\prod_{j=1}^N \beta_j(u)} \right)
  \end{equation}
\end{corollary}
\begin{proof} By the formal properties of the residue operation, for any $f \in R^W$ we
have
\begin{equation}
D_f \Res_\alpha \left( \frac{\varpi^2(u) e^{-i\lambda(u)}}{\prod_{j=1}^N \beta_j(u)} \right)
 = \Res_\alpha \left( \frac{ \varpi^2(u) f(u) e^{-i\lambda(u)}}{\prod_{j=1}^N \beta_j(u)} \right).
\end{equation}
By Theorem \ref{thm-linear-residue-formula}, the evaluation at $\lambda=\alpha$ gives
the intersection pairing $\int_{\FX_\alpha} \kappa(f) e^\omega$, and this property determines
the generating function $Z(\lambda)$ uniquely.
\end{proof}

\begin{remark} Generating functions for intersection pairings were also studied in \cite{GS95}.
\end{remark}

\begin{corollary} \label{corollary-residue-quiver-cohomology}
Let $\CQ$ be any acyclic quiver with dimension vector $\bd$, and let $\alpha$ be
a generic moment map level. Then the Kirwan map $H^\ast(BG_\bd) \to H^\ast(\FX_\alpha(\CQ, \bd))$
is surjective, and the intersection pairings of Kirwan classes may be computed
algorithmically by the Theorem \ref{thm-linear-residue-formula}. In particular, the
ring structure on $H^\ast(\FX_\alpha(\CQ,\bd))$ may be computed algorithmically.
\end{corollary}

\section{Hyperk\"ahler Residues} \label{sec-hyperkahler-residues} 

\subsection{Hyperk\"ahler Residues}
Let $X = \BC^N$ and $M = T^\ast \BC^N$. As before, we assume that a compact Lie group
$G$ acts linearly on $X$, and as before we let $T \subset G$ be a maximal torus. 
Then $M = T^\ast X$ which is naturally a flat hyperk\"ahler manifold.
As before we let $R = H^\ast(BT)$ and $R^W \iso H^\ast(BG)$. We have the
\emph{hyperk\"ahler Kirwan map}
\begin{equation}
  \khk: R^W \to H^\ast(\FM_\alpha).
\end{equation}
We would like to understand the kernel of the hyperk\"ahler Kirwan map by computing
cogenerators. Since $\FM_\alpha$ is non-compact, there is no natural way to integrate over 
it.\footnote{It is possible to define an equivariant integration theory \cite{HauselProudfoot}, but we wish
to work in ordinary cohomology. An equivariant version of the residue formula was developed in \cite{Martens}
as well as \cite{Szilagyi13}.} 
However, we can integrate over cycles. We define the \emph{hyperk\"ahler residue operation} to be
\begin{equation}
  Z: H_\ast(\FM) \to (R^W)^\vee, \ \ 
    \gamma \mapsto Z_\gamma(\lambda) = \frac{1}{|W|}\int_\gamma 
    \kappa \left(\sum_{s \in W} e^{i (\lambda-\alpha)( s\cdot u)} \right) e^\omega.
\end{equation}
As before, the following is an immediate consequence of the definition.
\begin{lemma} For any $f \in R^W$, we have
\begin{equation}
  D_f Z_\gamma(\alpha) = \int_\gamma \kappa(f) e^\omega.
\end{equation}
\end{lemma}

\begin{theorem} \label{thm-residue-complete-cogenerators}
The residues $\{ Z_\gamma \ | \ \gamma \in H_\ast(\FM_\alpha) \}$ are a 
complete set of cogenerators for the image of the hyperk\"ahler Kirwan map.
\end{theorem}
\begin{proof} If $\kappa(f) \neq 0$, then there is some $\gamma \in H_\ast(\FM_\alpha)$ 
such that $\ev{\gamma, \kappa(f)} \neq 0$. Hence $D_f Z_\gamma \neq 0$, and we find that
the annihilator of the $(R^W)^\vee$-module generated by the residues is exactly equal
to the kernel of the Kirwan map.
\end{proof}

In the symplectic case, the residue formula gave an effective method to compute the
residue, and hence the cohomology ring. In the hyperk\"ahler case, we will develop an
algorithm to compute a complete set of cogenerators.

\begin{remark} It is possible to define formal $S^1$-equivariant integration theory
on hyper\"ahler quotients \cite{HauselProudfoot}. However, due to the complicated structure
of the fixed-point set (as evidenced by the calculations in Chapter \ref{ch-quiver}),
these equivariant integrals can be difficult to compute in general. Instead, we choose
to work with ordinary integration. The main feature of this approach is that much information
about the cohomology of hyperk\"ahler quotients can be deduced from the cohomology of
simpler, compact symplectic quotients. 
\end{remark}

\subsection{K\"ahler Quotients as Integration Cycles}

In this section we will show that a large set of hyperk\"ahler residues may be
calculated directly from the residue formula. Suppose that $V \subset M$ is a $G$-invariant linear subspace such that
$\muc|_V = 0$ and that $\mur|_V$ is proper. Assume furthermore that $\alpha$ is
a regular value of $\mur|_V$. Then define
\begin{equation}
  \FV_\alpha := V \reda{\alpha} G.
\end{equation}
This is a natural compact subvariety of $\FM_\alpha$, and hence defines a homology class
$[\FV_\alpha] \in H_\ast(\FM_\alpha)$. Now recall that $H^\ast(\FM_\alpha)$ is independent
of $\alpha$ as long as $\alpha$ is generic. 

\begin{definition} A \emph{VGIT residue} is a hyperk\"ahler residue $Z_\gamma$ associated
to a class $\gamma = [\FV_\alpha]$, for some choice of $V \subset M$ and $\alpha$ as
above. The \emph{VGIT ring} is the quotient $R^W / I$, where the ideal $I$ is the
annihilator of the $(R^W)^\vee)$-module generated by the VGIT residues.
\end{definition}

In the above definition, VGIT is short for variation of GIT. The reason for this terminology
is the following proposition, which follows immediately from the definition.
\begin{proposition}
The VGIT ideal is equal to
\begin{equation}
  \bigcap_{V \subset M, \alpha} \ker \left(\kappa: R^W \to H^\ast(\FV_\alpha) \right)
\end{equation}
where the intersection runs over all $V \subset M$ as above and $\alpha$ generic.
As a consequence, the ordinary Kirwan map $R^W \to H^\ast(\FV_\alpha)$ factors through the VGIT ring, i.e.
\begin{equation}
  R^W \onto R^W / I \onto H^\ast(\FV_\alpha),  
\end{equation}
where $I$ is the VGIT ideal.
\end{proposition}

The residues associated to the classes $[\FV_\alpha]$ may be computed by the linear 
residue formula, Theorem \ref{thm-linear-residue-formula}. Hence the VGIT ideal and
VGIT ring may be computed algorithmically. Moreover, there is a natural inclusion
$\ker \khk \subseteq I_{VGIT}$, and hence a natural surjection $\im(\khk) \onto R_{VGIT}$.

\begin{conjecture} \label{conjecture-vgit-kirwan}
Let $G$ act linearly on $M = T^\ast X$ as above, and assume that the restriction of the
real moment map to $X$ is proper. Then the image of the hyperk\"ahler Kirwan map is equal
to the VGIT ring.
\end{conjecture}

\begin{remark} The hypothesis that the restriction of the real moment map to $X$ be
proper is essential. One may check explicitly that the Nakajima varieties associated to the
ADHM quiver (Figure \ref{fig-adhm-quiver}) do not satisfy $\im(\khk) = R_{VGIT}$. However,
because this quiver has a vertex self-loop, it is impossible to find a Lagrangian subspace
$X \subset M$ such that the restriction of $\mur$ is proper.
\end{remark}

This conjecture is motivated by the following results.
\begin{theorem}[Proudfoot \cite{ProudfootGIT}] Conjecture \ref{conjecture-vgit-kirwan} is true for hypertoric varieties.
\end{theorem}
\begin{theorem}[Konno \cite{KonnoPolygon}] Conjecture \ref{conjecture-vgit-kirwan} is true for hyperpolygon spaces.
\end{theorem}
In Theorem \ref{thm-vgit-rank-3-star-quiver} we compute the VGIT rings of rank 3 star
quivers (for small values of $n$) and find that they satisfy this conjecture.

\subsection{Quiver Varieties}
We expect that Conjecture \ref{conjecture-vgit-kirwan} may be true for all acyclic quivers,
but unfortunately we do not have a general result.\footnote{But see \S\ref{sec-residue-examples} for some
special cases.} Instead, we will use abelianization to give an explicit algorithm
to compute a complete set of hyperk\"ahler residues.

Let $\FM_\alpha = M \rreda{\alpha} G$ be a Nakajima quiver variety, $T \subset G$ a
maximal torus, and suppose that we have chosen a circle action on $M$ so that $\FM_\alpha$ is 
circle compact (see \S\ref{sec-quiver-morse-circle}). Then we may consider
the $S^1$-equivariant cohomology $H_{S^1}^\ast(\FM_\alpha)$, as well as the $S^1$-equivariant
Kirwan map
\begin{equation}
  \kappa_{S^1}: H^\ast(BG) \otimes H^\ast(BS^1) \to H_{S^1}^\ast(\FM_\alpha).
\end{equation}
Now $H^\ast(BS^1) \iso \BQ[u]$, and by Theorem \ref{thm-formal-mu-perf} we have
$H_{S^1}^\ast(\FM_\alpha) \iso H^\ast(\FM_\alpha)[u]$ as a $\BQ[u]$-module.
Hence we have a natural ring map
\begin{equation}
  q: H_{S^1}^\ast(\FM_\alpha) \onto H^\ast(\FM_\alpha)
\end{equation}
which is given by quotienting by the ideal generated by $u \cdot 1 \in H_{S^1}^\ast(\FM_\alpha)$.
The main technical tool we need is the following.

\begin{theorem} \cite{HauselProudfoot} \label{thm-residue-abelianization}
The image of the $S^1$-equivariant Kirwan map
is isomorphic to
\begin{equation}
  H_{S^1}^\ast(M \rreda{\alpha} T)^W / \ann(e),
\end{equation}
where $e$ is the class given by
\begin{equation}
  e = \prod_{\alpha \in \Delta} \alpha(u-\alpha).
\end{equation}
\end{theorem}

The above theorem gives a complete solution to the problem of computing hyperk\"ahler
residues of Nakajima quiver varieties.
\begin{theorem} \label{thm-residue-hyperkahler-algorithm}
The following algorithm produces a complete list of hyperk\"ahler
residues for any Nakajima quiver variety.
\begin{enumerate}
  \item Compute $H_{S^1}^\ast(M \rred T)$ using the techniques of Chapter \ref{ch-morse}.
  \item Compute a complete set of Weyl-invariant cogenerators of the above ring in dimensions
    less than or equal to $\dim_\BR \FM_\alpha + \deg(e)$.
  \item Apply the differential operator $D_e$ to the above set of cogenerators.
  \item Compute the $u$-independent linear combinations of the above cogenerators.
\end{enumerate}
\end{theorem}
\begin{proof} Lemmas \ref{lemma-residue-quotient-u} and \ref{lemma-residue-varpi} ensure
that the above procedure produces a complete set of cogenerators for the ring
$H_{S^1}^\ast(M \rreda{\alpha} T)^W / \ev{\ann(e), u}$, and by Theorem \ref{thm-residue-abelianization} 
and formality this is equal to the image of the hyperk\"ahler Kirwan map.
\end{proof}


\section{Examples} \label{sec-residue-examples}

\subsection{Hypertoric Varieties}

\begin{example}
We consider the hyperplane arrangement pictured in figure \ref{fig-hyperplane}. The normals
are given by the column vectors of the matrix
\begin{equation} \left( \begin{array}{rrrr}
1 & 0 &  0 & -1 \\
0 & 1 & -1 & -1 
\end{array} \right) \end{equation}
The kernel is generated by the vectors $(1,0,-1,1)^T$ and $(0,1,1,0)^T$.
Hence this arrangement corresponds to the $S^1 \times S^1$ representation
\begin{equation} V = \BC(1,0) \oplus \BC(0,1) \oplus \BC(-1, 1) \oplus \BC(1,0). \end{equation}

\begin{figure}[!htbp]
\centering
\begin{tikzpicture}



  \draw[fill,color=gray] (0,0) -- (2, 0) -- (1, 1) -- (0, 1) -- cycle;

  \draw[thick] (-0.5,0) -- (3,0);
  \draw[thick] (0,-0.5) -- (0,3);
  \draw[thick] (-0.5, 1) -- (3, 1);
  \draw[thick] (-0.5, 2.5) -- ( 2.5, -0.5);

  \draw[decorate,decoration={brace, amplitude=6pt}] (-0.1, 0.1) -- (-0.1,0.9);
  \draw[decorate,decoration={brace, amplitude=6pt}] (-0.1, 1.1) -- (-0.1,1.9);

  \node (L1) at (-0.6, 0.5) {$\lambda_2$};
  \node (L2) at (-0.6, 1.5) {$\lambda_1$};

\end{tikzpicture}
\caption{The hyperplane arrangement corresponding to the blow-up of $\BP^2$ at a point.}
\label{fig-hyperplane}
\end{figure}
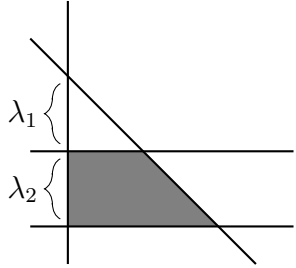

The moment map is given by
\begin{equation} \mu(x) = \frac{1}{2} \left( \begin{array}{c}
|x_1|^2 - |x_3|^2 + |x_4|^2 \\
|x_2|^2 + |x_4|^2
\end{array}\right) 
- \left( \begin{array}{r} \lambda_1 \\ \lambda_2 \end{array} \right) \end{equation}

To compute the symplectic volume, we have to calculate
\begin{eqnarray*}
\res{e^{+i\lambda_1 u_1 +i\lambda_2 u_2}}{u_1^2 u_2 (u_2 - u_1)}
  &=& -\theta(\lambda_2)\res{e^{i\lambda_1 u_1}}{u_1^3}
    +  \theta(\lambda_2)\res{e^{i(\lambda_1+\lambda_2)u_1}}{u_1^3} \\
  &=& \theta(\lambda_1|\lambda_2) \frac{\lambda_1^2}{2}
    - \theta(\lambda_1+\lambda_2|\lambda_2) \frac{(\lambda_1+\lambda_2)^2}{2}
\end{eqnarray*}
Thus we find that the symplectic volume is (proportional to)
\begin{equation} Z(\lambda) = \frac{1}{2}\left\{ \begin{array}{ll}
(\lambda_1+\lambda_2)^2 - \lambda_1^2, & \lambda_1 > 0, \lambda_2 > 0 \\
(\lambda_1+\lambda_2)^2, & \lambda_1 < 0, \lambda_1 + \lambda_2 > 0, \\
0, & \mathrm{otherwise}
\end{array} \right. \end{equation}
Hence there are two Jeffrey-Kirwan residues, $Z_1$ and $Z_2$. The wall-crossing
behavior of the residue formula is pictured in Figure \ref{fig-wall-crossing}.
Taking derivatives, we find
\begin{equation} \begin{array}{l|lllll}
 & \partial_1 & \partial_2 & \partial_1^2 & \partial_{12} & \partial_2^2 \\ \hline
Z_1 & \lambda_2 & \lambda_1+\lambda_2 & 0 & 1  & 1 \\
Z_2 & \lambda_1+\lambda_2 & \lambda_1+\lambda_2 & 1 & 1 & 1
\end{array} \end{equation}
Hence we find a single common relation in degree 4, $\partial_{12} = \partial_2^2$,
and of course in degrees $\geq 6$ every constant coefficient differential operator acts
as $0$. Hence the cohomology ring of the corresponding hypertoric variety is given by
\begin{equation} H^\ast(\FM) \iso \frac{\BC[u_1,u_2]}{\vev{u_1u_2-u_2^2, u_1^3, u_1^2 u_2}} \end{equation}



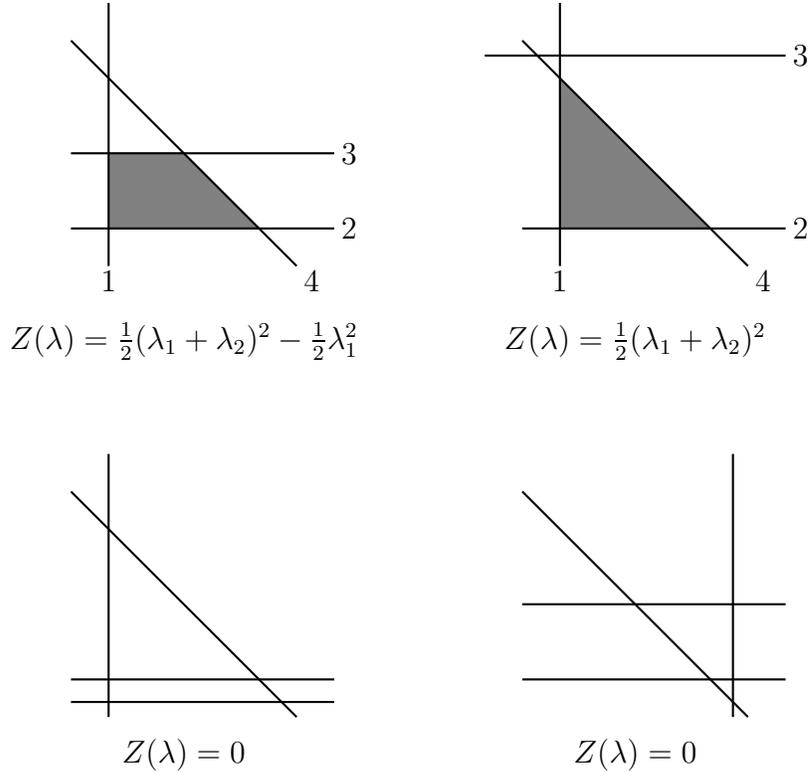
\begin{figure}[!htbp]
\centering
\begin{tikzpicture}


  \node (n1) at (1,-1.5) {$Z(\lambda) = \frac{1}{2} (\lambda_1+\lambda_2)^2 - \frac{1}{2} \lambda_1^2$};
  \node (n2) at (7,-1.5) {$Z(\lambda) = \frac{1}{2} (\lambda_1+\lambda_2)^2$};
  \node (n3) at (1,-7) {$Z(\lambda) = 0$};
  \node (n4) at (7,-7) {$Z(\lambda) = 0$};
  

  \draw[fill,color=gray] (0,0) -- (2, 0) -- (1, 1) -- (0, 1) -- cycle;

  \draw[thick] (-0.5,0) -- (3,0);
  \draw[thick] (0,-0.5) -- (0,3);
  \draw[thick] (-0.5, 1) -- (3, 1);
  \draw[thick] (-0.5, 2.5) -- ( 2.5, -0.5);

  \node (01) at ( 0.0,-0.7) {$1$}; 
  \node (02) at ( 3.2, 0.0) {$2$};
  \node (03) at ( 3.2, 1.0) {$3$};
  \node (04) at ( 2.7,-0.7) {$4$};

  \draw[fill,color=gray] (6,0) -- (8, 0) -- (6, 2) -- cycle;

  \draw[thick] ( 5.5,0) -- (9,0);
  \draw[thick] ( 6,-0.5) -- (6,3);
  \draw[thick] ( 5, 2.3) -- (9, 2.3);
  \draw[thick] ( 5.5, 2.5) -- ( 8.5, -0.5);

  \node (11) at ( 6.0,-0.7) {$1$}; 
  \node (12) at ( 9.2, 0.0) {$2$};
  \node (13) at ( 9.2, 2.3) {$3$};
  \node (14) at ( 8.7,-0.7) {$4$};

  \draw[thick] (-0.5,-6) -- (3,-6);
  \draw[thick] (0,-6.5) -- (0,-3);
  \draw[thick] (-0.5, -6.3) -- (3, -6.3);
  \draw[thick] (-0.5, -3.5) -- ( 2.5, -6.5);

  \draw[thick] ( 5.5,-6.0) -- (9,-6);
  \draw[thick] ( 8.3,-6.5) -- (8.3,-3);
  \draw[thick] ( 5.5,-5.0) -- (9, -5);
  \draw[thick] ( 5.5,-3.5) -- ( 8.5, -6.5);

\end{tikzpicture}
\caption{The three possible walls.}
\label{fig-wall-crossing}
\end{figure}
\end{example}

\subsection{Polygon Spaces}
We will show here that Theorem \ref{thm-linear-residue-formula} gives a very straightforward
calculation of the cohomology rings of polygon spaces. These correspond to a star quiver
with $n$ incoming arrows and dimension vector $\bd = (2, 1, \dots, 1)$. Let $x_1, \dots, x_n$
be coordinates on the vector space $X = \Rep(\CQ, \bd)$, where each $x_i$ is a $2 \times 1$
column vector. The action of $G_\bd$ on $X$ is given by
\begin{equation}
  x_i \mapsto g x_i h_i^{-1}
\end{equation}
where $g \in SU(2)$ and $h_i \in S^1$. (We restrict to $g \in SU(2)$ because the overall diagonal
in $G_\bd$ acts trivially.) Let $u_0, \cdots, u_n$ be coordinates on the maximal torus
of $SU(2) \times (S^1)^n$. Then the weights of the above action are $(u_i \pm u_0)$ for
$i = 1, \dots, n$. Hence the product of all weights is
\begin{equation}
  \prod_{\beta \in \CB} \beta = \prod_{i=1}^n (u_i^2-u_0^2).
\end{equation}
Also, we have $\varpi^2 = -u_0^2$. Hence we have to compute
\begin{equation}
  \Res \left( \frac{u_0^2 e^{-i\lambda_0 u_0 - i \sum_j \lambda_j u_j}}{\prod_{j=1}^n (u_j^2-u_0^2)} \right).
\end{equation}
For each $j$, we have a pole at $u_j = \pm u_0$. Taking the residues for $j=n$, we have
\begin{align}
  \Res \left( \frac{u_0^2 e^{-i\lambda_0 u_0 - i \sum_j \lambda_j u_j}}{\prod_{j=1}^n (u_j^2-u_0^2)} \right)
   &= \theta(\lambda_n) \Res \left( \frac{u_0^2 e^{-i(\lambda_0+\lambda_n) u_0 - i \sum_{j\neq n} \lambda_j u_j}}{2u_0\prod_{j\neq n} (u_j^2-u_0^2)} \right) \\
    & \ \ - \theta(\lambda_n)\Res \left( \frac{u_0^2 e^{-i(\lambda_0-\lambda_n) u_0 - i \sum_{j\neq n} \lambda_j u_j}}{2u_0\prod_{j\neq n} (u_j^2-u_0^2)} \right)
\end{align}
Next we take the residues at $u_{n-1} = \pm u_0$, and so on. Continuing inductively, we obtain
\begin{equation}
   \sum_{\sigma \in \{+,-\}^n} (-1)^\sigma
     \Res \left( \frac{u_0^2 e^{-i (\lambda_0 + \sigma \cdot \lambda) u_0 }}{2^n u_0^n} \right),
\end{equation}
where $\sigma \cdot \lambda = \sum_i \sigma_i \lambda_i$.
\begin{proposition} Let $\FX_\alpha$ be a polygon space. Then a cogenerator for its
cohomology ring is
  $D_{\lambda_0}^2 \sum_{\sigma \in \CS} (\lambda_0 + \sigma \cdot \lambda)^{n-1}$.
\end{proposition}

\begin{example} Take $\alpha = (1,1,1,2,2)$. Then according to the above proposition,
the cogenerator for the cohomology ring of $\FX_\alpha$ is (up to an irrelevant overall
constant)
\begin{equation}
\lambda_0^2 + \lambda_1^2 - 2\lambda_1\lambda_2 + \lambda_2^2 - 2\lambda_1\lambda_3 
 - 2\lambda_2\lambda_3 + \lambda_3^2 + \lambda_4^2 - 2\lambda_4\lambda_5 + \lambda_5^2
\end{equation}
Taking partial derivatives, we find the relations $\partial_4+\partial_5=0$ and
$\partial_0^2 - \partial_i^2 = 0$ for $i = 1, \dots, 5$. Hence the Poincar\'e polynomial
is $1 + 4t^2 + t^4$ and the cohomology ring is
\begin{equation}
  H^\ast(\FX_\alpha) = \frac{\BQ[u_0^2, u_1, \dots, u_5]}{\ev{u_0^2-u_i^2, u_4+u_5, \deg 6}}.
\end{equation}
\end{example}

\subsection{Rank 3 Star Quivers}

In this case, let $u_1, \dots, u_n$ be generators of $(S^1)^n$ and let $v_1, v_2$ be
generators of the maximal torus in $SU(3)$. Then $\varpi^2 = (v_1-v_2)(2v_1+v_2)(v_1+2v_2)$.
The weights are given by $u_i-v_1, u_i-v_2, u_i+v_1+v_2$ for each $i$.
For each $i$, there are three possible residues with respect to $u_i$: $u_i \to v_1, u_i \to v_2$, and
$u_i \to -v_1-v_2$. 
The residue as $u_n \to v_1$ produces the term
\begin{equation}
  \Res\left(\frac{\varpi^2 e^{-i\mu \cdot v -i\lambda \cdot u -i\lambda_n v_1}}{(v_1-v_2)(2v_1+v_2)} \right).
\end{equation}
The residue as $u_n \to v_2$ produces the term
\begin{equation}
  \Res\left(\frac{\varpi^2 e^{-i\mu \cdot v -i\lambda \cdot u -i\lambda_n v_2}}{(v_2-v_1)(v_1+2v_2)} \right).
\end{equation}
The residue as $u_n \to -v_1-v_2$ produces the term
\begin{equation}
  \Res\left(\frac{\varpi^2 e^{-i\mu \cdot v -i\lambda \cdot u +i\lambda_n (v_1+v_2)}}{(2v_1+v_2)(v_1+2v_2)} \right).
\end{equation}
Iterating as in the previous case, we see that the
residue will be a sum over disjoint subsets $S_1, S_2, S_3 \subset [n]$, where
$i \in S_1$ indicates that we take the residue $u_i \to v_1$, $i \in S_2$ denotes
the residue $u_i \to v_2$, and $i \in S_3$ denotes the residue $u_i \to -v_1-v_2$.
The denominator of such a term is
\begin{equation}
  (-1)^{|S_2|} (v_1-v_2)^{|S_1|+|S_2|} (2v_1+v_2)^{|S_1|+|S_3|} (v_1+2v_2)^{|S_2|+|S_3|}.
\end{equation}
The argument of the exponential of such a term is
\begin{equation}
  (S_1 \cdot \lambda-S_3 \cdot \lambda) v_1 + (S_2 \cdot \lambda - S_3 \cdot \lambda) v_2.
\end{equation}

It remains to take the residues at $v_2 \to v_1$, $v_2 \to -2v_1$, and $v_2 \to -\frac{1}{2} v_1$,
followed by the remaining residue at $v_1 \to 0$.

\begin{proposition} A cogenerator for the rank 3 star quiver is given by
\begin{equation}
  \sum_{S_1 \sqcup S_2 \sqcup S_3 = [n]} \left( R_1(S_1, S_2, S_3) + R_2(S_1, S_2, S_3) + R_3(S_1, S_2, S_3) \right)
\end{equation}
where 
\begin{align}
  \begin{split}
    R_1(S_1, S_2, S_2) &= \theta(S_1 \cdot \alpha + S_2 \cdot \alpha - 2 S_3 \cdot \alpha) 
    \theta(S_2 \cdot \alpha - S_3 \cdot \alpha) \\
    & \times \Res_{v_1 \to 0} \Res_{v_2 \to v_1}\left(f(S_1, S_2, S_2)\right) 
  \end{split} \\
  \begin{split}
    R_2(S_1, S_2, S_2) &= \theta(S_1 \cdot \alpha + S_3 \cdot \alpha - 2 S_2 \cdot \alpha) 
    \theta(S_2 \cdot \lambda - S_3 \cdot \lambda) \\
    & \times \Res_{v_1 \to 0} \Res_{v_2 \to -2v_1} \left(f(S_1, S_2, S_2)\right)
   \end{split} \\
   \begin{split}
     R_3(S_1, S_2, S_2) &= \theta(2S_1 \cdot \alpha - S_2 \cdot \alpha - 2 S_3 \cdot \alpha) 
     \theta(S_2 \cdot \alpha - S_3 \cdot \alpha) \\
     & \times \Res_{v_1 \to 0} \Res_{v_2 \to -v_1/2} \left(f(S_1, S_2, S_2)\right) 
   \end{split}
\end{align}
and $f(S_1, S_2, S_3)$ is the meromorphic function
\begin{equation}
  f(S_1, S_2, S_3) = \frac{\exp\left(-i(\mu_1 + S_1 \cdot \lambda - S_3 \cdot \lambda)v_1 -i (\mu_2 + S_2 \cdot \lambda - S_3 \cdot \lambda)v_2\right)}
   {(-1)^{|S_2|} (v_1-v_2)^{|S_1|+|S_2|-2} (2v_1+v_2)^{|S_1|+|S_3|-2} (v_1+2v_2)^{|S_2|+|S_3|-2}}
\end{equation}
\end{proposition}

The above expression, though explicit, is rather difficult to use directly by hand,
but it is very easy to implement in computer algebra software.
By varying the moment map level $\alpha$ across all the chambers, we produce cogenerators
for the VGIT ring. By comparison with the Betti numbers from Theorem \ref{thm-rank-3-betti-numbers},
we can then deduce Kirwan surjectivity. Using our computer implementation, we find the following.
\begin{theorem} \label{thm-vgit-rank-3-star-quiver}
For $n=4,5,6$, the hyperk\"ahler Kirwan map is surjective and its kernel is equal to the VGIT ideal.
Moreover, the cohomology ring is level. For $n=5$ the cohomology ring is
\begin{equation}
  H^\ast(\FM_\alpha) \iso \BQ[v_1^2+v_1^2+(v_1+v_2)^2, v_1^3+v_2^3-(v_1+v_2)^3, u_1, \cdots, u_5] / I
\end{equation}
where $I$ is the ideal generated by all degree $6$ polynomials in the generators
as well as the following degree 4 polynomials
\begin{align}
  & 4u_2u_3 + 4u_2u_4 + 4u_3u_4 + 4u_2u_5 + 4u_3u_5 + 4u_4u_5 + 3v_1^2 + 3v_2^2, \\
  & 3u_1u_2 + u_2u_3 + u_2u_4 - 2u_3u_4 + u_2u_5 - 2u_3u_5 - 2u_4u_5, \\
  & 3u_1u_3 + u_2u_3 - 2u_2u_4 + u_3u_4 - 2u_2u_5 + u_3u_5 - 2u_4u_5, \\
  & -2u_2u_3 + 3u_1u_4 + u_2u_4 + u_3u_4 - 2u_2u_5 - 2u_3u_5 + u_4u_5, \\
  & -2u_2u_3 - 2u_2u_4 - 2u_3u_4 + 3u_1u_5 + u_2u_5 + u_3u_5 + u_4u_5.
\end{align}
\end{theorem}

\begin{remark} The computer calculation can easily be extended to higher values of $n$
as well as higher ranks, but as the list of generators of the kernel of the Kirwan map
grows very quickly we have decided only to state the results for $r=3$ and $n=5$.
\end{remark}



\chapter{Star Quivers and Integrable Systems} \label{ch-integrable}

\section{Flags and Stars}

\subsection{The Springer Resolution}
We begin by recalling the Springer resolution, which is of fundamental importance in
geometric representation theory \cite{ChrissGinzburg}. Let $G$ be a simply connected complex semisimple
group with Lie algebra $\fg$.
The \emph{nilpotent cone} $\CN \subset \fg$ is a singular affine variety which carries
a natural Poisson structure induced by the Lie-Poisson structure on $\fg$. Let $\CB$ be the set
of all Borel sub-algebras of $\fg$. Then we define a variety $\tilde{\CN}$ by
\begin{equation}
  \tilde{\CN} = \{ (x, \fb) \in \CN \times \CB \suchthat x \in \fb \}.
\end{equation}
Note that there is a natural projection map $\mu: \tilde{\CN} \to \CN$.
\begin{theorem} The variety $\tilde{\CN}$ is isomorphic to $T^\ast(G/B)$,
where $G/B$ is the flag variety of $G$ (for $B$ a fixed Borel subgroup).
The map $\mu: T^\ast(G/B) \to \CN$ is a resolution of singularities.
Moreover, the map $\mu$ is the moment map for the canonical $G$-action on $T^\ast(G/B)$, and
in particular, it is a Poisson map.
\end{theorem}
\begin{definition} The map $\mu: T^\ast(G/B) \to \CN$ is called the \emph{Springer resolution}.
\end{definition}
In the case of type $A$ (i.e., $\mathfrak{sl}_n$), the Springer resolution may be constructed
using Nakajima quiver varieties.\footnote{In fact, Nakajima varieties provide a large class
of \emph{conical symplectic resolutions}, of which the Springer resolution is a special case.}
Consider the type $A$ quiver pictured in Figure \ref{fig-integrable-type-a},
and let $\FX_\alpha$ be the K\"ahler quiver variety associated this quiver.
There is a natural residual action of $U(r)$ given by the action on the left-most vertex.
This action is Hamiltonian, with moment map $\mu: \FX_\alpha \to \fu_r$.
Similarly, the hyperk\"ahler quotient $\FM_\alpha$ has a residual hyperhamiltonian action
of $U(r)$, with real and complex moment maps given by
\begin{align}
  \mur([x, y, x_1, y_1, \cdots, x_l, y_l]) &= \frac{1}{2} xx^\ast - \frac{1}{2} y^\ast y \\
  \muc([x, y, x_1, y_1, \cdots, x_l, y_l]) &= xy
\end{align}

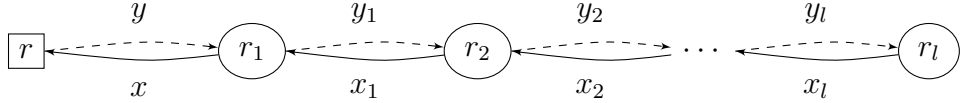
\begin{figure}[h]
\centering
\begin{tikzpicture}
  \coordinate (dx) at (3.0,0);
  \coordinate (dy) at (0, 0.5);
  \node[rectangle,draw] (n0) at (0,0) {$r$};
  \node[ellipse,draw] (n1) at ($(dx)$) {$r_1$};
  \node[ellipse,draw] (n2) at ($2*(dx)$) {$r_2$};
  \node               (n3) at ($3*(dx)$) {$\cdots$};
  \node[ellipse,draw] (n4) at ($4*(dx)$) {$r_l$};

  \node (edgex) at ($0.5*(dx) + (dy)$) {$y$};
  \node (edgey) at ($0.5*(dx) - (dy)$) {$x$};

  \foreach \s in {1,2,3,4} {
    \pgfmathsetmacro{\t}{\s-1}
    \doublearrow{n\s}{n\t};
  }

  \foreach \s in {1,2,3} {
    \ifnum \s=3
      \node (edgex) at ($0.5*(dx) + \s*(dx) + (dy)$) {$y_l$};
      \node (edgey) at ($0.5*(dx) + \s*(dx) - (dy)$) {$x_l$};
    \else
      \node (edgex) at ($0.5*(dx) + \s*(dx) + (dy)$) {$y_\s$};
      \node (edgey) at ($0.5*(dx) + \s*(dx) - (dy)$) {$x_\s$};
    \fi
  }

\end{tikzpicture}
\caption{Type $A$ quiver for the Springer resolution.}
  \label{fig-integrable-type-a}
\end{figure}

\begin{proposition} \label{prop-int-flag}
The moment map $\mu: \FX_\alpha \to \mathfrak{u}_r$ is an embedding
and identifies $\FX_\alpha$ with a $U(r)$-coadjoint orbit of flag type
$(r_l, r_{l-1}, \cdots, r_1, r)$. The hyperk\"ahler quotient
$\FM_\alpha$ is naturally identified with $T^\ast \FX_\alpha$. 
The complex moment map $\muc: \FM_\alpha \to \gl_r$ takes values in
the nilpotent cone, and in the case of a complete flag (i.e. $(1,2,\dots,r)$),
it is the Springer resolution for $\mathfrak{sl}_r$.
\end{proposition}
\begin{proof} This is a standard exercise. 
\end{proof}

\begin{remark} \label{rmk-lie-poisson} Note that since $\muc: \FM_\alpha \to \gl_r$
is a moment map, it is in particular Poisson with respect to the Lie-Poisson structure
on $\gl_r$. That is,
\begin{equation}
  \{ \mu_{ij}, \mu_{kl} \} = \delta_{jk} \mu_{il} - \delta_{il} \mu_{kj},
\end{equation}
where $\mu_{ij}$ denotes the $(i,j)$th entry of the moment map $\mu$, and $\{\cdot,\cdot\}$
is the Poisson bracket. This observation will be crucial when we try to understand the Poisson
geometry of star quiver varieties in the sections that follow.
\end{remark}

\subsection{Star Quivers}
Next we consider a more general class of star quivers than those studied in Chapter \ref{ch-quiver}.

\begin{definition} A \emph{generalized star quiver} is a quiver $\CQ$ consists of
a single sink with $n$ incoming paths, such that the dimensions of the vertices are
strictly increasing as one moves toward the sink. We call the dimension of the sink
the \emph{rank}, and denote it by $r$. If the dimension vectors of all of the incoming
paths are complete, i.e. $(r, r-1, r-2, \dots, 1)$ we call the star quiver \emph{complete},
otherwise we call the star quiver \emph{incomplete}.
\end{definition}

\begin{figure}[h]
\centering
\begin{tikzpicture}
  \node[ellipse,draw] (sink) at (0,3.25) {$r$};
  \coordinate (dx) at (2,0);
  \coordinate (dy) at (0,1.5);

  \foreach \s in {1,2,3,4} {
    \ifnum \s=2
     \node (n) at ($\s*(dy) + (dx)$) {$\cdots$};
    \else
     \node[ellipse,draw] (n1) at ($(0,0)+\s*(dy) + (dx)$) {$r-1$};
     \node[ellipse,draw] (n2) at ($(0,0)+\s*(dy) + 2*(dx)$) {$r-2$};
     \node (n3) at ($(0,0)+\s*(dy) + 3*(dx)$) {$\cdots$};
     \node[ellipse,draw] (n4) at ($(0,0)+\s*(dy) + 4*(dx)$) {$1$};
     \path[draw,->,>=latex'] (n4) -- (n3) -- (n2) -- (n1) -- (sink);
    \fi
  }

\end{tikzpicture}
\caption{A complete star quiver.}
  \label{fig-complete-star-quiver}
\end{figure}
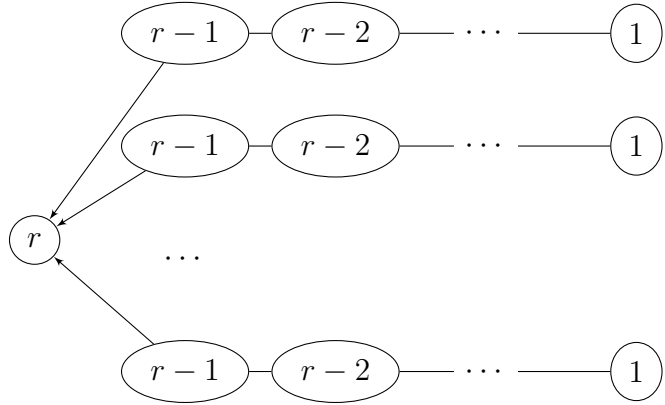

It follows immediately from Proposition \ref{prop-int-flag} that we have the 
following alternative description of $\FX_\alpha$ and $\FM_\alpha$.
\begin{proposition} \label{prop-int-flag-quotient}
The K\"ahler quiver variety associated to a generalized star quiver may
be identified, via reduction in stages, with the K\"ahler quotient of a product
of $U(r)$ coadjoint orbits by the diagonal $SU(r)$-action. Similarly, the hyperk\"ahler
quiver variety associated to a star quiver may be identified with the hyperk\"ahler
quotient of a product of cotangent bundles of flag varieties. In the case of a complete
star quiver, $\FM_\alpha$ has complex dimension $(n-2)r^2 -nr + 2$.
\end{proposition}

\begin{remark} This identification allows for an easy proof of hyperk\"ahler Kirwan surjectivity
for these spaces, see \cite[Remark 4.3]{HauselProudfoot}.
\end{remark}

\subsection{Parabolic Bundles}
Let $X$ be a Riemann surface and $E \to X$ a rank $r$ holomorphic vector bundle. Let $D = \sum_{i=1}^n a_i$
be a fixed divisor, and assume that the $a_i$ are pairwise distinct. For each $i$, fix a 
flag type $n > n_1 > \cdots > n_{l_i} > 0$. 

\begin{definition}
A \emph{parabolic structure} on $E$ is a 
choice of flag 
\begin{equation}
  E_{a_i} = E_{a_i}^0 \supset E_{a_i}^1 \supset \cdots \supset E_{a_i}^{l_i} \supset 0,
\end{equation}
at each of the points $a_i$ of the divisor. 
A \emph{parabolic Higgs field} is a section
\begin{equation} \phi \in H^0(X, \End E \otimes \CO_X(D) \otimes K_X) \end{equation}
such that at each point $a_i$ of the divisor, we have
\begin{equation}
  \phi_i(E_{a_i}^j) \subseteq E_{a_i}^j,
\end{equation}
where $\phi_i$ is the residue of $\phi$ at the point $z=a_i$.
We say that $\phi$ is \emph{strictly parabolic} if $\phi_i(E_{a_i}^j) \subseteq E_{a_i}^{j+1}$
for all $i = 1, \dots n, j=0, \dots, l_i$.
\end{definition}

\begin{theorem} \label{thm-integrable-parabolic}
Let $\CQ$ be a generalized star quiver of type $(r,n)$ and let
$D = \sum_{i=1}^n a_i$ be a divisor on $\BP^1$. Then the K\"ahler quiver variety $\FX_\alpha$
is naturally identified with a moduli space of parabolic structures on the trivial rank
$r$ vector bundle $E \to \BP^1$, and the hyperk\"ahler quiver variety $\FM_\alpha$ is 
naturally identified with a moduli space of strictly parabolic Higgs bundles on $\BP^1$,
whose underlying bundle type is trivial.
\end{theorem}
\begin{proof} Let $E \to X$ be a trivial bundle. Then we may choose some trivialization
such that $E \iso \BC^r \times \BP^1$. Hence under this trivialization, the set of 
possible parabolic structures is a product of flag varieties. The diagonal action of
$GL(r)$ corresponds to a change of trivialization, and hence the space of
parabolic structures is naturally identified with the GIT quotient of a product of flag
varieties by the diagonal $GL(r)$-action. By Proposition \ref{prop-int-flag-quotient}, this may be
identified with a star quiver variety $\FX_\alpha$ (see Remark \ref{rmk-symplectic-git}).

Now consider a product of cotangent bundles of flag varieties, $M = \prod_{i=1}^n T^\ast \FF_i$.
Let $\mu_i$ denote the complex moment map on $T^\ast \FF_i$. Let $z$ be an affine coordinate on $\BP^1$
(assume that $z=0$ is not a point in the divisor $D$). Then consider the $\gl_r$-valued meromorphic form
\begin{equation} \label{eqn-integrable-higgs-defn}
  \phi(z) dz = \sum_{i=1}^n \frac{\mu_i dz}{z-a_i}.
\end{equation}
Then $\phi(z) dz$ is well-defined as a
section of $\End E \otimes K(D)$ on the vanishing set $V( \sum_i \mu_i ) \subset M$,
and hence defines a Higgs field.\footnote{The condition $\sum_i \mu_i =0$ is easily seen to be
equivalent to the condition that $\phi(z) dz$ is regular at $z=\infty$.}
Moreover, by Proposition \ref{prop-int-flag},
$\mu_i(x_i,y_i)$ is a nilpotent matrix preserving the flag defined by $x_i \in \FF_i$,
and hence $\phi(z)$ is a strictly parabolic Higgs field. 
The residual action of $GL(r)$ corresponds to change of trivialization, so the GIT quotient
$V( \sum_i \mu_i ) \red GL(r)$ is a moduli space of strictly parabolic Higgs bundles.
But $\sum_i \mu_i$ is exactly the complex moment map for the diagonal $GL(r)$-action,
and hence this quotient may be identified with $M \rred U(r) \iso \FM_\alpha$.
\end{proof}

\begin{remark} This construction is a natural generalization of the rank 2 construction
in \cite{GodinhoMandini}, wherein it was also shown that
the stability condition on quiver representations is equivalent to parabolic slope stability,
for an appropriate choice of parabolic weights. We suspect that this remains true in the
higher rank case, but as we shall not need this result, we make no attempt to prove it.
\end{remark}

\begin{remark} The moduli space of strictly parabolic Higgs bundles is a symplectic leaf
in the moduli space of parabolic Higgs bundles \cite{LogaresMartens}, and it is known to
be completely integrable. However, we take the viewpoint that star quivers should
be regarded as ``toy models'' of Hitchin systems, so we will instead prove integrability
by direct calculation.
\end{remark}

\begin{remark} If we fix a line bundle $L$, then a \emph{Hitchin pair} is a pair
$(E, \phi)$ where $E$ is a vector bundle over a curve and $\phi \in H^0(\End E \otimes L)$.\footnote{Of
particular interest is the case $L = -K_X$, which corresponds to \emph{co-Higgs fields} \cite{RayanThesis}.
These may be identified with holomorphic bundles in the sense of generalized complex geometry \cite{GualtieriGCX}.}
The moduli space of (stable) Hitchin pairs carries a natural family of Poisson structures which are all completely 
integrable with respect to the Hitchin map \cite{Bottacin,Markman}.
The construction of this section can be generalized to give a construction
of Hitchin pairs (over $\BP^1$) associated to the quiver pictured in Figure \ref{fig-twisted-quiver}.
However, it is not clear whether this construction preserves Poisson structures.
\begin{figure}[h]
\centering
\begin{tikzpicture}
  \coordinate (c) at (0,0);
  \dotnodebelow{sink}{c}{$r$};

  \coordinate (dx) at (1.5,0);
  \coordinate (dy) at (0,0.75);

  \foreach \s in {-2,-1,0,1,2} {
    \ifnum \s=0
     \node (n) at ($(dx)+\s*(dy)$) {$\cdots$};
    \else
     \coordinate (c) at ($(dx) + \s*(dy)$);
     \dotnoderight{n}{c}{$r-1$};
     \soliddoublearrow{n}{sink};
    \fi
  }

  \coordinate (c1) at ($-0.25*(dx)+(dy)$);
  \coordinate (c2) at ($-0.5*(dx)$);
  \coordinate (c3) at ($-0.25*(dx)-(dy)$);

  \path[->,>=latex',draw] (sink) edge[loop left] (sink);

\end{tikzpicture}
\caption{A quiver representing a space of Hitchin pairs.}
  \label{fig-twisted-quiver}
\end{figure}
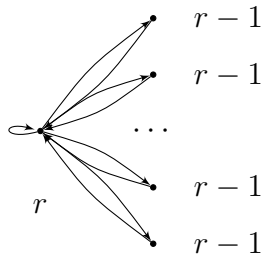
\end{remark}

\section{Complete Integrability}

\subsection{The Hitchin Fibration}
In this section, we describe the Hitchin map on $\FM_\alpha$ and show by direct
calculation that it is a coisotropic fibration.
\begin{definition} The \emph{Hitchin base} is
\begin{equation}
  B = \bigoplus_{i=1}^r H^0(\BP^1, K(D)^i),
\end{equation}
where $K$ is the canonical sheaf of $\BP^1$.
The \emph{Hitchin map} $h: H^0(\BP^1, \End E \otimes K(D)) \to B$
is the map
\begin{equation}
  \phi \mapsto ( \Tr(\phi), \Tr(\phi^2), \dots, \Tr(\phi^r) ).
\end{equation}
\end{definition}
Note that this map is invariant under $\phi \mapsto g^{-1} \phi g$, and hence descends
to a well defined map on $\FM_\alpha$, which we also denote by $h$. 

\begin{lemma} \label{lemma-hitchin-base}
On $\FM_\alpha$, the Hitchin map takes values in $B_0 \subset B$, where
\begin{equation}
  B_0 = \bigoplus_{i=1}^r H^0(\BP^1, K^i \otimes \CO(D)^{i-1}).
\end{equation}
\end{lemma}
\begin{proof} Since the residue of $\phi$ is strictly parabolic at each of the marked points, 
it is in particular nilpotent, and hence the poles of $\Tr(\phi^i)$ have order at most $i-1$.
Hence $\Tr(\phi^i)$ may be naturally identified with a section of 
$K(D)^i \otimes \CO(-D) \iso K^i \CO(D)^{i-1}$.
\end{proof}

\begin{proposition} We have $\dim B_0 = \frac{1}{2}\left( (n-2) r^2 -nr + 2 \right)$.
In the case of a complete flag quiver variety $\FM_\alpha$, we have
$\dim \FM_\alpha = 2 \dim B_0$.
\end{proposition}
\begin{proof} The line bundle $K^i \otimes \CO(D)^{i-1}$ has degree $-2i + n(i-1)$,
and hence 
\begin{equation}
  h^0(K^i \otimes \CO(D)^{i-1}) = (n-2)i -n+ 1.
\end{equation}
Summing these from $i=2$ to $i=r$, we obtain the result.
\end{proof}

Next we will show that the Hitchin map defines a coisotropic fibration. For convenience,
in this section we will work on $M = \prod_{i=1}^n T^\ast \FF_i$. Define the components $I_m$
of the Hitchin map via
\begin{equation}
  I_m (dz)^m = \Tr((\phi(z)dz)^m) \in H^0(\BP^1, K(D)^m \otimes \CO(D)^{m-1}).
\end{equation}
By construction, the $I_m$ are rational functions in $z$, and hence we may expand them
as formal power series $I_m(z) = \sum_n I_m^n z^n$. (By a change of coordinate if necessary,
we can assume that none of the $a_i$ is zero, so that $\phi(z)$ is regular at $0$.)
It will be convenient to treat $z$ as a formal parameter, so that $I_m(z)$ can be thought
of as a generating function for the coefficients $I_m^n$. \footnote{We think of the coefficients $I_m^i$
as giving coordinates on the Hitchin base $B$. Since they are the Taylor coefficients of 
a rational function, only finitely many of them are algebraically independent.}
Trivially, we have the following.
\begin{lemma} We have $\{I_m^i, I_n^j\}=0$ for all $i,j$ if and only if 
$\{I_m(z), I_n(w)\} = 0$ as an element of $\BC[[z,w]]$, where the formal parameters
$z,w$ are Casimirs of the Poisson bracket.
\end{lemma}
Next, we define a matrix-valued formal power series as follows:
\begin{equation}
  \Delta(z,w) = \frac{\phi(z)}{w-z} + \frac{\phi(w)}{z-w}.
\end{equation}
This matrix encodes the Poisson structure, in the following sense.
\begin{lemma} We have
  $\{ \phi_{ij}(z), \phi_{kl}(w) \} = \delta_{jk} \Delta_{il}(z,w) - \delta_{il} \Delta_{kj}(z,w)$.
\end{lemma}
\begin{proof} We simply compute
\begin{align}
  \{ \phi_{ij}(z), \phi_{kl}(w) \} &= \sum_{m,n} \frac{\{(\mu_m)_{ij}, (\mu_n)_{kl} \}}{(z-a_m)(w-a_n)} \\
   &= \sum_{m,n} \frac{\delta_{mn} \delta_{jk} (\mu_m)_{il}-\delta_{mn} \delta_{il} (\mu_m)_{kj}}{(z-a_m)(w-a_n)} \\
   &= \sum_m \frac{\delta_{jk} (\mu_m)_{il}-\delta_{il} (\mu_m)_{kj}}{(z-a_m)(w-a_m)} \\
   &= \delta_{jk} \Delta_{il}(z,w) - \delta_{il} \Delta_{kj}(z,w),
\end{align}
as claimed.
\end{proof}

\begin{proposition} \label{prop-integrable-coisotropic}
The components of the Hitchin map Poisson commute,
hence the Hitchin map $h: \FM_\alpha \to B_0$ is a coisotropic fibration.
\end{proposition}
\begin{proof} First, note that 
\begin{equation}
  \Tr(\phi(z)^m) = \sum_{i_1, \dots, i_m} \phi_{i_1 i_2}(z) \phi_{i_2 i_3}(z) \cdots \phi_{i_m i_1}(z).
\end{equation}
Hence, taking the Poisson bracket $\{I_m(z), I_n(w)\}$, expanding, and applying the Leibniz rule,
we have
\begin{align}
 \begin{split}
  \{I_m(z), I_n(w)\} &= \sum_{\stackrel{a,i_1,\dots, i_m}{b,j_1, \dots, j_n}} 
     \phi_{i_1 i_2}(z) \cdots \widehat{\phi_{i_a i_{a+1}} (z)} \cdots \phi_{i_m i_1}(z) \\
     &\ \ \ \  \times \phi_{j_1 j_2}(w) \cdots \widehat{\phi_{j_b j_{b+1}} (w)} \cdots \phi_{j_m j_1}(w) \\
      &\ \ \ \  \times \left( \delta_{i_{a+1} j_b} \Delta_{i_a j_{b+1}}(z,w) - \delta_{i_a j_{b+1}} \Delta_{j_b i_{a+1}}(z,w) \right)
    \end{split} \\
   \begin{split}
      &= \sum_{a,b} \Tr\left(\phi^{m-1}(z) \Delta(z,w) \phi^{n-1}(w) \right) \\
       &\ \ - \sum_{a,b} \Tr\left( \phi^{n-1}(w) \Delta(z,w) \phi^{m-1}(z)\right) \end{split} \\
      &= mn \Tr\left(\Delta(z,w) [\phi(w)^{n-1}, \phi(z)^{m-1}] \right).
\end{align}
Now recall that $\Delta(z,w) = \phi(z)/(w-z) + \phi(w)/(z-w)$.
This is a sum of two terms, one commuting with $\phi(z)$ and the other commuting with $\phi(w)$, 
hence the above trace is identically zero.
\end{proof}

\subsection{Spectral Curves}
It remains to show that the generic fibers of $h: \FM_\alpha \to B_0$ are Lagrangian subvarieties.
To do this, we will use the standard technique of spectral curves \cite{HitchinStableBundles, BNB89}.
Let $L \to K(D)$ be the pullback of $K(D)$ to the total space of $K(D)$.
Then $L$ has a tautological section, which we denote by $\lambda$.
\begin{definition} Let $b \in B$ and denote the components of $b$ by $b_i$ for $i=1, \dots, n$. 
Then we define the \emph{spectral curve} $X_b \subset K(D)$ to be the subvariety defined by
\begin{equation}
  X_b = \{ (\lambda, z) \in K(D) \suchthat \lambda^r + b_1(z) \lambda^{r-1} + \dots + b_r(z) = 0 \}.
\end{equation}
\end{definition}

\begin{proposition} Suppose that $n \geq 3$, and let $b \in B_0$ be generic. 
Then the spectral curve $X_b$ is smooth and of genus $g(X_b) = \frac{1}{2}( (n-2)r^2 -nr + 2)$.
\end{proposition}
\begin{proof} We will adapt Hitchin's argument \cite{HitchinStableBundles} to our setting.
By the general Bertini theorem \cite[Chapter III]{Hartshorne}, a generic spectral curve will be smooth
except possibly at the base points of the linear system defined by $B_0$. By construction,
$B_0$ is the subspace of $\oplus_{i=2}^r H^0(\BP^1, K(D)^i)$ consisting of sections that
vanish on $D$, so the base locus of $B_0$ is exactly the divisor $D$. 
Let $f(\lambda, z)$ be the equation defining $X_b \subset S$. At any point
$a_i \in D$, all of the coefficients of the characteristic polynomial vanish, so that $\lambda=0$.
At such a point we have $df = d b_r$, and hence $df \neq 0$ as long as $b_r$ has only simple zeroes on $D$.
Since this is the generic behavior, we find that a generic spectral curve is smooth.

Let $S$ be the complex surface $S = \BP(\CO\oplus K(D))$, which is a compactification of the
total space of $K(D)$. Then $X_b \subset S$, and there is a curve $C \subset S$ isomorphic to $\BP^1$ 
corresponding to the zero section of $K(D)$. Note that as divisors on $S$, we have 
$X_b \sim r C$. From the basic theory of ruled surfaces \cite[Chapter V]{Hartshorne}, 
$C.C = \deg(K(D)) = n-2$, and hence by the adjunction formula $C.K_S = -n$.
Applying the adjunction formula to $X_b$, we find
\begin{equation}
  2g(X_b) - 2 = X_b.X_b + X_b.K_S = r^2 C.C + r C.K_S = r^2(n-2) -rn,
\end{equation}
and hence $g(X_b) = \frac{1}{2}((n-2)r^2-nr+2)$.
\end{proof}

\begin{remark} For incomplete quivers, it is not necessarily the case that generic
spectral curves are smooth. In the rank 3 incomplete case studied in Chapter \ref{ch-quiver},
one may check explicitly
that $\det(\phi)$ vanishes to \emph{second order} on $D$, and hence the spectral curve
is always singular on $D$. This complicates the situation, and for this reason we
restrict our attention to the complete flag case.
\end{remark}

\begin{theorem} \label{thm-integrable-star-quiver}
Let $\FM_\alpha(r,n)$ be the complete flag star quiver of rank $r$ with
$n$ incoming paths. Then the Hitchin map $h: \FM_\alpha(r,n) \to B_0$ is a coisotropic fibration
whose generic fibers are Lagrangian subvarieties of $\FM_\alpha(r,n)$. In particular,
$\FM_\alpha(r,n)$ is an algebraic completely integrable system.
\end{theorem}
\begin{proof} By Lemma \ref{lemma-hitchin-base} and Proposition \ref{prop-integrable-coisotropic}, this is a coisotropic fibration
whose base is half the dimension of $\FM_\alpha(r,n)$, so we just have to show that
$dh$ generically has full rank. Since $\FM_\alpha$ may be identified with a GIT quotient of 
$Z \subset \prod_{i=1}^n T^\ast \FF_i$, it suffices to show that $h: Z \to B_0$ generically
has full rank.
Let $b \in B_0$ be a generic point such that the spectral curve $X_b$ is
smooth. Let $p: X_b \to \BP^1$ be the projection. By standard arguments (and in particular,
those of \cite{LogaresMartens}), any line bundle $L \in \Jac(X_b)$ determines a
degree $0$ vector bundle $E := p_\ast L$ on $\BP^1$, together with a canonical Higgs field
$\phi \in H^0(\BP^1, \End E \otimes K(D))$. For a generic $L$, the induced bundle
$E$ will be trivial, and choosing a trivialization we may write
\begin{equation}
  \phi(z)dz = \sum_{i=1}^n \frac{\phi_i dz}{z-a_i}.
\end{equation}
Since the coefficients of the characteristic polynomial vanish on $D$ (since
$b \in B_0 \subset B$), the residues $\phi_i$ are nilpotent. Hence by the Proposition \ref{prop-int-flag}, 
we may find points $m_i \in T^\ast\FF_i$ such that for each $i$, $\phi_i = \mu_i(m_i)$.
Since $\phi$ is a section of $\End E \otimes K(D)$, the sum of the residues is $0$, and hence
we have $(m_1, \cdots, m_n) \in Z$. In particular, the fiber $h^{-1}(b) \subset Z$ is non-empty.

This shows that the image of $h: Z \to B_0$ contains a Zariski open and hence dense subset
of $B_0$. By Sard's theorem, the image of $h: Z \to B_0$ must therefore contain at least one regular
value. Hence the set of $z \in Z$ such that $dh$ has full rank is non-empty, and since
this set is Zariski open, it is dense in $Z$.
\end{proof}
\addcontentsline{toc}{chapter}{Bibliography}
\bibliographystyle{alpha}
\bibliography{ut-thesis}

\end{document}